\documentclass[reqno,12pt]{amsart}
\usepackage{tikz}
\usepackage{tikz-cd}
\usetikzlibrary{matrix,calc,arrows}

\usepackage{amssymb,amsmath,mathrsfs,amsthm, amsfonts, amscd}
\usepackage{color}
\usepackage{tikz}
\usetikzlibrary{arrows}
\usetikzlibrary{patterns}
\usepackage[all]{xy}

\usepackage[T1]{fontenc}
\usepackage[bitstream-charter]{mathdesign} 
\usepackage[colorlinks,linkcolor=red,citecolor=green,anchorcolor=blue]{hyperref}
\usepackage[top=2.5cm,bottom=2.5cm,left=2.6cm,right=2.6cm]{geometry}

\input diagxy
\input xy
\xyoption{all}

\newtheorem{theorem}{Theorem}[section]
\newtheorem{proposition}[theorem]{Proposition}
\newtheorem{lemma}[theorem]{Lemma}

\newtheorem{assumption}[theorem]{Assumption}
\newtheorem{example}[theorem]{Example}
\newtheorem{defn}[theorem]{Definition}
\newtheorem{remark}[theorem]{Remark}

\theoremstyle{remark}

\numberwithin{equation}{section}
\def \fk{\mathfrak}
\def \msf{\mathsf}

\def \scr{\mathscr}
\def \inv{^{-1}}

\def \v{\vskip 0.1in}
\def \n{\noindent}
\def \cplane{\mathbb{C}}
\def \integer{\mathbb{Z}}
\def \rone{\mathbb R}

\def \0{\infty}

\def \rto{\rightarrow}

\def \rrto{\rightrightarrows}
\def \hrto{\hookrightarrow}
\def \Rto{\Rightarrow}
\def \qq{\quad}
\def \codimc{\text{codim}_\cplane}
\def \codim{\text{codim}}
\def \dimc{\dim_\cplane}

\def \G{\msf G}
\def \E{\msf E}
\def \X{\msf X}

\def \oS{\msf S}
\def \N{\msf N}

\def \Z{\msf Z}

\def \B{\msf B}
\def \U{\msf U}

\def \f{\msf f}

\def \wa{{\fk a}}
\def \wb{{\fk b}}

\def \Xa{{\underline{\X}_\wa}}

\def \K{\scr K}
\def \om{\overline{\scr M}}

\def \lN{{\overline{\msf N}}}
\def \K{\scr K}

\def \Hol{\mathrm{Hol}}

\def \oC{[\overline{\cplane^n}]^{\fk b}_{\fk a}}
\def \Aut{\mathrm{Aut}}

\allowdisplaybreaks

\begin{document}

\title[Weighted-blowup correspondence of orbifold GW invariants]
{Weighted-blowup correspondence of orbifold Gromov--Witten invariants and applications}

\author{Bohui Chen}
\address{Department of Mathematics and Yangtze Center of Mathematics, Sichuan University, 610065, Chengdu, 
China.}
\email{bohui@cs.wisc.edu.}

\author{Cheng-Yong Du}
\address{Department of Mathematics, Sichuan Normal University, 610068, Chengdu, 
China.}
\email{cyd9966@hotmail.com}

\author{Jianxun Hu}
\address{Department of Mathematics, Sun Yat-sen University, Guangzhou, 510275, 
China}
\email{stsjxhu@mail.sysu.edu.cn}


\date{\today}

\begin{abstract}
Let $\X$ be a compact symplectic orbifold groupoid with $\oS$ being a compact symplectic sub-orbifold groupoid, and $\Xa$ be the weight-$\wa$ blowup of $\X$ along $\oS$ with $\Z$ being the 
exceptional divisor.
We show that there is a weighted-blowup correspondence between some certain absolute orbifold Gromov--Witten invariants of $\X$ relative to $\oS$ and some certain relative orbifold Gromov--Witten invariants of the pair $(\Xa|\Z)$. As an application, we prove that the symplectic uniruledness of symplectic orbifold groupoids is a weighted-blowup invariant.
\end{abstract}

\maketitle

\tableofcontents

\section{Introduction}

Symplectic birational geometry, proposed by Li--Ruan \cite{LtR09}, concerns properties and structures of symplectic manifolds that are preserved by symplectic birational cobordisms. More explicitly, it studies properties and structures of symplectic manifolds that are captured by genus zero Gromov--Witten invariants. A symplectic birational cobordism consists of a finite sequence of symplectic reductions of Hamiltonian $S^1$-actions of (different) symplectic manifolds at (different) regular values, see Guillemin--Sternberg \cite{GS89}, Hu--Li--Ruan \cite{HLR08}. It was proved in \cite{GS89} that a symplectic birational cobordism can be decomposed into a finite sequence of symplectic blowups/blow-downs and $\integer$-linear deformations of symplectic forms. Since Gromov--Witten invariants are deformation invariant, so the essential part is to discover the influence of symplectic blowups/blow-downs on the properties and structures captured by genus zero Gromov--Witten invariants.
Therefore a beautiful and useful blowup formula for Gromov--Witten theory is a core to symplectic birational geometry. Gathmann \cite{Ga01} proved some blowup formulae of genus zero Gromov--Witten invariants for blowup of convex projective varieties along points. Hu \cite{Hu00,Hu01} proved similar blowup formulae of genus zero and one Gromov--Witten invariants for blowup of symplectic manifolds along points, curves, and surfaces. When the ambient manifold has dimension less than 8, Hu's formulae hold for all genera.

A compact symplectic manifold $(X,\omega)$ is {\em symplectic uniruled} if it has a non-vanishing genus zero Gromov--Witten invariant
\begin{align}\label{E 1}
\langle [pt],\alpha_2,\ldots,\alpha_k\rangle^X_{0,\beta}\neq 0
\end{align}
with $0\neq \beta\in H_2(X,\integer)$ and $k\geq 1$. Koll\'ar \cite{Ko98} and Ruan \cite{Ru99} proved that for a smooth projective uniruled variety $X$, there is a non-zero Gromov--Witten invariant as \eqref{E 1}
, i.e. $X$ is symplectic uniruled. It is well-known that projective uniruledness is a birational invariant. Therefore a natural question is that, if the symplectic uniruledness of symplectic manifolds is a symplectic birational cobordism invariant. By using those blow-up formulae proved by Gathmann and Hu, one concludes that
  the blowup manifolds are symplectic uniruled provided that ambient manifolds are symplectic uniruled.

The big break-through for general symplectic blowup was made by Hu--Li--Ruan in \cite{HLR08}. Let $X$ be a compact symplectic manifold, $S\subseteq X$ be a compact symplectic submanifold, $\tilde X$ be the symplectic blowup of $X$ along $S$, and $E$ be the exceptional divisor. Hu--Li--Ruan \cite{HLR08} proved a general algorithm, known as a blowup correspondence of Gromov--Witten invariants, to compare certain relative Gromov--Witten invariants of $(\tilde X|E)$ and certain absolute Gromov--Witten invariants of $X$ relative to $S$. When $\codimc S=1$, i.e. $S$ is a symplectic divisor of $X$, such a blowup correspondence was obtained by Maulik--Pandharipande  \cite{MauP06}. With this blowup correspondence Hu--Li--Ruan showed that symplectic uniruledness of symplectic manifolds is a symplectic birational cobordism invariant.

Similarly, as a symplectic analog of rational connectedness in birational algebraic geometry one can also consider symplectic rational connectedness of symplectic manifolds. A compact symplectic manifold $(X,\omega)$ is {\em symplectic rational connected} if it has a non-vanishing genus zero Gromov--Witten invariant
\begin{align}\label{E 2}
\langle [pt],[pt],\alpha_3,\ldots,\alpha_k\rangle^X_{0,\beta,k}\neq 0
\end{align}
with $0\neq \beta\in H_2(X,\integer)$ and $k\geq 2$. Recently there are a lot of works on symplectic rational connectedness. For example, Li--Ruan \cite{LtR13}, Hu--Ruan \cite{HuR13}, Voisin \cite{V08}, and Tian \cite{Tian12,Tian15} studied the symplectic rational connectedness of symplectic manifolds and projective varieties.

Orbifolds are natural generalization of manifolds. Roughly speaking orbifolds are manifolds with finite quotient singularities. Symplectic orbifolds appear naturally in symplectic reductions. It would be more convenient and challenging to involve orbifolds in symplectic birational geometry. This viewpoint was also addressed in \cite{LtR09}. We can adapt the definitions of symplectic uniruledness and symplectic rational connectedness of symplectic manifolds to symplectic orbifolds naturally. Then one may ask if symplectic uniruledness and symplectic rational connectedness of symplectic orbifolds are still birational invariants. There already have some works on symplectic uniruledness of orbifolds. He--Hu \cite{HeHu15} and Du \cite{Du17a} proved several formulae for orbifold Gromov--Witten invariants of weighted blowups along smooth points, which implies that when the ambient symplectic orbifold $\X$ is symplectic uniruled then the weighted blowup $\Xa$ of $\X$ along smooth points with weight $\wa$ is also symplectic uniruled. In this paper we study the change of Gromov--Witten invariants under weighted blowups of symplectic orbifolds along general symplectic sub-orbifolds. We show that symplectic uniruledness of symplectic orbifolds is an invariant under weighted blowups. This is a consequence of a weighted-blowup correspondence 
of orbifold Gromov--Witten invariants
. We next describe the results.

\subsection{Weighted-blowup correspondence}
We first show a correspondence result for orbifold Gromov--Witten invariants under weighted blowups. Let $\Xa$ be the weight-$\wa$ blowup of a symplectic orbifold groupoid $(\X,\omega)$ along a symplectic sub-orbifold groupoid $\oS$. Denote the exceptional divisor by $\Z$.
let $\kappa:\Xa\rto\X$ be the natural projection associated to the blowup.
By taking a self-dual basis $\sigma_\star$ of $H^*_{CR}(\oS)$ we get a base $\Sigma_\star$ and a dual basis $\Sigma^\star$ of $H_{CR}^*(\Z)$ (see \S\ref{sec rel-inv-E-B}). We consider relative orbifold Gromov--Witten invariants of $(\Xa|\Z)$ with relative insertions coming from $\Sigma_\star$ and absolute insertions coming from $\K:=\msf I\kappa^*(H^*_{CR}(\X))$. We denote the set of admissible relative data (cf. \S \ref{sec 4.2}) by $\scr R_{\Sigma_\star,\K}(\Xa|\Z)$. We assign a partial order to $\scr R_{\Sigma_\star,\K}(\Xa|\Z)$ by using degeneration of moduli space associated to the degeneration of $\Xa$ along $\Z$ with trivial weight
(cf. \S \ref{S partial-order} and Theorem \ref{thm partial-ord}). We use $\scr R_{\Sigma_\star,\K}(\Xa|\Z)$ to generate a linear space $\rone_{\Sigma_\star,\K}$. By using the Gromov--Witten invariants of relative data we get a vector $v_{\Sigma_\star,\K}\in \rone_{\Sigma_\star,\K}$. We also consider absolute $\sigma_\star$-data of $(\X,\oS)$. In an admissible $\sigma_\star$-data ${\bf A}^\bullet(I;I_\oS)$ (cf. \S \ref{sec 4.1}) we require that $I$ has no descendent insertions, i.e. $\psi$-classes, and $I_\oS$ are all supported on
$\oS$, i.e. are of the form $(\sigma_\star\cup [{\sf IS}])\psi^c$, where $[{\sf IS}]$ is the Thom class of the inertia space (cf. Definition \ref{def twisted-sector}) $\sf IS$ of $\oS$ in the inertia space $\sf IX$ of $\X$. We denote the set of admissible absolute $\sigma_\star$-datum of $(\X,\oS)$ by $\scr A_{\sigma_\star}(\X,\oS)$. For each admissible relative data ${\bf R}^\bullet(I|J)\in \scr R_{\Sigma_\star,\K}(\Xa|\Z)$ we associate it an admissible absolute data ${\bf A}^{\bullet}(I_\kappa;I_\msf S)\in\scr A_{\sigma_\star}(\X,\oS)$. This gives us a map $\Psi: \scr R_{\Sigma_\star}(\Xa|\Z)\rto \scr A_{\sigma_\star}(\X,\oS)$, which is a bijection when $\codimc \oS>1$, and an injection when $\codimc\oS=1$, see Theorem \ref{thm crepdns-on-datum}. By using the absolute invariants of absolute data in the image of $\Psi$ we get another vector $v_{\sigma_\star}\in\rone_{\Sigma_\star,\K}$. The map $\Psi$ gives us a linear map $L:\rone_{\Sigma_\star,\K}\rto \rone_{\Sigma_\star,\K}$. We have

\begin{theorem} {\em (Theorem \ref{thm lower-traingl})} \label{thm main}
The linear map $L:\rone_{\Sigma_\star,\K}\rto \rone_{\Sigma_\star,\K}$ has the following properties:
\begin{itemize}
\item $L(v_{\Sigma_\star,\K})=v_{\sigma_\star}$,
\item the matrix of $L$ with respect to the basis $\scr R_{\Sigma_\star,\K}(\Xa|\Z)$ is lower triangle and has non-vanishing diagonals.
\item $L$ is invertible.
\end{itemize}
\end{theorem}

We can restrict the correspondence $\Psi$ to genus zero relative/absolute datum which have a point class insertion as the first absolute insertion. Then we have a similar result (see Theorem \ref{thm correspds-pt} and Remark \ref{rmk other-restriction-of-L}).


\subsection{Symplectic uniruledness of orbifolds}\label{subsec 1.2}

In the manifold case, the motivation of symplectic uniruledness comes from the uniruledness in algebraic geometry which is about the existence of embedded rational curves passing through a given point, which leads to the equation \eqref{E 1}. Now we consider the orbifold case. Let $\mathsf X$ be a compact  orbifold  (cf. Remark \ref{rmk kernel-G} for compactness of orbifold). $[pt]$ in \eqref{E 1} should be replaced by certain class in $H^\ast_{CR}(\mathsf X)=H^\ast(\mathsf{IX})$.
 However, when the orbifold $\mathsf X$ is ineffective, the situation is slightly complicated, which is explained below.

Let $\msf{ker\,X}=\msf {IX}^{\text{top}}$ be the union of top dimensional twisted sectors of $\X$ (cf. Remark \ref{rmk kernel-G} and Remark \ref{rmk ker-G=IG-top}). The $[pt]$ in \eqref{E 1} should be replaced by generators of  $H^{\dim\X}(\msf{ker\,X})$. Note that when $\sf X$ is effective, $\msf{ker\,X}=\X(1)\cong\X$.  When $\X$ is ineffective, $\msf{ker\,X}$ has more than one  component.
This leads to the following definition. 
 We say that a compact symplectic orbifold groupoid $\X$ is {\bf symplectic uniruled} if there is a nonzero orbifold Gromov--Witten invariant of $\X$ of the form
\begin{align}\label{E 3}
\langle [pt]_{\sf ker}, \alpha_2,\ldots,\alpha_k\rangle^\X_{0,A}\neq 0
\end{align}
with $k\geq 1$ and $0\neq A\in H_2(|\X|;\integer)$ (see Definition \ref{def uniruled-orbifold} and Remark \ref{rmk 7.2}). Here $[pt]_{\sf ker}$
depends on the location component of the point, and it runs through the  generators of $H^{\dim \X}(\msf{ker\,X})$.

One of the main application in this paper is the following

\begin{theorem}[Theorem \ref{thm inv-sym-unirule}]
Let $(\X,\omega)$ be a compact symplectic orbifold groupoid with $\Xa$ being its weight-$\wa$ blowup along a compact symplectic sub-orbifold $\oS$. Then $\X$ is symplectic uniruled if and only if $\Xa$ is symplectic uniruled.

\end{theorem}

Furthermore, it is natural to ask whether there exists an embedded rational curve with orbifold structure passing through an orbifold point with given monodromy. That means, for symplectic orbifolds,  we may consider a more general notion of symplectic uniruledness  which involves point classes of general twisted sectors of $\X$ not only $\msf{ker\,X}$. We will discuss the invariance of this version of symplectic uniruledness in \S \ref{subsec general-symplec-uniruled}.

\subsection*{Acknowledgements}
The authors thank the anonymous referee for careful reading and helpful suggestions, especially for the suggestion to consider the more general version of symplectic uniruledness discussed in \S \ref{subsec general-symplec-uniruled}.
 The first author is partially supported by NSFC (No. 11431001 $\&$ No. 11890663). The second author is partially supported by NSFC (No. 11501393) and by Sichuan Science and Technology Program (No. 2019YJ0509). The third author is partially supported by NSFC (No. 11831017 $\&$ No. 11890662 $\&$ No. 11521101 $\&$ No. 11771460).

\section{Orbifolds as proper \'etale Lie groupoids}

In this paper, we treat an orbifold as a Morita equivalence class of proper \'etale Lie groupoids, which are called orbifold groupoids. There are some nice references on orbifold groupoids, for example \cite{ALR07} and \cite{MoP97}. See also \cite{CH06}. It is known that an effective orbifold is equivalent to a Morita equivalence class of effective proper \'etale Lie groupoids. But when effectiveness
fails,  orbifold groupoids should be used. Even for
effective orbifolds, the structure of Lie groupoids provides a powerful way to control many tedious issues of orbifolds.
In this section, we introduce the  material of Lie groupoids needed in this paper.

\subsection{Proper \'etale Lie groupoids}

We review some basic definitions.

\begin{defn}[Lie groupoids]
\label{Lie-gpoid:def}
 A {\bf Lie groupoid}
 $\sf G$  consists of two smooth manifolds $G^0$ and $G^1$, together with five smooth maps $(s, t, m, u, i)$ satisfying the following properties.
  \begin{enumerate}
\item  The source map  and the target map $s, t: G^1 \rto  G^0$ are both submersions.
\item The composition map
\[
m:  G^{[2]}: =\{(g_1, g_2) \in  G^1 \times  G^1: t(g_1) = s(g_2)\} \rto G^1
\]
written as $m(g_1, g_2) = g_1\circ  g_2$ for composable elements $g_1$ and $g_2$,
satisfies the obvious associative property.
\item The unit map $u: G^0 \rto G^1$ is a two-sided unit for the composition.
\item The inverse map $i: G^1 \rto G^1$, $i(g) = g^{-1}$,  is a two-sided inverse for the composition.
\end{enumerate}
 For simplicity, we may denote $\sf G$  by  ${\sf G} = (G^1 \rightrightarrows G^0)$, where double arrows represent the source and target maps.
 $G^0$ is called the space of objects or units, and $G^1$  is called the space of arrows. An element $\alpha\in G^1$ is called an arrow from its source
 $s(\alpha)$ to its target $t(\alpha)$.

 The Lie groupoid $\sf G $   is called {\bf  proper } if $(s, t): G^1 \rto G^0 \times G^0$ is proper, and is called {\bf   \'etale} if $s$ and $t$ are local diffeomorphisms.
A proper \'etale Lie groupoid is also called an {\bf orbifold groupoid}.
\end{defn}
\begin{remark}
$G^1$ defines an equivalence relation on $G^0$: for $x,y\in G^0$ we say $x\sim y$
if and only if there exists an arrow from $x$ to $y$. Denote the quotient space $G^0/G^1$   by $|\sf G|$, and denote the projection map by $\pi_{\sf G}$.
$|\sf G|$ is called the {\bf coarse space} of $\sf G$. Conversrly, $\sf G$ is called a Lie groupoid structure of the topological space $|\sf G|$. A Lie groupoid is {\bf connected}, if its coarse space is connected.
\end{remark}
\begin{remark}
A groupoid is a category in the following sense:
  $G^0$ is the set of objects; an arrow from object $x$ to $y$
  is a morphism from
  $x$ to $y$;
  $m$ defines  the composition of morphisms.
  \end{remark}
\begin{defn}\label{def equivalence}
Let $\sf G$ and $\sf H$ be two Lie groupoids. By a {\bf strict morphism}
from $\sf G$ to $\sf H$ we mean a functor $\sf f: \sf G\rto \sf H$
which consists of a pair of smooth maps
$f^i: G^i\rto H^i, i=0,1$ that are compatible with structure maps of
groupoids. Every strict morphism $\sf f:G\rto H$ induces a continuous map $|\sf f|:|G|\rto|H|$ on coarse spaces.

A strict morphism $\sf f: G\rto H$ is called an {\bf equivalence} if
\begin{enumerate}
\item[(1)]  the map
       $t\circ \mathrm{proj}_2:
       G^0\times _{f^0,H^0,s} H^1
       \xrightarrow{\mathrm{proj}_2} H^1\xrightarrow{t} H^0$
       is a surjective submersion;
\item[(2)]  the square
\[
\xymatrix{
 G^1\ar[rr]^{ f^1} \ar[d]_{s\times t} &&
 H^1\ar[d]^{s\times t}\\
 G^0\times G^0\ar[rr]^{ f^0\times f^0} &&
 H^0\times H^0.
}\]
is a fiber product of smooth manifolds.
\end{enumerate}


By a (weak) {\bf morphism} from $\sf G$ to $\sf H$ we mean a triple
$(\sf g, \sf M,\sf f)$ with the diagram
$$
\sf G\xleftarrow{\sf g}\sf M\xrightarrow{\sf f}\sf H,
$$
where $\sf g:\sf M\rto \sf G$ is an equivalence and $\sf f$ is a strict morphism.
\end{defn}

We give some typical examples of Lie groupoids that will be used in this paper.
 \begin{example}[Quotient groupoids]\label{example_1}
 Let $M$ be a smooth manifold and $G$ be a Lie group. Suppose that
 $G$ acts on $M$ smoothly with the action map
 $
 \phi: M\times G\rto M.
 $
 The quotient space $M/G$ admits a Lie groupoid structure
 $$
G \ltimes M:= (M\times G\rightrightarrows M),
 $$
 where $s(m,g)=m$ and $t=\phi$.
 This is called the quotient groupoid for $M/G$. 

 The following are special cases.
 \begin{enumerate}
 \item
By identifying a manifold $M$ as a quotient groupoid
 $M/\{e\}$, hence $\{e\}\ltimes M\cong (M\rightrightarrows M)$  gives a canonical groupoid structure on $M$;
 \item
 for any  group $G$ acting trivially on a point $\{pt\}$, we denote the quotient groupoid by $BG$;
 \item for any positive integer $r$, we denote $B\mathbb Z_r$ by $\sf B_r$; in this sense, $\msf B_1=\{pt\}$. In this paper, we identify
 $$
 \mathbb Z_r=\langle \zeta_r\rangle,
 $$
     where $\zeta_r=e^{2\pi i\frac{1}{r}}$.
 \end{enumerate}
 \end{example}

 \begin{example}\label{example_2}
 Let $\msf G=(G^1\rightrightarrows G^0)$ be a Lie groupoid. For any
 open subset $U^0\subseteq G^0$, define
 $$
 U^1= s^{-1}(U^0)\cap t^{-1}(U^0).
 $$
 Then we call $(U^1\rightrightarrows U^0)$ is a groupoid  induced on
 $U^0$; we denote it by $\msf G(U^0)$ or $\msf G|_{U^0}$.

 For any
 open subset $\bf U\subseteq |\sf G|$, let $U^0=\pi_{\msf G}^{-1}(\bf U)$.
 Then we have an induced groupoid
 $\msf G(U^0)=(U^1\rightrightarrows U^0)$
  for $\bf U$.

  In fact, for any open subset $U^0\subseteq G^0$,
  $\msf G(U^0)$ is equivalent to $\msf G(\pi_{\msf G}^{-1}(\pi_{\msf G}(U^0))$, that is the natural inclusion
  $\msf G(U^0)\hrto \msf G(\pi_{\msf G}^{-1}(\pi_{\msf G}(U^0))$ is an equivalence of Lie groupoids.

 When $\bf U$ is a connected component of $|\G|$ we call $\msf G(\pi_{\msf G}^{-1}(\bf U))$ a connected component of $\G$.
 \end{example}

\begin{remark}
Let $\msf G$ be a proper \'etale Lie groupoid. We have following properties.
\begin{enumerate}
\item For every $x\in G^0$,
$$
G_x:=(s, t)^{-1}(x, x) = s^{-1}(x)\cap t^{-1}(x)
$$
is a finite group. We call  $G_x$ to be the {\bf isotropy group}  at $x$.
\item When $\G$ is effective (cf. \cite[Definition 1.38]{ALR07}), there exists a small
neighborhood $U^0$ of $x$ such that the induced groupoid structure
$$\msf G(U^0)=G_x\ltimes U^0,$$
 which is the groupoid structure for ${\bf U}=\pi_{\msf G}(U^0)$. It is also called an orbifold chart over ${\bf U}$ or a local orbifold chart around $x$.
\end{enumerate}
This leads to the following definition which is given in \cite{ALR07}:
a proper \'etale Lie groupoid $\msf G$ is an {\bf orbifold
structure} of the topological space $X=|\msf G|$. A space $X$ is called an {\bf orbifold} if it admits a proper et\'ale Lie groupoid as its orbifold structure.
\end{remark}

\begin{remark}\label{rmk kernel-G}
Let $\msf G=(G^1\rrto G^0)$ be a connected proper \'etale Lie groupoid. Let $ker \G$ be the space of constant arrows (or ineffective arrows) (cf. \cite[Definition 2.27]{ALR07}). It is showed in \cite[Lemma 2. 32]{ALR07} that $ker \G$ consists of a union of connected components in $G^1$. $\G$ acts on $ker \G$ (cf. Definition \ref{def groupoid-action}) naturally by conjugations. We denote the resulting action groupoid $\G\ltimes ker \G$ by $\msf{ker\,G}$. It is also proper \'etale. However, it may not be connected.
When $\G$ is effective, $\msf{ker\,G}\cong\G$ is connected. Otherwise $\sf ker\,\G$ has several connected component and we call $\G$ an ineffective orbifold groupoid.

In this paper we say that a proper \'etale Lie groupoid $\G$ is {\bf compact} if the coarse space $|\msf{ker\,G}|$ of $\msf{ker\,G}$ is compact. In this circumstance, local orbifold charts also exist for ineffective orbifold groupoids.
\end{remark}

In the following we only consider connected compact proper \'etale Lie groupoids.
\subsection{Space of morphisms and twisted sectors}

Let $\msf G=(G^1\rightrightarrows G^0)$ and $\msf H=(H^1\rightrightarrows H^0)$ be two Lie groupoids.
Let $\mathrm{SMor}(\msf H,\msf G)$ be the space of strict morphisms
from $\msf H$ to $\msf G$. For any two morphisms
$\msf f, \msf g\in \mathrm{SMor}(\msf H,\msf G)$,
 we say that they are equivalent, denote as $\msf f\sim\msf g$, if
there is a natural transformation $\phi:\msf f\Rto \msf g$, i.e, a smooth map
$
\phi: H^0\rto G^1$,
such that
$$
s\circ\phi(x)=f^0(x), \;\;\;
f^1(\alpha)\circ \phi(t(\alpha))=\phi(s(\alpha))\circ g^1(\alpha)$$
for any $x\in H^0$ and $\alpha\in H^1$. Let
$\bf{SMor}(\msf H,\msf G)$ be the quotient space $\mathrm{SMor}(\msf H,\msf G)/\sim$.

$\bf{SMor}(\msf H,\msf G)$ admits a natural groupoid structure, denoted by
$$
\msf{SMor}(\msf H,\msf G)=
(\mathrm{SMor}^1(\msf H,\msf G)\rrto \mathrm{SMor}^0(\msf H,\msf G)),
$$
where $\mathrm{SMor}^0(\msf H,\msf G)=\mathrm{SMor}(\msf H,\msf G)$ and $\mathrm{SMor}^1(\msf H,\msf G)$ consists of
all
 natural transformations between any two morphisms.
  A natural transformation from
  $\msf f$ to  $\msf g$  is called an arrow
 from
$\msf f$ to $\msf g$.
\begin{remark}We make some remarks.
\begin{enumerate}
\item $\mathrm{SMor}^0(\msf H,\msf G)$ can be only a topological space even with proper Sobolev norms, hence
    $\msf{SMor}(\sf H,G)$ is only a topological groupoid when
    proper topology is endowed.
\item If
$\msf G$ and $\msf H$ are proper and \'etale, so is $\msf{SMor}(\msf H,\msf G)$;
\item In order to consider all possible morphisms from $\msf H$ to $\msf G$, one must consider weak morphisms. This leads to a groupoid $\msf {Mor}(\msf H,\msf G)$ (cf. \cite{CDW17b}). If $\msf H$ and $\msf G$ are two proper \'etale Lie groupoids,
    $\msf {Mor}(\msf H,\msf G)$ is equivalent to a proper et\'ale
    Banach Lie groupoid when proper Sobolev norms are adapted.
    \end{enumerate}
\end{remark}

Let $\sf G$ be an orbifold groupoid. Take  $\msf H=\msf B_r$ (cf. Example \ref{example_1}). Let
$$
\mathrm{IMor}^0(\msf B_r, \msf G)
=\{\msf f=(f^0,f^1)| f^1 \mbox{ is  injective}\}
\subseteq \mathrm{SMor}^0(\msf B_r,\msf G).
$$
Hence every $\msf f=(f^0,f^1)\in \mathrm{IMor}^0(\msf B_r,\msf G)$ corresponds to a monomorphism $f^1:\integer_r\hrto G_{f^0(pt)}$. Then $\msf{SMor}(\msf B_r,\msf G)$ induces a Lie  groupoid
\[
\msf{IMor}_r(\msf G):=(\mathrm{IMor}^1(\msf B_r,\msf G)\rrto \mathrm{IMor}^0(\msf B_r,\msf G))=\msf{SMor}(\msf B_r,\msf G)|_{\mathrm{IMor}^0(\msf B_r,\msf G)},
\]
which is proper and \'etale. So $$\mathrm{IMor}^1(\msf B_r,\msf G)=\{\sigma:\msf f_1\Rto \msf f_2\mid \msf f_1,\msf f_2\in \mathrm{IMor}^0(\msf B_r,\msf G)\},$$ and every natural transformation $\sigma:\msf f_1=(f^0_1,f^1_1)\Rto \msf f_2=(f^0_2,f^1_2)$ corresponds to an arrow $\sigma(pt):f^0_1(pt)\rto f^0_2(pt)$ that satisfies
\[
\sigma(pt)\inv \circ  f^1_1(\zeta_r) \circ \sigma(pt)=f^1_2(\zeta_r).
\]

\begin{defn}\label{def twisted-sector}
The proper \'etale Lie groupoid
$$
\msf{IG}=\bigsqcup_{r\in \mathbb Z^+} \msf{IMor}_r(\msf G)
$$
is called the {\bf inertia space} of $\msf G$.
\end{defn}

There is a natural morphism $e:\sf IG\rto G$ called the {\bf evaluation morphism}, which on objects maps each $\f=(f^0,f^1)\in \mathrm{IMor}^0(\msf B_r,\msf G)$ to $f^0(pt)$ and on arrows maps each $\sigma \in \mathrm{IMor}^1(\msf B_r,\msf G)$ to $\sigma(pt)$.

According to the decomposition of the coarse space $|\msf {IG}|$ into connected components, $\sf IG$ decompose into connected components. We call each connected component a {\bf twisted sector} of $\G$.
We say that a twisted sector is of top dimension, if the restriction of $e:\sf IG\rto G$ is a local diffeomorphism on object spaces. Denote the union of top dimensional twisted sectors by $\msf {IG}^{\text{top}}$.

\begin{remark}\label{rmk ker-G=IG-top}
We have $\msf {IG}^{\emph{top}}=\msf{ker\,G}$. Therefore $\msf{ker\,G}$ is a (possibly union of) top dimensional twisted sector of $\G$.
\end{remark}


\begin{example}[Inertia space of quotient groupoids]
Suppose that $\msf G=M\rtimes \Gamma$ is a quotient groupoid. Then
\[
\msf{IMor}_r(\msf G)= \bigsqcup_{(\gamma),\,\text{\em ord}(\gamma)=r} C_\Gamma(\gamma) \ltimes M^\gamma,
\]
where $(\gamma)$ is a conjugacy class of $\gamma$, $M^\gamma$ is the space fixed by $\gamma$, $C_\Gamma(\gamma)$ is the centralizer of $\gamma$.
\end{example}

\begin{remark}\label{rmk notation-of-twisted-sector}
Now we give a local description of twisted sectors of a compact orbifold groupoid $\G$. Let $x\in G^0$
and $G_x \ltimes U^0$ be a local orbifold chart. For any
$g\in G_x$,
 there is a  twisted sector for the conjugate class $(g)$  which is locally given by $ C_{G_x}(g)\ltimes(U^0)^g$.
In \cite{CR04}, $(g)$ is treated as the index for the twisted sector. For instance we have $\G(1)$, where $1$ is the unit element in isotropy groups. Moreover, $e:\msf{IG}\rto\G$ restricts to an isomorphism $\G(1)\cong \G$.

Let $\scr T^{\msf G}$ be the set of indices of twisted sectors. A twisted sector is denoted by $\msf G(\delta)$ for $\delta\in \scr T^{\msf G}$. The restriction of $e:\sf IG\rto G$ on a twisted sector $\G(\delta)$ is denoted by $e_\delta$.
For convenience, when we use local models, we may abuse the notations to treat $(g)$ as an element of $\scr T^{\msf G}$. We set $\scr T^\G_{\sf ker}$ to be the index set of top dimensional twisted sectors. Then
\[
\msf{ker\,G}=\msf{IG}^{\emph{top}}=\bigsqcup_{\delta\in\scr T^\G_{\sf ker}} \G(\delta).
\]
When $\G$ is effective, $\msf{ker\,G}=\G(1)\cong\G$ via $e$.

%

\end{remark}


The Chen--Ruan cohomology of an orbifold groupoid $\msf G$ is defined to be
$$
H^\ast_{CR}(\msf G) = H^\ast(\msf{IG})
$$
with proper degree shiftings. In \cite{CR04}, Chen and Ruan introduced a natural Poincar\'e pairing for $H^\ast_{CR}$. We recall the definition.
For a morphism $\msf f: \msf B_r\rto\msf G$ given by
$$
f^0:pt\mapsto x,\;\;\;
f^1:\zeta_r\mapsto \gamma\in G_x
$$
define a ``dual'' morphism $\bar{\msf f}:\msf B_r\rto \msf G$
given by
$$
\bar f^0:pt\mapsto x,\;\;\;\bar f^1:\zeta_r\mapsto \gamma^{-1}\in G_x.
$$
This induces a morphism $\bar{\cdot}: \msf{IG}\rto \msf{IG}$
and also
\begin{align}\label{E def-bar-on-twistedsector}
\bar{\cdot}: \scr T^{\msf G}\rto \scr T^{\msf G}.
\end{align}
The Poincar\'e pairing on $H^\ast_{CR}(\msf G)$ is defined to be
\begin{align}
\langle\cdot,\cdot\rangle:
H^\ast(\msf G(\delta))\times H^\ast(\msf G(\bar\delta))\rto \mathbb R,\;\;\;
\langle\alpha,\beta\rangle=\int_{\msf G(\delta)}\alpha\wedge {\bar{\cdot}}^\ast(\beta)
\end{align}
for all $\delta\in \scr T^{\msf G}$.

\subsection{Groupoid actions and fiber bundles}

\begin{defn}[Groupoid actions]\label{def groupoid-action}
A {\bf right action} of $\msf G$ on a space $X^0$ consists of two smooth maps $\pi^0: X^0\rto G^0$ and
\begin{eqnarray*}
\phi: X^1:=X^0\times_{\pi^0,G^0,s} G^1=\{(x,g)|\pi^0(x)=s(g)\} &\rto& X^0, \\
(x,g)&\mapsto& xg
\end{eqnarray*}
such that
$$
(xg)h=x(gh),\;\;\; x u(\pi^0(x))=x,\;\;\; \pi^0(xg)=t(g).
$$
We call $\pi^0$ the {\bf base map}, and $\phi$ the {\bf action map}.

We associate a groupoid  $(X^1\rightrightarrows X^0)$ to the action by setting the groupoid structure maps to be
$$
s(x,g)=x,\; t(x,g)=xg,$$
$$ m((x,g),(xg,h))=(x, gh),\;
u(x)=(x,u(\pi^0(x))),\; i(x,g)=(xg, g^{-1}).
$$
We denote this groupoid by $\G\ltimes X^0$, and call it the {\bf action groupoid} associated to the $\G$-action on $X^0$.

Let $\pi^1: X^1\rto G^1$ be the projection. Then $\pi=(\pi^0,\pi^1):
\G\ltimes X^0\rto \msf G$ is a strict morphism. We call it the projection of the groupoid action.
\end{defn}

\begin{defn}[Fiber bundles]
By a {\bf fiber bundle} over $\msf G$ with fiber $F$ we mean a strict morphism $\pi:\msf E\rto \msf G$, such that
\begin{enumerate}
\item $\pi^0: E^0\rto G^0$
is a fibration with fiber $F$;
\item $\msf E=\G\ltimes E^0$ for some $\msf G$-action on $E^0$, and
$\pi$ is the projection of the groupoid action.
\end{enumerate}
We call $\msf E$ to be a {\bf vector
bundle} over $\msf G$ if $\pi^0: E^0\rto G^0$ is
a vector bundle and  the action map
$\phi(x,g)$ is linear in $x$.
\end{defn}
As the bundle theory over manifolds, we introduce the concepts of  principal bundle and associated bundle.
\begin{defn}[Principal bundles]
Let $K$ be a Lie group and $\pi: \msf P\rto \msf B$ be a $K$-fiber bundle. We say that $\msf P$ is a $K$-{\bf principal bundle} if $\pi^0:P^0\rto B^0$
is  a (left) $K$-principal bundle and the $K$-action on $P^0$ commutes with the $\msf B$-action.
\end{defn}
For a principal $K$-bundle $\pi:\msf P\rto\msf B$, $\pi^i:P^i\rto B^i,i=0,1$
are left $K$-principal bundles. Then
$$
K\backslash\msf P:=(K\backslash P^1\rightrightarrows K\backslash P^0)\cong \msf B.
$$
For  any right $K$-space $F$, we may define an {\bf associated bundle}
$$
F\times_K\msf P=(F\times_KP^1\rightrightarrows F\times_K P^0).
$$
 This can be generalized by allowing that $F$ is a groupoid. Suppose that $\msf F=(F^1\rightrightarrows F^0)$ admits a group $K$-action, then
 $$
 \msf F\times_K \msf P:=(F^1\times_K P^1\rightrightarrows F^0\times_KP^0).
 $$
\begin{remark}
The right and left actions defined above can be switched to be left and right ones.
\end{remark}

\subsection{Symplectic orbifold groupoids}

Let $\msf X=(X^1\rightrightarrows X^0)$ be an orbifold groupoid.
By a symplectic form $\omega$ on $\msf X$ we mean a pair of symplectic
forms $(\omega^0,\omega^1)$ with $\omega^i\in\Omega^2(X^i),i=0,1$, such that
$s^\ast \omega^0=t^\ast\omega^0=\omega^1$. We call $(\X,\omega)$ a {\bf symplectic orbifold groupoid} if $\omega$ is a symplectic form on $\X$. A strict morphism $\msf f=(f^1,f^0):(\X,\omega)\rto(\msf Y,\eta)$ is called a symplectomorphism if both $f^0$ and $f^1$ are symplectomorphisms, and then we call that $(\X,\omega)$ is symplectic isomorphic to $(\msf Y,\eta)$.

We say that an orbifold groupoid $\msf S=(S^1\rightrightarrows S^0)$ is a {\bf sub-orbifold groupoid} of $\msf X$ if $S^0$ is a smooth submanifold of $X^0$ and $S^1=s^{-1}(S^0)\cap t^{-1}(S^0)$, hence $S^1$ is also a smooth submanifold of $X^1$. We call $\msf S$ to be a {\bf symplectic sub-orbifold groupoid} of $(\msf X,\omega)$ if $S^i,i=0,1$ are symplectic submanifolds of $(X^i,\omega^i)$. In this case we call $(\msf X,\msf S)$ a symplectic orbifold groupoid pair. We also have the following {\bf symplectic neighborhood theorem}.

\begin{theorem}[\cite{DCW18}]\label{thm sym-neigh}
Let $(\X,\oS)$ be a symplectic orbifold groupoid pair, and suppose that $\oS$ is compact. Then there is a symplectic sub-orbifold groupoid $\oS_\ast=(S^1_\ast\rrto S_\ast^0) \subseteq\oS$ of $\X$ such that
\begin{itemize}
\item
      the natural inclusion $\oS_\ast\hrto \oS$ is an equivalence in the sense of Definition \ref{def equivalence};
\item
      there is a tubular neighborhood $\msf U=(U^1\rrto U^0)$ of $\oS_\ast$ in $\X$, i.e. $\msf U$ is an (open) sub-orbifold groupoid of $\X$ and $U^i$ are tubular neighborhood of $S^i_\ast$ in $X^i$;
\item
      $(\msf U,\omega|_{\msf U})$ is symplectic isomorphic to a disc bundle of the normal bundle $\N$ (with the induced symplectic form) of $\oS_*$ in $\X$.
\end{itemize}
Then we can take compatible almost complex structures on both $\N$ and $\msf U$ such that the symplectomorphism $\msf U\rto\N$ also preserves the almost complex structures.
\end{theorem}

We see that $U^i$ are symplectic tubular neighborhood of $S_\ast^i$ in $(X^i,\omega_i)$, which are symplectomorphic to disc bundle of the normal bundle $N^i$ of $S_\ast^i$ in $X^i$ for $i=1,2$. 

\section{Weighted blowups for symplectic orbifolds}

Let $(\msf X,\msf S)$ be a symplectic orbifold groupoid pair. We will discuss the weighted blowups of $\msf X$ along $\msf S$. By Theorem \ref{thm sym-neigh}, it is sufficient to discuss the weighted blowups of vector bundles along their 0-sections.

\subsection{Weighted projective spaces}\label{sec wght-prj-sps}
Suppose $S^1$ acts on $\cplane^n$ with weight
$
\fk a=(\alpha_1,\ldots, \alpha_n)$, $\alpha_i\in\integer_{>0}.
$
Namely,
$$
t\cdot(z_1,\ldots,z_n)=(t^{\alpha_1}z_1,\ldots, t^{\alpha_n}z_n).
$$
For simplicity, we denote the right hand side by $t^{\fk a}\cdot z$ and denote this action by $S^1_{\fk a}$. The quotient
 groupoid $S^1_{\fk a}\ltimes S^{2n-1}$ is called a {\bf weight-$\wa$ projective space}, and is denoted by $\msf P_\fk a$.
We say that $\msf P_\fk a$ is the {\bf weight-$\wa$ projectivization} of $\cplane^n$ with respect to the $S^1_\fk a$ action.

The following  concepts are related:
 \begin{enumerate}
 \item the {\bf weight-$\wa$ projectification} of $\cplane^n$ with respect to the $S^1_\fk a$ action is defined to be
$$
[\overline{\cplane^n}]_\fk a:=\msf P_{(\fk a,1)}= S^1_{({\fk a},1)}\ltimes S^{2n+1},
$$
where  $S^1_{({\fk a},1)}$ acts as
$
t\cdot (z,w)= (t^\fk a z, tw).$ Hence $\msf P_{(\fk a,1)}$ is obtained by adding
$\msf P_\fk a$ to $\cplane^n$ at infinity.
\item  the {\bf weight-$\wa$ blowup} of $\cplane^n$ at origin is defined to be
$$[\underline{\cplane^n}]_\fk a
:=S^{2n-1}\times_{S^1(\fk a,-1)}\cplane,
$$
where $S^1$ acts on $S^{2n-1}$ and $\cplane$ with weights $\fk a$ and $-1$ respectively. There is a natural map
\begin{align}\label{E kappa-blowup-Cn}
\kappa: [\underline{\cplane^n}]_\fk a \rto \cplane^n;
\;\;\;\;
[z,w]\mapsto w^\fk a\cdot z,
\end{align}
where $\kappa^{-1}(0)=\msf P_\fk a$ is called the exceptional divisor.
\end{enumerate}
Let $\Gamma$ be a finite group.  Suppose that it acts on $\cplane^n$  via the representation
$$
\mu: \Gamma\rto GL(n,\cplane)
$$
and it commutes with the $S^1_\fk a$ action.
Then we have following constructions for the
orbifold $\Gamma^\mu\ltimes\cplane^n$ (here, we write $\Gamma^\mu$ for $\Gamma$):
\begin{equation}\label{eqn_proj}
\msf P_\fk a^\mu =\msf P_\fk a/\Gamma^\mu,
\;\;\;
[\overline{\cplane^n}]^\mu_\fk a=[\overline{\cplane^n}]_\fk a/ \Gamma^\mu,\;\;\;
[\underline{\cplane^n}]^\mu_\fk a=[\underline{\cplane^n}]_\fk a/\Gamma^\mu.
\end{equation}
In particular, when $\Gamma=\mathbb Z_r$ and the action weight is
\begin{equation}
\fk b=(\beta_1,\ldots, \beta_n), \;\;\; \mbox{ where }
\;\;\; 1\leq \beta_1,\ldots, \beta_n\leq r,
\end{equation}
we denote these spaces as $\msf P_\fk a^\fk b,
[\overline{\cplane^n}]^\fk b_\fk a$ and $[\underline{\cplane^n}]^\fk b_\fk a$ respectively.
\begin{remark}
Note that the range of $\beta$ is taken to be $[1,r]$
rather than $[0,r-1]$.
\end{remark}

We give an explicit description of the inertia space of $\msf P_\fk a^\mu$. Since
$$
\msf P_\fk a^\mu= (\Gamma^\mu\times S^1_\fk a)\ltimes S^{2n-1},
$$
its twisted sectors are parameterized by conjugate classes $(\delta),
\delta\in \Gamma\times S^1$.
The inertia space is
$$
\msf{IP}^\mu_\fk a=\bigsqcup_{(\delta),\,\delta\in \Gamma\times S^1}
(S^{2n-1})^\delta\ltimes C_{\Gamma\times S^1}(\delta).
$$
Be precise, suppose that $\delta=(\eta, t)$, then
$$
C_{\Gamma\times S^1}(\delta)=C_\Gamma(\eta)\times S^1.
$$
Denote the fixed point set of $\delta$ by
$$
\cplane^{(\delta)}=(\cplane^n)^\delta,\;\;\;
S^{(\delta)}=(S^{2n-1})^\delta.
$$
 Then
$$
\msf P_\fk a^\mu{(\delta)}=\frac{S^{(\delta)}}{C_\Gamma(\eta)\times S^1_{\fk a(\delta)}}
$$
  where $\fk a(\delta)$ is the  action weight on $\cplane^{(\delta)}$.

  Now suppose that $\Gamma=\mathbb Z_r=\langle\zeta_r\rangle$ with $\zeta_r=e^{2\pi i\frac{1}r}$
  and the action weight is $\fk b=(\beta_1,\ldots,\beta_r)$, i.e,
  $$
  \zeta_r\cdot (z_1,\ldots,z_n)=(\zeta_r^{\beta_1}z_1,\ldots,
  \zeta_r^{\beta_n}z_n).
  $$
We give an explicit description for its twisted sectors.
Let $\msf p_i=[\ldots,0,z_i=1,0,\ldots]$. Denote the isotropy group of $\msf p_i$ by $G_i$. Then
\begin{equation}
G_i=\{(e^{-2\pi i\frac{ b}{r}}, e^{2\pi i\frac{b\beta_i+ar}{\alpha_ir}})| 0\leq b\leq r-1,
0\leq a\leq \alpha_i-1\}.
\end{equation}
We denote the group elements by $g_{i,a,b}$. We list some easy lemma without proofs.

\begin{lemma}
The index set of twisted sectors of $\msf P^\fk b_\fk a$
$$
\scr T\subseteq \mathbb Z_r\times S^1
$$
is bijective to the union of $G_i\subseteq \mathbb Z_r\times S^1$.
\end{lemma}

\begin{lemma}
For $\delta\in \scr T$, let
$I(\delta)=\{i|\delta\in G_i\}$. Then
$$
\cplane^{(\delta)}=\{z\in \cplane^n| z_j=0 \iff j\notin I(\delta)\},
$$
and
$$
\msf P^\fk b_\fk a(\delta)=
S^{(\delta)}\ltimes (\mathbb Z_r\times S^1)
$$
which is a weighted projective space spanned by $\msf p_i, i\in I(\delta)$.
\end{lemma}
\begin{lemma}
Suppose that
$$
\delta=(e^{-2\pi i\frac{b}{r}}, e^{2\pi iR}).
$$
The degree shifting of $\delta$ in $\msf P^\wb_\wa$ (cf. \cite{CR04}) is
\begin{equation}\label{eq degree-shift}
degsh(\delta)=
\sum_{u\not\in I(\delta)}\left\{-\frac{b}{r}\beta_u
+\alpha_uR
\right\}=\sum_{u}\left\{-\frac{b}{r}\beta_u
+\alpha_uR
\right\}
\end{equation}
Here $\{c\}:=c-[c]$ is the fractional part of $c$.
\end{lemma}
\begin{proof}
For $1\leq u\leq n$, $\delta$ acts on $z_u$ via
\begin{equation}
\delta z_u= e^{-2\pi i{\frac{b\beta_u}{r}}}
e^{2\pi i\alpha_u R}z_u.
\end{equation}
In particular when $u\in I(\delta)$,
$\delta z_u=z_u$, that is $-\frac{b}{r}\beta_u
+\alpha_uR\in\integer$.
Then by the definition of degree shifting, the formula follows.
\end{proof}
\begin{remark}
In later computation, we will use the case that
$$
\delta=(\zeta_r^{-1}, e^{2\pi iR}).
$$
Define
\begin{align}\label{eq def-tau}
\tau(R,u):=-\frac{\beta_u}{r}+\alpha_u R.
\end{align}
Then
$$
degsh(\delta)=\sum_u \{\tau(R,u)\}.
$$
\end{remark}

\subsection{Weighted blowups of vector bundles}\label{sec blp-bdle}

Now let $\pi:\msf E\rto\msf B $ be a rank $n$ complex vector bundle.  Suppose that $\msf P$ is the $K$-principal bundle of $\msf E$, i.e,
$$
\msf E=\msf P\times_K\cplane^n
$$
where $\msf P$ is a $K$-principal bundle. Let $S^1_\fk a$
be an action on $\cplane^n$ that commutes with $K$-action.
Then we make the following definitions:
\begin{enumerate}
\item the weight-$\wa$ projectivization of $\msf E$ with respect to the $S^1_\fk a$ action:
    $$
    \msf {PE}_\fk a=\msf P\times_K \msf P_\fk a,
    $$
\item the weight-$\wa$ blowup of $\msf E$ along $\msf B$ with respect to the
$S^1_\fk a$ action:
$$
\underline{\msf E}_\fk a= \msf P\times_K [\underline{\cplane^n}]_\fk a,
$$
\item
 the weight-$\wa$  projectification of $\msf E$ with respect to the
$S^1_\fk a$ action:
$$
\overline{\msf E}_\fk a= \msf P\times_K [\overline{\cplane^n}]_\fk a.
$$
\end{enumerate}

There are natural morphisms from
$\msf{PE}_\fk a, \overline{\msf{E}}_\fk a$ and
$\underline{\msf E}_\fk a$ to $\msf B$ induced from $\pi:\E\rto\B$, we denote
 them all by $\pi$ and call them projections. There is a natural morphism $\kappa:\underline{\msf E}_\fk a\rto \msf E$ induced by
 $\kappa:[\underline{\cplane^n}]_\fk a\rto \cplane^n$ (cf. \eqref{E kappa-blowup-Cn}), which we still denote by $\kappa$. Then
$${\kappa}^{-1}(\msf B)\cong\msf{PE}_\fk a$$ is the exceptional divisor of
the weight-$\wa$ blowup. There is an infinity divisor $\msf{PE}_\fk a$ in $\overline{\msf E}_\fk a$ and
$$
\overline{\msf E}_\fk a\setminus \msf {PE}_\fk a=\msf E.
$$

We now describe some relations between $\msf {PE}_\fk a$
and $\msf B$.

First of all, as in the smooth case, we still have
\begin{align}\label{E cohomo-of-PE}
H^\ast(\msf{PE}_\fk a)\cong H^\ast (\msf B)\{1,H,\ldots, H^{n-1}\}
\end{align}
where $H$ is the first Chern class of tautological line bundle of $\msf{PE}_\fk a$.

Second, the inertia spaces of them have the following relation.
\begin{lemma}\label{lem pit-twist-sec}
The projection $\pi: \msf{PE}_\fk a\rto \msf B$ induces a morphism between their inertia spaces:
$$
{\msf I}\pi:\msf{IPE}_\fk a\rto\msf {IB}.
$$
In particular, it induces a map on $\scr T$:
$$
\pi_{\scr T}: \scr T^{\msf{PE}_\fk a}\rto \scr T^{\msf B}.
$$
When restricting on twisted sectors,
$$
\msf I\pi: \msf{PE}_\fk a(\delta)\rto \msf{B}(\pi_{\scr T}(\delta))
$$
is a projectivization of certain vector bundle.
\end{lemma}
\begin{proof}
 We give an explicit description of morphisms  by using local coordinate charts.

 Let $U\rtimes G_x$ be a local orbifold chart
for $\mathbf{U}\subseteq|\msf B|$ such that the bundle $\msf E$ over $\mathbf U$ can be trivialized as
$$
\msf E|_{\mathbf U}\cong G_x \ltimes(U\times\cplane^n).
$$
Then
$$
\msf{PE}_\fk a|_\mathbf U\cong(G_x\times S^1_\fk a)\ltimes(U\times S^{2n-1}).
$$
Locally,
\begin{enumerate}
\item a twisted sector of $\msf B$ is indexed by a conjugate class $(k), k\in G_x$, and is
$C_{G_x}(k)\ltimes U^k$;
\item
a twisted sector of $\msf {PE}_\fk a$
is indexed by a conjugate class $(\delta)$,
\begin{align}
\label{E delta=(k,t)}
\delta=(k,t)\in G_x\times S^1,
\end{align}
and  is
a fibration
$$
\msf I\pi:(C_{G_x}(k)\times S^1)\ltimes(U\times S^{2n-1})^\delta
\rto C_{G_x}(k)\ltimes U^k.$$
which is the projectivization of
$$
C_{G_x}(k)\ltimes(U\times \cplane^n)^\delta\rto C_{G_x}(k)\ltimes U^k .
$$
\end{enumerate}
Set $\pi_{\scr T}(\delta)=(k)$. This gives us the map $\pi_{\scr T}$.
\end{proof}

\subsection{Weighted blowups of $\msf X$ along $\msf S$}
\label{subsec blp-X-S}

Suppose that the real codimension of $\msf S$ is $2n$.
Let $\msf N$ be the normal bundle of $\msf S$.
Then the weight-$\fk a$ blowup of $\msf N$ induces a
weight-$\fk a$ blowup of $\msf X$ along $\msf S$, we denote the resulting orbifold groupoid by $\Xa$.

As in the smooth case, $\msf X$ then degenerates into two components:
$$
\X^-=\Xa,\;\;\;     \X^+=\lN_\wa.
$$
They have a common divisor $\msf Z=\msf {PN}_\fk a$. We write
$$
\X\xrightarrow{\text{degenerate}}
\X^-\wedge_{\msf{PN}_\fk a} \msf X^+,
\qq
(\X,\oS)\xrightarrow{\text{degenerate}}
\X^-\wedge_{\msf {PN}_\fk a}(\X^+,\oS).
$$

There is also a natural projection $\kappa:\Xa\rto\X$ obtained from $\kappa:\lN_\wa\rto\N$. Hence there is an induced morphism on inertia spaces $\msf I\kappa:\msf I\Xa\rto\msf I\X$ and an induced map on the index set of twisted sectors as Lemma \ref{lem pit-twist-sec}:
\[
\kappa_{\scr T}:\scr T^{\Xa}\rto\scr T^\X.
\]
$\kappa$ is an isomorphism out of a neighborhood of $\Z$ and of $\oS$.

For a twisted sector $\X(t)$ of $\X$,
\begin{itemize}
\item if $\X(t)\cap \sf IS=\varnothing$, after blowing up it becomes a twisted sector of $\Xa$, which we denote by $\Xa(t)$.
\item if $\X(t)\cap \sf IS\neq \varnothing$, after blowing up it lifts to several twisted sectors of $\Xa$ with one of them corresponding to $\delta=(k,1)$ in \eqref{E delta=(k,t)}; we also denote this one by $\Xa(t)$.
\end{itemize}
For both cases, we call $\Xa(t)$ the direct transformation of $\X(t)$. We denote the set of indices of this kind of twisted sectors of $\Xa$ by $\scr T^\Xa_1$. With these notation, $\kappa_{\scr T}$ restricts to a bijection $\scr T^\Xa_1\rto \scr T^\X$, and $\msf I\kappa$ restricts to $\msf I\kappa:\bigsqcup\limits_{(t)\in\scr T^\Xa_1}\Xa(t)\rto \bigsqcup\limits_{(t)\in\scr T^\Xa_1} \X(\kappa_{\scr T}(t))$.
In particular, $\scr T^\Xa_{\sf ker}\subseteq \scr T^\Xa_1$ and $\kappa_{\scr T}$ matches it up with $\scr T^\X_{\sf ker}$. Therefore $\msf I\kappa$ maps $\msf{ker\,\underline{X}}_{\wa}$ to $\msf{ker\,X}$, and matches up their connected components. 


\section{Orbifold Gromov--Witten theory}\label{sec 4}

In this section, we briefly review the Gromov--Witten
theory for orbifolds. We emphasize three types of invariants: absolute Gromov--Witten invariants of a compact symplectic orbifold $\msf X$;
relative Gromov--Witten invariants of a relative pair
$(\msf X|\msf Z)$, where $\msf Z$ is a divisor of $\msf X$;
absolute Gromov--Witten invariants of a pair $(\msf X,\msf S)$, where $\msf S$ is a compact symplectic sub-orbifold of $\msf X$. The main purpose of this section is to introduce notations used in this paper.

\subsection{Absolute Gromov--Witten invariants}\label{sec 4.1}

Let $\om_{g,A,\fk t}(\msf X)$ be the moduli space of
 holomorphic curves of the type that
 \begin{enumerate}
 \item the domain curve is connected and of genus $g$,
 \item there are
$k$-marked points each of which is associated with a twisted sector $\msf X(t_i), 1\leq i\leq k$,
and $\fk t=(t_1,\ldots,t_k)$;
\item the curve
represents a homology class $A\in H_2(\msf X,\mathbb Z)$.
\end{enumerate}
Let ${\bf A}$
denote the combinatoric topological data $(g,A,\fk t)$. We may denote the moduli space by $\om_{\bf A}(\msf X)$.

To associate invariants to the moduli space $\om_{\bf A}(\msf X)$ , for each marking we associate it an insertion
$\tau_{d_i}\alpha_i$ where $\alpha_i\in H^\ast(\msf X(t_i))$ and $d_i\geq 0$. Denote
the insertion data by
$$
I=(\tau_{d_1}\alpha_1,\ldots, \tau_{d_k}\alpha_k).
$$
The Gromov--Witten invariant is defined to be
$$
\langle{\bf A}(I)\rangle=\int^{\mathrm{vir}}_{\om_{\bf A}(X)}\scr F_I,\qq
\mbox{where}\qq \scr F_I=\prod_{i=1}^k (\psi^{d_i}\cup ev_i^\ast\alpha_i).
$$
When at least one of $d_i$ is positive, $\langle{\bf A}(I)\rangle$ is called a descendent invariant, otherwise a non-descendent invariant or a primary invariant.

When the degree of $\scr F_I$ matches
the virtual dimension of the moduli space, we call ${\bf A}(I)$
an  {\bf admissible data} of (absolute) Gromov--Witten invariant. Let $\scr A(\msf X)$ denote the set of admissible data. It is convention to use $\scr A^\bullet(\msf X)$  when we allow the domain curve to be disconnected.

\subsection{Relative Gromov--Witten invariants}\label{sec 4.2}

Now consider a relative pair $(\msf X|\msf Z)$ where $\msf Z$ is a divisor of $\msf X$. Let
$$\om_{g,A, \fk t|\fk r}(\msf X|\msf Z)$$
denote the relative moduli space of stable curves, which has following properties:
\begin{enumerate}
\item the domain curve is connected, and of genus $g$;
\item there are $k$ (absolute) marked points, each of which is associated with a twisted sector $\msf X(t_i)$ of $\msf X$, and $\fk t=(t_1,\ldots, t_k)$,
\item there are $h$ relative marked points, each of which is associated with a twisted sector $\msf Z(s_j)$ of $\msf Z$ and a contact order $\ell_j$,
and
$$
\fk r=((s_1,\ell_1),\ldots, (s_h,\ell_h)),
$$

 \item
 the curve represents a homology class $A\in H_2(\msf X,\mathbb Z)$.
 \end{enumerate}
Let ${\bf R}$ denote the combinatoric topological data $(g,A, \fk t|\fk r)$. We denote the moduli space by $\om_{\bf R}(\msf X|\msf Z)$.

Now for each marking, we associate it an insertion:
\begin{itemize}
\item for $i$-th absolute marking, an insertion
 $\tau_{d_i}\alpha_i$, where $\alpha_i\in H^\ast(\msf X(t_i))$;
\item for $j$-th relative marking, an insertion
$\beta_j\in H^\ast(\msf Z(s_i))$.
\end{itemize}
Let
$$
I=(\tau_{d_1}\alpha_1,\ldots, \tau_{d_k}\alpha_k),\;\;\;
J_{\msf Z}=(\beta_1,\ldots, \beta_h)$$ be the collection of absolute insertions and relative insertions.

 A relative Gromov--Witten invariant is defined to be
$$
\langle{\bf R}(I|J_{\msf Z})\rangle=
\int_{\om_{\bf R}(\msf X|\msf Z)}^{\mathrm{vir}}
 \scr F_I\cup\scr G_{J_{\msf Z}}
 $$
  where
$$
\scr F_I= \prod_{i=1}^k (\psi_i^{d_i}\cup ev_i^\ast\alpha_i),\qq
\scr G_{J_{\msf Z}}=\prod_{j=1}^h (rev_j^\ast\beta_j),
$$
 and $rev_j: \om_{\bf R}(\msf X|\msf Z)\to \msf Z$ denotes the $j$-th relative evaluation map.
Again we call  ${\bf R}(I|J_{\msf Z})$ an
admissible data of relative Gromov--Witten invariant
if the total degree of $\scr F_I$ and
  $\scr G_{J_{\msf Z}}$
  matches the virtual dimension of the moduli space.
Let
\[
\scr{R}(\msf X|\msf Z), \qq (\scr R^\bullet(\msf X|\msf Z)  \mbox{ resp.} )
\]
 be the collections of admissible data of relative Gromov--Witten invariants.

\begin{remark}\label{rmk enlarge-JZ}
In the following for an admissible relative data ${\bf R}(I|J_\Z)$ with ${\bf R}=(g,A, \fk t|\fk r)$ and $\fk r=((s_1,\ell_1),\ldots, (s_h,\ell_h))$, we will assume that $J_\Z$ denotes not only the relative insertions but also the relative data $\fk r$, i.e.
\[
J_\Z=((s_1,\ell_1,\beta_1),\ldots, (s_h,\ell_h,\beta_h)).
\]
Then we use the following notations
\[
\fk r(J_\Z):=\fk r=((s_1,\ell_1),\ldots, (s_h,\ell_h)),\qq
\fk i(J_\Z)=(\beta_1,\ldots, \beta_h).
\]
\end{remark}

\begin{remark}\label{rmk Sigma-relative-data}

Let $\Sigma$ be a basis of $H^\ast_{CR}(\msf Z)$. If the relative insertion $\beta_j\in \Sigma$,  we call
 ${\bf R}(I|J_{\msf Z})$ a $\Sigma$-relative admissible data. Let $\scr{R}_\Sigma(\msf X|\msf Z)$  ($\scr{R}^\bullet_\Sigma(\msf X|\msf Z)$ resp.) denote the set of $\Sigma$-data.
\end{remark}

\subsection{The degeneration formula}\label{sec 4.3-dege-formula}

\subsubsection{Degenerate absolute invariants}
Suppose that
$\msf X$ degenerates to a pair of
orbifolds $\msf X^{\pm}$ with a common divisor $\msf Z$. Let $\Sigma^+$ be a basis of $H^\ast_{CR}(\msf Z)$
and $\Sigma^-$ be its orbifold Poincar\'e dual.

\begin{defn}\label{def matched-pair}
Let $${\bf R}^\bullet_\pm(I^\pm|J^\pm_{\msf Z})\in
\scr R^\bullet_{\Sigma^\pm}(\msf X^\pm|\msf Z)$$
be a pair of admissible relative datum. We say that they form a {\bf matched pair}  if $J^+_{\msf Z}$ and $J^-_{\msf Z}$ are matched in the following sense:
they have same cardinality, and for each $j$,
$$
s_j^+=(s_j^-)^{-1},\;\;\; \ell_j^+=\ell_j^-,\;\;\;
\beta^+_j=(\beta^-_j)^\ast.
$$

\n If it is this case, with the understanding in Remark \ref{rmk enlarge-JZ} we write
\[
J^+_{\msf Z}=\check J^-_{\msf Z}.
\]
\end{defn}
 Suppose that
 \begin{equation}
 ({\bf R}^\bullet_+(I^+|J^+_{\msf Z}),
 {\bf R}^\bullet_-(I^-|J^-_{\msf Z}))
 \end{equation}
 is a matched pair. The degeneration theory of Gromov--Witten invariants says that the pair
  can be glued to an admissible data of an absolute Gromov--Witten invariant
${\bf A}^\bullet(I)$ where $I=I^+\cup I^-$. We write
\begin{equation}\label{eqn_glu}
{\bf A}^\bullet(I)={\bf R}^\bullet_+(I^+|J^+_{\msf Z})
\ast {\bf R}^\bullet_-(I^-|J^-_{\msf Z}).
\end{equation}
Let
$\scr D_{{\bf A}^\bullet(I)}(I^+,I^-)$ denote the collection of matched pairs $({\bf R}_\pm^\bullet(I^\pm|J^\pm_{\msf Z}))$ such that \eqref{eqn_glu} holds.

Then  we have a degeneration formula
\[
\langle{\bf A}^\bullet(I)\rangle
=\sum_{\Omega\in \scr D_{{\bf A}^\bullet(I)}(I^+,I^-)}
c_{\Omega}
\cdot
\langle {\bf R}^\bullet_-(I^-|J^-_{\msf Z})\rangle^{\X^-|\Z}
\cdot
\langle {\bf R}^\bullet_+(I^+|J^+_{\msf Z})\rangle^{\X^+|\Z}.
\]
where $\Omega=({\bf R}^\bullet_-(I^-|J^-_{\msf Z}), {\bf R}^\bullet_+(I^+|J^+_{\msf Z}))$, and $c_{\Omega}$ is a constant depends
on $\fk r_\pm$ in ${\bf R}^\bullet_\pm$.
It would happen that one of the admissible relative data in $\Omega=({\bf R}^\bullet_-(I^-|J^-_{\msf Z}), {\bf R}^\bullet_+(I^+|J^+_{\msf Z}))$ is empty, i.e., ${\bf R}^\bullet_-(I^-|J^-_{\msf Z})=\varnothing$, or ${\bf R}^\bullet_+(I^+|J^+_{\msf Z})=\varnothing$, then the coefficient $c_\Omega=1$, and the invariant for the empty relative data is also set to be $1$.

\subsubsection{Degenerate relative invariants}

We also have the degeneration formula for relative orbifold Gromov--Witten invariants. We consider a special case which we will use latter. Let $(\X|\Z)$ be a relative pair and ${\bf R}^\bullet(I|J)$ be an admissible relative data. We blow up $(\X|\Z)$ along $\Z$ with trivial weight $\wa=(1)$ to degenerate $(\X|\Z)$ in to
\[
(\X|\Z)\xrightarrow{\text{degenerate}} (\X^-|\Z^-)\wedge_\Z(\Z^+|\X^+|\Z)=(\X|\Z)\wedge_\Z(\Z_\infty|\lN|\Z_0)
\]
where $\lN$ is trivial weight projectification of the normal bundle of $\Z$ in $\X$ and $\Z_0,\Z_\infty$ are the zero and infinity sections. Then we have
\begin{align}\label{eq dege-formula-rel}
\langle{\bf R}^\bullet(I|J)\rangle^{\X|\Z}
=\sum_{\Omega\in \scr D_{{\bf R}^\bullet(I|J)}(I^+,I^-)}
c_{\Omega}
\cdot
\langle {\bf R}^\bullet_-(I^-|J^-_{\msf Z})\rangle^{\X|\Z}
\cdot
\langle {\bf R}^\bullet_+(J^+_{\msf Z}|I^+|J)\rangle^{\Z_\0|\lN|\Z_0}
\end{align}
where $\scr D_{{\bf R}^\bullet(I|J)}(I^+,I^-)$ denote the collection of matched pairs
$$
\Omega=({\bf R}^\bullet_-(I^-|J^-_{\msf Z}),{\bf R}^\bullet_+(J^+_{\msf Z}|I^+|J))
$$
(they match along $\Z\in\X$ and $\Z_\infty\in\X^+$ like Definition \ref{def matched-pair}, then $J^-_{\Z}=\check J^+_\Z$), such that
\[
{\bf R}^\bullet(I|J)={\bf R}^\bullet_-(I^-|J^-_{\msf Z})\ast{\bf R}^\bullet_+(J^+_{\msf Z}|I^+|J).
\]

We can also consider nontrivial weight degeneration of $\X$ along $\Z$.

By the Gromov compactness theorem, for a fixed absolute (resp. relative) data ${\bf A}^\bullet(I)$ (resp. ${\bf R}^{\bullet}(I|J)$) the set $\scr D_{{\bf A}^\bullet(I)}(I^+,I^-)$ (resp. $\scr D_{{\bf R}^\bullet(I|J)}(I^+,I^-)$) is a finite set.

\subsection{Absolute Gromov--Witten invariants of symplectic pairs}
\label{subsec blp-X-S-inv}

 Let $\msf S$ be a compact symplectic sub-orbifold of $\msf X$. A cohomology class $\alpha\in
H^\ast_{CR}(\msf X)$ is called to be supported on $\msf S$ if it is of the form $\theta\cup [\msf {IS}]$, where $\theta\in H^\ast_{CR}(\msf S)$ and $[\msf {IS}]$ denotes the Thom class of $\msf {IS}$ in $\msf{IX}$. If we fix
 a basis $\sigma$ of $H^\ast_{CR}(\msf S)$ and let $
\sigma^\ast$ be its dual, we say $\theta\cup [\msf {IS}]$ is a $\sigma$-class if $\theta\in \sigma$.

\begin{defn}
We call an admissible data ${\bf A}(I)$ of
absolute Gromov--Witten invariant
is a data {\bf relative to} $\msf S$ if, for
each insertion $\tau_{d_i}\alpha_i$,  $d_i>0$ only when
$\alpha_i$ is supported on $\msf S$. We call these  insertions to be  {\bf $\msf S$-supported insertions}.
${\bf A}(I)$ is called a {\bf $\sigma$-data} relative to
$\msf S$ if those class $\alpha_i$ supported on $\msf S$ are $\sigma$-classes.

We separate $I$ into two parts: let $I_{\msf S}$ denote the collection of $\msf S$-supported insertions and $I'$ denote the rest. We write ${\bf A}(I)$ to be
${\bf A}(I'; I_{\msf S})$. Let $\scr{A}^\bullet(\msf X,\msf S)$
(resp. $\scr A^\bullet_\sigma(\msf X,\msf S)$) denote the set of (resp. $\sigma$-)
admissible data of absolute invariants relative to $\msf S$.
\end{defn}

We may apply the degeneration formula to $\langle {\bf A}(I;I_{\msf S})\rangle $ for the
degeneration
$$
(\msf X;\msf S)\xrightarrow{\text{degenerate}}
\msf X^-\wedge_\Z(\msf X^+,\msf S)=\Xa\wedge_\Z(\overline{\msf N}_\fk a,\msf S),
$$
then we have
\begin{equation}
\langle{\bf A}(I;I_{\msf S})\rangle
=\sum_{\Omega\in \scr D_{{\bf A}(I;I_{\msf S})}(I^+,I^-)}
c_{\Omega}
\cdot
\langle{\bf R}^-(I^-|J^-_{\msf Z})\rangle^{\Xa|\Z}
\cdot
\langle{\bf R}^+(I^+;I_{\msf S}|J^+_{\msf Z})\rangle^{\lN_\wa|\Z},
\end{equation}
where $\Omega=({\bf R}_-(I^-|J^-_{\msf Z}), {\bf R}_+(I^+;I_{\msf S}|J^+_{\msf Z}))$ and $J^-_\Z=\check J^+_\Z$.

\section{Certain invariants of $(\overline{\msf E}_{\fk a}|\sf PE_{\fk a})$}\label{sec 5}

In this section we study certain relative invariants
of $(\overline{\msf E}_{\fk a}|\sf PE_{\fk a})$. We treat $\msf B$ as the 0-section $\msf E$.

Description of topological data ${\bf R}$ and insertions:
\begin{enumerate}
\item  The domain curve
$\msf S^2$ is the
orbifold sphere  with 2 orbifold points at $0$ and $\infty$.
Be precise, there are two orbifold charts
$$
\msf C_0=\mathbb Z_{r}\ltimes \cplane,\;\;\;
 \msf C_\infty =\mathbb Z_{h}\ltimes \cplane,
 $$
 respectively. Set
 $$
 \msf q_0=\mathbb Z_{r}\ltimes \{0\},\;\;\;
 \msf q_\infty=\mathbb Z_{h}\ltimes\{\infty\}.
 $$

\item Marked points: $\msf q_0$ is an absolute marked point and
$\msf q_\infty$ is a relative marked point.

\item The insertion for $\msf q_0$ is supported in
$\msf B$: $\tau_c(\theta\cup[\sf IB]),$ where $\theta
\in H^\ast(\msf B(t))$ and $[\sf IB]$ is the Thom form of $\msf B(t)$ in $\msf E(t)$.

\item The insertion for $\msf q_\infty$ is $\beta\in
H^\ast(\msf {PE}_\fk a(s))$; suppose that $\ell$ is
the contact order at $\msf q_\infty$.

\item $A\in H_2(|\overline{\msf E}_\fk a|)$ is of fiber
class, i.e, $\pi_\ast(A)=0$.
\end{enumerate}
Then \def \Hol{\mathrm{Hol}} \def \Aut{\mathrm{Aut}}
\begin{equation}
{\bf R}=(g=0, A, \fk t=(t)|\fk r=(s,\ell)),\;\;\;
I_{\msf B}=(\tau_c(\theta\cup[\msf B(t)])),\;\;\;
J_{\msf {PE}_\fk a}=(\beta).
\end{equation}
Let $\om_{\bf R}(\overline{\msf E}_\fk a|{\msf{PE}_\fk a})$ be the moduli space. The top stratum of the moduli space is
\[
\scr M_{\bf R}(\overline{\msf E}_\fk a|{\msf{PE}_\fk a})
=\frac{\msf{Hol}_{\bf R}(\overline{\msf E}_\fk a|{\msf{PE}_\fk a})}{\Aut({\msf S^2\setminus\{\msf q_0,\msf q_\infty\}})}.
\]
Here $\msf{Hol}_{\bf R}(\overline{\msf E}_\fk a|{\msf{PE}_\fk a})
\subseteq \msf{Mor}(\overline{\msf E}_\fk a|\msf{PE}_\fk a)$
consists of holomorphic morphisms. Usually,  $\msf{Hol}_{\bf R}(\overline{\msf E}_\fk a|{\msf{PE}_\fk a})$
is  denoted by $\widetilde{\scr M}_{\bf R}(\overline{\msf E}_\fk a|
\msf{PE}_\fk a)$.  Here, we use the new notation to emphasize its groupoid structure.

\begin{remark}
Suppose that $\msf u:\msf S^2\rto \overline{\msf E}_\fk a$
is a holomorphic morphism in $\Hol^0_{\bf R}(\overline{\msf E}_\fk a|\msf{PE}_\fk a)$.
 Then
 \begin{enumerate}
\item the morphism covers a fiber, i.e, for
$
\msf S^2\xrightarrow{\msf u}\overline{\msf E}_\fk a\xrightarrow{\pi}\msf B,$
the induced map $|\pi\circ \msf u|:|\msf S^2|\rto|\msf B|$ on coarse spaces is constant;
\item $\msf u$ maps $\msf C_0$ to $\msf E$;
 \item  $\msf u: \msf q_0\rto\msf E(t)$ and
$\msf u: \msf q_\infty\rto \msf{PE}_\fk a(s)$.
\end{enumerate}
\end{remark}
The goal in this section is to understand invariants
$$
\langle {\bf R}(I_{\msf B}|J_{\msf {PE}_{\fk a}})\rangle
$$
under certain conditions.

We first give a completely understanding of invariants  when $\msf E=\mathbb Z_r^\fk b\ltimes \cplane^n$, an orbifold bundle over $\msf B_r$. This is explained in \S\ref{sec 5.1}--\S\ref{sec 5.2}. Then in \S\ref{sec 5.3} and \S\ref{sec 5.4} we reduce the study of general cases to this special case
when ${\bf R}(I_{\msf B}|J_{\msf {PE}_{\fk a}})$ is under certain conditions.

\subsection{Special case: $\msf E=\mathbb Z_r^{\fk b}\ltimes \cplane^n$}\label{sec 5.1}

We start with this special case. Then
$$\overline{\msf E}_{\fk a}=[\overline{\cplane^n}]^{\fk b}_{\fk a}=\msf P_{(\fk a,1)}^{(\fk b,0)},\;\;\;
\msf{PE}_{\fk a}^{\fk b}=\msf P_{\fk a}^{\fk b}.
$$
 By the definition of the moduli space,
\begin{equation}
\scr M_{\bf R}(\oC|\msf P^{\fk b}_{\fk a})=\frac{\msf{Hol}_{\bf R}
(\msf S^2,\oC)}{\Aut(\msf S^2\setminus\{\msf q_0,\msf q_\infty\})}
\end{equation}
Without loss of generality, we make the following assumption.
\begin{assumption}\label{assump}
 The twisted sector $t$ associated to
$\msf q_0$ is
$\zeta_r=\exp^{2\pi i\frac{1}r}\in \mathbb Z_r$.
\end{assumption}
The invariant to compute is
\begin{equation}\label{eq rel-inv-C}
\langle{\bf R}(\tau_c(\Theta_{(t)})|H_{(s)}^d)\rangle
=\int_{\om_{\bf R}} \psi^c\cup\msf{ev_0}^\ast\Theta_t
\cup \msf{rev}^\ast_\infty H_{(s)}^d,
\end{equation}
where $\Theta_{(t)}$ is the Thom class of $\msf E(t)\rto \msf B(t)$ and $H_{(s)}$ is the generator of $H^*(\msf{P}^\wb_\wa(s))$.
The following theorem is the main theorem of the section, and is the most important technical result for this paper.
\begin{theorem}\label{thm c-dertmin-all}
Suppose the data ${\bf R}(\tau_c(\Theta_{(t)})|H_{(s)}^d)$  is
admissible and satisfies Assumption \ref{assump}. Then
\begin{enumerate}
\item[(1)] $c$ can be any nonnegative integer,
\item[(2)]
$\bf R$ and $d$ are uniquely determined by $c$,
 \item[(3)]
 the invariant $\langle{\bf R}(\tau_c(\Theta_{(t)})|H_{(s)}^d)\rangle$ is non-zero.
\end{enumerate}
\end{theorem}
From now on, in the rest of subsection we write $\scr M_{\bf R}$ for top strata of the moduli space.

\subsubsection{Description of $\scr M_{\bf R}$.}\label{sec MR}
We want to give a more explicit description of the right hand side of \eqref{eq rel-inv-C}. This consists of two steps.
\vskip 0.1in
\noindent
{\em Step 1. Compare $\msf {Hol}_{\bf R}(\msf S^2,\oC)$
with $\msf {Hol}(\msf C_0,\mathbb Z_r^\fk b)\ltimes \cplane^n$.}
\noindent
The restriction of morphisms from $\msf S^2$ to $\msf C_0$
induces a groupoid morphism
 $$
 \Psi:
 \msf{Hol}_{\bf R}(\msf S^2,\oC)\hookrightarrow
 \msf{Hol}(\msf S^2,\oC)
 \xrightarrow{|_{\msf C_0}}\msf{Hol}(\msf C_0,\oC)
 \to \msf{Hol}(\msf C_0,
 \mathbb Z_r^\fk b\ltimes\cplane^n).
 $$
 The last arrow follows from the fact that
  $\msf u: \msf C_0\rto \msf E$ when $\msf u\in \msf{Hol}_{\bf R}$.

Conversely, for
$$
\msf v\in \Hol^0(\msf C_0,
 \mathbb Z_r^\fk b\ltimes\cplane^n)$$
(with finite energy), by the removable singularity theorem, $\msf v$
can be extended to define a morphism
$$
\Phi^0(\msf v):=\bar {\msf v}:\msf S^2 \rto \oC.
$$
We say that $\msf v$ is of ${\bf R}$-type if $\bar {\msf v}\in \Hol^0_{\bf R}(\msf S^2,\oC)$. Let
$$
\Hol_{\bf R}(\msf C_0,\mathbb Z_r^\fk b\ltimes\cplane^n)\subseteq \Hol(\msf C_0,\mathbb Z_r^\fk b\ltimes\cplane^n)
$$
be the space of morphisms of ${\bf R}$-type. Then removable singularity theorem for orbifold holomorphic morphisms says that
$$
\Phi:\msf{Hol}_{\bf R}(\msf C_0,\mathbb Z_r^\fk b\ltimes\cplane^n
)\rto \msf{Hol}_{\bf R}(\msf S^2,
\oC)
$$
is an equivalence, which is the inverse of $\Psi$. On the other hand,
\begin{lemma}
$\msf{Hol}_{\bf R}(\msf C_0,\mathbb Z_r^\fk b\ltimes\cplane^n
)=\mathbb Z_r\ltimes\left(\Hol^0_{\bf R}(\msf C_0,\mathbb Z_r^\fk b\ltimes\cplane^n
)\right).
$
\end{lemma}
\begin{proof}
For any $\msf u\in \Hol^0$, the arrows from $\msf u$ are
identified with constant maps from $\cplane$ to $\mathbb Z_r$. Hence $\mathrm{Hom}(\msf u,\cdot)\cong \mathbb Z_r$. For $\xi\in \mathbb Z_r$,
$$
(\xi\cdot\msf u)^0=\xi\cdot u^0,\;\; (\xi\cdot\msf u)^1=u^1.
$$
Hence $\Hol^1$ acts on $\Hol^0$ as a $\mathbb Z_r$-action.
\end{proof}

\vskip 0.1in
\noindent
{\em Step 2, treatment of automorphism groups.}
For $\msf C_0=\integer_r\ltimes\cplane$ we write
$$\msf C_0=\mathbb Z_r\ltimes\cplane_o,\;\;\;
\msf C_0^\ast=\mathbb Z_r\ltimes\cplane^\ast_o
\cong \cplane^\ast.
$$
We have an exact sequence
$$
1\to \mathbb Z_r\to \cplane^\ast_o\to \cplane^\ast\to 1.
$$
On the other hand,
$$
\cplane^\ast_o=\Aut(\cplane_o^\ast),\;\;\;
\cplane^\ast=\Aut(\cplane^\ast)=\Aut(\msf S^2\setminus
\{\msf q_0,\msf q_\infty\}).
$$
Hence,
$$
1\to \mathbb Z_r\to \Aut(\cplane^\ast_o)\to \Aut(
\msf S^2\setminus\{\msf q_0,\msf q_\infty\})\to 1.
$$
It is not hard to see that
$$
\left(\Hol^0_{\bf R}(\msf C_0,\mathbb Z^{\wb}_r\ltimes\cplane^n)/\mathbb Z_r
\right)/\Aut(\msf S^2\setminus\{\msf q_0,\msf q_\infty\})
\cong \left(\Hol^0_{\bf R}(\msf C_0,\mathbb Z^\wb_r\ltimes\cplane^n)
\right)/\Aut(\cplane_o^\ast).
$$
We conclude that
\begin{lemma}\label{lem 5.5}
Suppose that $\msf C_0= \mathbb Z_r\ltimes\cplane_o$. Then
\begin{equation}
\scr M_{\bf R}
\cong {\cplane^\ast_o}\ltimes{\Hol^0_{\bf R}(\msf C_0, \mathbb Z_r^\fk b\ltimes\cplane^n)}.
\end{equation}
\end{lemma}

\subsubsection{$\psi$-class associated to $\msf q_0$}
For the marking $\msf q_0$ one may associate a line bundle $\scr L$ over $\om_{\bf R}$: for each morphism $\msf u$ we assign it a complex line which is dual to $\cplane(\cong \cplane_o/\mathbb Z_r)$. Using Lemma \ref{lem 5.5} we may give a concrete expression for $\scr L$.

Let $\check{\cplane}_o$ be the dual space of $\cplane_o$.
$$
\tilde{\scr L}:=\frac{\check{\cplane}_o\times\left(
\Hol_{\bf R}^0(\msf C_0, \mathbb Z_r^\fk b\ltimes\cplane^n)\right)}{
\cplane^\ast_o}
\rto \frac{\Hol_{\bf R}(\msf C_0,\mathbb Z_r^\fk b\ltimes \cplane^n)}{
\cplane^\ast_o}=\scr M_{\bf R}
$$
is a line bundle. Since $\check\cplane\cong \check\cplane_o/\mathbb Z_r$, we have
$$
\tilde{\scr L}^{\otimes r}\cong \scr L.
$$
The psi-class is $\psi=c_1(\scr L)$.

\subsubsection{Description of ${\bf R}$}

We may assume that $\scr M_{\bf R}$ consists of
one of following morphisms
$$
\{\msf u[j,m]| 1\leq j\leq n, m\in \mathbb Z^+\}.
$$
Here
$\msf u[j,m]: \msf S^2\rto [\overline{\cplane^n}]^\fk b_\fk a$ is a representable morphism such that (cf. Assumption \ref{assump})
\begin{equation}\label{eqn_ujm}\begin{split}
&\msf u[j,m]: \msf C_0= \integer_r\ltimes\cplane_o\rto \mathbb Z_r^\fk b\ltimes\cplane^n;\\
&x\mapsto (\ldots,0,z_j=x^m, 0,\ldots),\,\, id: \integer_r\rto\integer_r.
\end{split}\end{equation}
\begin{lemma}\label{lem 5.6}
 For $\msf u[j,m]$  we have
\begin{enumerate}
\item $m\equiv \beta_j \pmod r$;
\item $\msf u[j,m]: [\overline{\cplane}]^1_p\rto [\overline{\cplane^n}]^\fk b_\fk a$, and
    $$
    u^0[j,m]: [x,y]\mapsto [\ldots 0, z_j=x^m,0,\ldots,0, w=y^{q}],
    $$
    where $p,q$ are determined by $p/q={\alpha_j}/{m}$ such that $(p,q)=1$,
\item $\msf u[j,m]:\langle
e^{2\pi i\frac{1}{pr}}\rangle\ltimes\{q_\infty\}\mapsto
\langle(e^{-2\pi i\frac{1}{r}},
e^{2\pi i\frac{m}{\alpha_jr}})\rangle\ltimes\{p_j\}.$
\end{enumerate}
Here we identify $\msf S^2$ as $[\overline{\cplane}]^1_p$.
\end{lemma}
\begin{remark}\label{rmk 5.7}
We make several important remarks:
\begin{itemize}
\item orbifold structure of $\msf S^2$ (at $\infty$)
 is determined by $\msf u[j,m]|_{\msf C_0}$, or by the pair $(j,m)$;
 \item the twisted sector $s$ associated to $\msf q_\infty$ given in (3) is determined by $(j,m)$ as well, moreover the twisted sector
\[
\pi_{\scr T}(s)=\bar t,
\]
 as $(e^{-2\pi i\frac{1}{r}},
e^{2\pi i\frac{m}{\alpha_jr}})\mapsto e^{-2\pi i\frac{1}{r}}$ (see \eqref{E def-bar-on-twistedsector} for the definition of $\bar \cdot$ on inertia spaces and on the index set $\scr T$);
\item
for two different pairs $(j_1,m_1)$ and $(j_2,m_2)$,  the twisted sectors associated to
$\msf q_\infty$ are the same if and only if
\[
\exp\left(2\pi i \frac{m_{1}}{\alpha_{j_1}r}\right)=\exp\left({2\pi i \frac{m_{2}}{\alpha_{j_2}r}}\right);
\]
\item for two different pairs
$(j_1,m_1)$ and $(j_2,m_2)$,
 they are in same $\scr M_{\bf R}$, i.e, they also
represent a same homology class if and only if
$$
\frac{m_{1}}{\alpha_{j_1}r}=\frac{m_{2}}{\alpha_{j_2}r}.
$$
\end{itemize}
We conclude that ${\bf R}$ is completely determined by
the number $\frac{m}{\alpha_jr}$. This leads to the following function
$$
\Lambda: [1,n]_{\mathbb Z}\times \mathbb Z_{\geq 0}\rto \mathbb Q,\;\;\;
\Lambda(j, a)=\frac{\beta_j+ar}{\alpha_jr}.
$$
Here $[1,n]_{\mathbb Z}=[1,n]\cap \mathbb Z$. We write $m$
as $\beta_j+ar, a\geq 0$. Set $\scr R$ be the image of
$\Lambda$.
\end{remark}
These observations implies the following lemma.
\begin{lemma}\label{lem 5.8}
Given $R\in \scr R$, it determines a topological data ${\bf R}$ uniquely by
$$
t=\zeta_r,\;\;\; s=(\zeta_r^{-1}, e^{2\pi i R}), \,\,\,\, A=R[F],
$$
and contact order $\ell=R$.  Conversely, if $\scr M_{\bf R}$ is non-empty, ${\bf R}$ must be obtained by this way.
\end{lemma}
Another way to say is that ${\bf R}$ is completely determined by the homology class $A$ under Assumption \ref{assump}. From now on
we denote the moduli space by $\scr M_R, R\in \scr R$.

\subsubsection{Dimension of $\scr M_R$.}
\vskip 0.1in \def \trk{rk^\circ} \def \brk{rk_\circ}
 In order to compute the virtual dimension of moduli space $\scr M_R, R\in \scr R$, we introduce
two functions (of ranking)  $rk^\circ$ and $rk_\circ$ on $\scr R$:
\begin{align}
\trk(R):=& \sharp\{(j,a)|\Lambda(j,a)<R\}+1;\\
\brk(R):=&\sharp\{(j,a)|\Lambda(j,a)\leq R\}.
\end{align}
\begin{remark}We list some facts.
\begin{itemize}
\item The dimension of the twisted sector of
$s=(e^{-2\pi i\frac{1}r},e^{2\pi i R})$ is
$$
D_s:=\sharp \Lambda^{-1}(R)-1=\brk(R)-\trk(R).
$$
\item The dimension of the twisted sector of $t$ is
 $$
 D_t:=\sharp\{j|\beta_j=r\}.
 $$
\end{itemize}
\end{remark}
\begin{theorem}\label{thm 5.10}
The virtual dimension of moduli space $\scr M_R$ is
\begin{align}\label{eq dimc-simplified}
\dim_\cplane (\scr M_R)= \brk(R)-1+ D_t.
\end{align}
\end{theorem}
The proof of this theorem is given in \S\ref{sec proof thm5.10}.

\subsubsection{Proof of Theorem \ref{thm c-dertmin-all}: (1) and (2)}\label{sec 5.1.5}

We assume Theorem \ref{thm 5.10}.

First, suppose that $\Lambda$ is injective, then $\brk=\trk$, which we denote by $rk$. Note that $rk\circ\Lambda$ gives a ranking for  pairs $\{(j,a)\}$
and for elements in $\scr R$. Hence
$$
[1,n]_\mathbb Z\times \mathbb Z_{\geq 0}
\xrightarrow{\Lambda} \scr R\xrightarrow{rk}\mathbb Z_+
$$
is bijective.

Now given $c\in \mathbb Z_{\geq 0}$ we explain the explicit formula for ${\bf R}$ (i.e, $R$) and $d$:
\begin{itemize}
\item since $D_s=0$, $d$ must be $0$;
\item choose $R=rk^{-1}(c+1)$.
In fact, $c$ satisfies the formula
$$
D_t+c=\dim_\cplane(\scr M_R)= rk(R)-1+D_t,
$$
where we use Theorem \ref{thm 5.10} for the last "=".
\end{itemize}
This finishes the proof of (1) and (2) when $\Lambda$ is injective.

Now suppose that $\Lambda$ is not injective. Define
$$
\tilde{\scr R}=\{(R, \ell)\in \scr R\times \mathbb Z|
0\leq \ell\leq D_s\}.
$$
Recall that
 $ D_s=\sharp\Lambda^{-1}(R)-1$.
One may define a one to one correspondence
$$
\psi: \tilde{\scr R}\to [1,n]_\mathbb Z\times \mathbb Z_{\geq 0}
$$
such that for each $R$, $\psi$ identifies $\{R\}\times [0,D_s]_\mathbb Z$ with $\Lambda^{-1}(R)$. Write $\tilde \Lambda=\psi^{-1}$.
Define
$$
\widetilde{rk}: \tilde{\scr R}\to \mathbb Z_+,\;\;\;
\widetilde{rk}(R,\ell)=\brk(R)-\ell.$$
Then one can see that
$$
[1,n]_{\mathbb Z}\times \mathbb Z_{\geq 0}\xrightarrow{\tilde\Lambda}
\tilde{\scr R}\xrightarrow{\widetilde{rk}} \mathbb Z_{+}
$$
is bijective. Now in Theorem \ref{thm c-dertmin-all}, $c$ determines
$R$  and $d$ by
\[
(R,d)= \widetilde{rk}^{-1}(c+1).
\]

\subsubsection{Proof of Theorem \ref{thm 5.10}}\label{sec proof thm5.10}

Take a pair $(j,a)$ such that
$$
\Lambda(j,a)=\frac{\beta_j+ar}{\alpha_jr}=R.
$$
Set $m=\beta_j+ar$. Then $\msf u[j,m]$ is in $\scr M_R$.
 $\msf u[j,m]$  maps $\msf q_\infty$ to $\langle s\rangle\ltimes
\{\msf p_j\}$, where
$$
s=\big(e^{-2\pi i\frac{1}{r}}, e^{2\pi iR}\big).
$$
The degree shifting of $s$ in $\overline{\E}_\wa$ is
(cf. \eqref{eq degree-shift} and the definition of $\tau(R,u)$ in \eqref{eq def-tau})
\begin{align}\label{eq degree-shift-2}
degsh(s)=\sum_{u\not=j}\left\{-\frac{\beta_u}{r}
+ \alpha_u R \right\} +\left\{R\right\}= \sum_{u}
\{\tau(R,u)\}+\left\{R\right\}.
\end{align}
Note that $\tau(R,j)= (m-\beta_j)/r=a$, consequently $\{\tau(R,j)\}=0$.

\begin{proposition}\label{prop dimc}
The virtual dimension of $\scr M_R$ is
\begin{align}\label{eq dimc}
\dim_\cplane(\scr M_R)= \sum_{u=1}^n [\tau(R,u)] + n-1+D_t,
\end{align}
where $D_t$ is the complex dimension of the twisted sector
of $t=e^{-2\pi i\frac{1}{r}}$.
\end{proposition}

\begin{proof}
By the formula of the virtual dimension for the relative moduli space, we have
$$
\dim_\cplane(\scr M_R)
=\frac{1+\sum_{u}\alpha_u}{r}\frac{m}{\alpha_j}+n-1
-\sum_{u}\left\{\frac{\beta_u}{r}\right\}
-degsh(s)- \left[R\right].
$$
Note that the first term can be simplified as
$$
\frac{1+\sum_{u}\alpha_u}{r}\frac{m}{\alpha_j}
=(1+\sum_{u} \alpha_u)\, R.
$$
By the fact that $1\leq \beta_u\leq r$
$$
\sum_{u}\left\{\frac{\beta_u}{r}\right\}
=\sum_{u}\frac{\beta_u}{r}-D_t.
$$
Then
$$
\dim_\cplane(\scr M_R)=\sum_{u}\left(\alpha_uR-\frac{\beta_u}{r}\right)
+n-1+D_t-degsh(s)+\left\{ R \right\}.
$$
Plug in the formula for $degsh(s)$, we can prove the claim.
\end{proof}
We have some easy consequences.

\begin{lemma}\label{lem 5.12}
\begin{enumerate}
\item $\dim_\cplane(\scr M_R)$ increases strictly  in $R$;
\item $\dim_\cplane(\scr M_{R+1})=\dim_\cplane (\scr M_R)+|\fk a|$,
where
$
|\fk a|=\sum_{u}\alpha_u.$
\end{enumerate}
\end{lemma}

\begin{proof}
(1) Suppose that $R_1<R_2$, by definition
$$
\tau(R_1,u)<\tau(R_2,u),
$$
hence, $\dim_\cplane(\scr M_{R_1})\leq\dim_\cplane(\scr M_{R_2})$. On the other hand, there exists
certain $j$ such that $\tau(R_2,j)$ is an integer, then
$[\tau(R_1,j)]<[\tau(R_2,j)]$. We conclude that $\dim_\cplane(\scr M_{R_1})<\dim_\cplane(\scr M_{R_2})$. This proves (1).

(2) follows from the simple observation that $\tau(R+1,u)=\alpha_u+\tau(R,u)$.
\end{proof}
Set
$$
\scr R_k=\{R\in \scr R| k<  R \leq k+1\},
\;\;\;\;\;\tilde{\scr R}_k=\{(R,\ell)\in \tilde{\scr R}|
R\in \scr R_k\}.$$
\begin{lemma}\label{lem 5.13}
We have following formulae.
\begin{enumerate}
\item $|\tilde{\scr R}_k|=|\fk a|$;
\item for $R\in \scr R_k$,
$$
k|\fk a|\leq \dim (\scr M_R)-D_t<(k+1)|\fk a|.
$$
\item If $\Lambda$ is injective,
$$
\dim(\scr M_R)-D_t=rk(R)-1.
$$
\end{enumerate}
\end{lemma}
\begin{proof}
Note that
$$
\tilde\Lambda: \Lambda^{-1}(\scr R_k)\to \tilde{\scr R}_k
$$
is a bijection. On the other hand,
$(j,a)\in \Lambda^{-1}(\scr R_k)$ if and only if $k\alpha_j\leq a \leq (k+1)\alpha_j-1.$
Therefore, for each fixed $j$ there are $\alpha_j$ pairs, i.e.
$(j,k\alpha_j),\ldots,(j,(k+1)\alpha_j-1)\in \Lambda^{-1}(\scr R_k)$.
 This implies that the size of $\tilde{\scr R}_k$ is $|\fk a|$. We proved (1).

Now we prove (2).
We first show that
$$
0\leq \dim_\cplane(\scr M_R)-D_t< |\fk a|,\qq \forall\, R\in \scr R_0.
$$
First of all $\Lambda(j,0)=\frac{\beta_j}{\alpha_jr}\in\scr R_0$
for $1\leq j\leq n$. Suppose
$$
\frac{\beta_1}{\alpha_1r}=\ldots= \frac{\beta_m}{\alpha_mr}<
\frac{\beta_{m+1}}{\alpha_{m+1}r}\leq \ldots\leq \frac{\beta_n}{\alpha_nr}.
$$
Set $R_{min}=\frac{\beta_1}{\alpha_1r}\in\scr R_0$. Then $\trk(R_{min})=1$ and $\brk(R_{min})=m$. Hence
\begin{eqnarray*}
\tau(R_{min},u)=0, &\mbox{ when } &1\leq u\leq m;
\end{eqnarray*}
and
\begin{eqnarray*}
-1< \tau(R_{min},u)<0, &\mbox{ when } & m+1\leq u\leq n.
\end{eqnarray*}
By the formula \eqref{eq dimc}, we have
\begin{align}\label{E dim-M-R-min}
\dim_\cplane(\scr M_{R_{min}})-D_t&= m-1=\brk(R_{min})-1 \\
&\leq n-1<|\wa|\nonumber.
\end{align}
By (2) of Lemma \ref{lem 5.12},
$$
\dim_\cplane(\scr M_{R_{min}+1})-D_t=m-1+|\fk a|.
$$
Note that for any $R\in \scr R_0$ we have
\[
R<R_{min}+1.
\]
Therefore for $1\leq u\leq n$ we have
\[
\tau(R,u)<\tau(R_{min}+1,u).
\]
In particular for $1\leq u\leq m$ since $\tau(R_{min}+1,u)=\alpha_u$ we have
$$
[\tau(R,u)]\leq [\tau(R_{min}+1,u)]-1.
$$
We conclude that for any $R\in\scr R_0$
$$
\dim_\cplane(\scr M_R)-D_t\leq \dim_\cplane(\scr M_{R_{min}+1})-D_t-m
=|\wa|-1.
$$
For any $R$ in any $\scr R_k$, we may apply (2) of Lemma \ref{lem 5.12}
and induction on $[R]$ to complete the proof.

(3) is an easy consequence of (1) of Lemma \ref{lem 5.12} and (2).
In fact, $\dim_\cplane(\scr M_R)-D_t$ increases strictly in $R$ and maps $\scr R_0$ to $[0,|\fk a|-1]\cap\integer=\{0,1,2,\ldots,|\fk a|-1\}$,
but both sets have $|\fk a|$-elements. Then the map has no other choice.
Moreover, under the hypothesis that $\Lambda$ is injective
we know that $\trk=\brk=rk$ (cf. \S\ref{sec 5.1.5}) and $\dimc \scr M_{R_{min}}=0$.
\end{proof}
\begin{lemma}\label{lem 5.14}
We have $\brk(R+1)=\brk(R)+|\wa|$.
\end{lemma}
\begin{proof}
In fact $\Lambda(j,a+\alpha_j)=\Lambda(j,a)+1$ implies that the map $\scr R_k\rto\scr R_{k+1},\, R\mapsto R+1$ is bijective.
Then for $R\in\scr R_k$ we have
\[\begin{split}
\brk(R+1)
=&\brk(R)+\sharp\{R'+1 |R'<R, R'\in \scr R_k\}+\sharp\{R'\in\scr R_k|R'>R\}\\
=&\brk(R)+|\scr R_k|\\
=&\brk(R)+|\wa|.
\end{split}\]
\end{proof}
We next come to the proof of Theorem \ref{thm 5.10}.
\begin{proof}[Proof of Theorem \ref{thm 5.10}]
Suppose that $R_1<R_2$ are two elements in $\scr R$. We claim that
\begin{align}\label{eq dim<dim-|lambda|}
\dim_\cplane(\scr M_{R_1})\leq \dim_\cplane(\scr M_{R_2})-\sharp\Lambda^{-1}(R_2).
\end{align}
In fact, suppose that
$$
\Lambda^{-1}(R_2)=\{(j_1,a_1),\ldots, (j_m,a_m)\}.
$$
Then
$[\tau(R_1,u)]\leq [\tau(R_2,u)]$ and for $u=j_1,\ldots,j_m$
$$
[\tau(R_1,u)]\leq [\tau(R_2,u)]-1.
$$
Then the claim \eqref{eq dim<dim-|lambda|} follows from the definition of $\tau(R,u)$ in the formula \eqref{eq degree-shift-2} (see also \eqref{eq def-tau}).

Now suppose that
$$
\scr R_0=\{R_{min}=R_1<R_2<\ldots<R_{|\wa|}\}.
$$
We claim that formula \eqref{eq dimc-simplified} is true for $R\in \scr R_0$, i.e,
$$
\dim_\cplane(\scr M_{R_k})-D_t=\brk(R_k)-1,\,\,\,\, \forall\, R_k\in\scr R_0.
$$
This consists of several steps.

\v\n
{\em Step (i)} . Formula \eqref{eq dimc-simplified} is true for $R_{min}=R_1$. This is already proved in the proof of Lemma \ref{lem 5.13}, see \eqref{E dim-M-R-min}. 
\v\n
{\em Step(ii)}. For any $R_k\in \scr R_0, 2\leq k\leq |\wa|$ we have
$$
\dim_\cplane(\scr M_{R_k})-D_t\geq \brk(R_k)-1.
$$
\n
{\em Proof. }We show for $k=2$.   By \eqref{eq dim<dim-|lambda|} and {\em Step (i)},
\begin{eqnarray*}
\dim_\cplane(\scr M_{R_2})-D_t &\geq& \dim_\cplane
(\scr M_{R_{min}})-D_t+|\Lambda^{-1}(R_2)|\\
&=& \brk(R_{min})-1+|\Lambda^{-1}(R_2)|\\
&=&|\Lambda^{-1}\{R_{min},R_2\}|-1\\
&=&\brk(R_2)-1.
\end{eqnarray*}
The argument can be used inductively to derive the claim for any $3\leq k\leq |\wa|$. \qed 

\v\n
{\em Step (iii)}. For any $R_k\in \scr R_0, 2\leq k\leq |\wa|$ we have
$$
\dim_\cplane(\scr M_{R_k})-D_t\leq \brk(R_k)-1.
$$
\n
{\em Proof. }By (2) of Lemma \ref{lem 5.13}, we have already proved that
$$
\dim_\cplane (\scr M_{R_{|\wa|}})-D_t\leq |\wa|-1=\brk(R_{|\wa|})-1.
$$
This verifies the claim for $k=|\wa|$. By \eqref{eq dim<dim-|lambda|}, we show the claim is true for $k=|\wa|-1$. In fact
\begin{eqnarray*}
\dim_\cplane(\scr M_{R_{|\wa|-1}})-D_t
&\leq& \dim_\cplane(\scr M_{R_{|\wa|}})-D_t-|\Lambda^{-1}(R_{|\wa|})|\\
&\leq& \brk(R_{|\wa|})-1-|\Lambda^{-1}(R_{|\wa|})|\\
&=&|\Lambda^{-1}(\scr R_0\setminus\{R_{|\wa|}\})|-1\\
&=&\brk(R_{|\wa|-1})-1.
\end{eqnarray*}
The argument can be used inductively to derive the claim for all $2\leq k\leq |\wa|$. \qed 
\v
Combining these three steps, we show that formula \eqref{eq dimc-simplified} is true for $R\in \scr R_0$. Then by (2) of Lemma \ref{lem 5.12} and Lemma \ref{lem 5.14} the formula \eqref{eq dimc-simplified} is easy to derive for all $R\in \scr R$.
\end{proof}

\subsection{Computation of invariants}\label{sec 5.2}
In this subsection, we compute the relative invariant
$$
\langle{\bf R}(\tau_c(\Theta_{(t)})|H_{(s)}^d)\rangle
=\int_{\scr M_R}\msf{ev}_0^\ast\Theta_{(t)}\cup \psi^c\cup
\msf{rev}_\infty^\ast H_{(s)}^d
$$
and prove (3) in Theorem \ref{thm c-dertmin-all}, i.e.
\begin{theorem}
$\langle{\bf R}(\tau_c(\Theta_{(t)})|H_{(s)}^d)\rangle$ is non-zero.
\end{theorem}
Let $\msf p_o$ be the origin of $\mathbb Z_r^\fk b\ltimes \cplane^n$.
Let ${\scr M}_R^\circ$ consist of morphisms that maps
$\msf q_0$ to $\msf p_o$. Recall that
$$
\scr M_R=\frac{\Hol^0_R(\msf C_o, \mathbb Z_r^\fk b\ltimes\cplane^n)}{\cplane_o^\ast}.
$$
Similarly, set $(\Hol^\circ)^0_R(\msf C_o,\mathbb Z_r^\fk b\ltimes\cplane^n)$ to consists of
morphisms that maps $\msf q_0$ to $\msf p_o$. Then
$$
\scr M_R^\circ=\frac{(\Hol^\circ)^0_R(\msf C_o,\mathbb Z_r^\fk b\ltimes\cplane^n)}{\cplane_o^\ast}.
$$
\begin{lemma}
$\langle{\bf R}(\tau_c(\Theta_{(t)})|H_{(s)}^d)\rangle$ can be reduced to be
\begin{equation}\label{eq H}
\scr H:=\int_{\scr M^\circ_R}\psi^c\cup\msf{rev}_\infty^\ast H_{(s)}^d.
\end{equation}
\end{lemma}
\begin{proof} The twisted sector of $t$ is
$$
\mathbb Z_r\ltimes(\cplane^n)^{\zeta_r},
$$
where the action is trivial.
 The evaluation map is
$$
ev_0: \Hol^0_R(\msf C_0,\mathbb Z_r^\fk b\ltimes\cplane^n)\rto (\cplane^n)^{\zeta_r},\;\;\;
\msf u\mapsto u^0(\msf q_0).
$$
Since $(\Hol^\circ)^0_R$ is regular, this is a trivial fibration over the neighborhood of $\msf p_o$. Namely,
$$
\Hol^0_R(\msf C_0,\mathbb Z_r^\fk b\ltimes\cplane^n)|_{U}
\cong U\times (\Hol^\circ)^0_R(\msf C_0,\mathbb Z_r^\fk b\ltimes\cplane^n),
$$
where $U$ is a small neighborhood of $\msf p_o$ in $(\cplane^n)^{\zeta_r}$. Therefore
$$
\scr M_R|_U\cong U\times \scr M^\circ_R.
$$
Since
$
\int_U\Theta_{(t)}=1,
$
we have
$$
\int_{\scr M_R}\msf{ev}_0^\ast\Theta_{(t)}\cup \psi^c\cup
\msf{rev}_\infty^\ast H_{(s)}^d=\int_{\scr M^\circ_R}\psi^c\cup
\msf{rev}_\infty^\ast H_{(s)}^d.
$$
This proves \eqref{eq H}.
\end{proof}

In this rest of subsection, we compute the integration
 \eqref{eq H}.

\subsubsection{Contact orders at $\msf p_0\in\msf P^\wb_\wa$ of stable maps}

Suppose
$
rR=\frac{q}{p}$, where  $(p,q)=1$ (cf. Lemma \ref{lem 5.6}). We introduce
{\bf contact orders} at $\msf p_0\in\msf P^\wb_\wa$ (relative to $\msf q_0\in \msf S^2$) of morphisms in $\scr M^\circ_R$.

Let
$$
\msf u\in (\Hol^\circ)^0_R(\msf C_o,\mathbb Z_r^\fk b\ltimes\cplane^n).
$$
Suppose that $u^0: \cplane_0\rto \cplane^n$ is
$$
u^0(x)= (u^0_1(x),\ldots, u^0_n(x)).
$$
Roughly speaking, we define the $j$-th contact order of $\msf u$ at $\msf p_0$
to be the vanishing order of $u^0_j(x)$ at $x=0$.

Let us first consider the case that $u^0_j(x)\not\equiv 0$. $u^0_j$ defines a morphism
$$
\msf u_j:[\overline{\cplane}]^1_p\rto[\overline{\cplane_j}]
^{\beta_j}
_{\alpha_j}.
$$
The $j$-th contact order of $\msf u$ at $\msf p_o$ is defined to be the contact order of $\msf u_j$ at $\msf p_0$. We denote it by $ord_j(\msf u,\msf p_0)$.
To be explicit,
$$
\msf u_j[x,y]=
\left[
\sum_{\ell=0}^{[\tau(R,j)]} c_\ell y^{q\alpha_j-p(\beta_j+\ell r)}
x^{\beta_j+\ell r},
y^q\right],
$$
or
$$
u^0_j(x)=\sum_{\ell=0}^{[\tau(R,j)]} c_\ell x^{\beta_j+\ell r}.
$$
\begin{defn}
For $\msf u$ given as above, we define the
$j$-th contact order of $\msf u$ at $\msf p_0$, denoted by
$ord_{j}(\msf u,\msf p_0)$, to be:
 \begin{itemize}
 \item if $\msf u_j(x)\equiv0$
$$
ord_j(\msf u,\msf p_0):=\frac{\beta_j}{r}+[\tau(R,j)]+1;
$$
\item if $\msf u_j(x)\not\equiv0$
$$
ord_j(\msf u,\msf p_0):=\frac{\beta_j}{r}+k,
$$
when $c_k\not=0$ and $c_i=0$ for $i<k$.
\end{itemize}
\end{defn}
\begin{remark}
For each $j$, the $j$-th contact order $c_j$ satisfies
$$
c_{\min,j}:=\frac{\beta_j}{r}\leq c_j
\leq \frac{\beta_j}{r}+[\tau(R,j)]+1=:c_{\max,j}+1
$$
In particular, when $ord_j(\msf u)=c_{\max,j}+1$, $\msf u_j\equiv0$.
\end{remark}
Since we assume that $\msf u$ maps $\msf q_0$ to $\msf p_0$, we find that
$$
ord_j(\msf u,\msf p_0)\geq \frac{\beta_{j}}{r}>0.
$$
Given $\fk c=(c_1,\ldots,c_n)$ we define
$$
\scr M_{R}(\fk c)=\{
[\msf u]\in\scr M_R^\circ\,|\,
ord_j(\msf u,\msf p_0)\geq c_j, 1\leq j\leq n\}.
$$
(Here, we require that $c_u-\frac{\beta_u}{r}$ is a nonnegative integer).
Let
$$
\fk c_{\min}=(c_{\min,1},\ldots,c_{\min,n}).
$$
By the definitions, the following lemma is obvious.
\begin{lemma}
$
\scr M_R(c_{\min})=\scr M_R^\circ.$
\end{lemma}
On the other hand, suppose that
$$
\Lambda^{-1}(R)=\{(j_1,a_1),\ldots, (j_m,a_m)\}.
$$
Let $J=\{j_1,\ldots,j_m\}$.
Define $\fk c_{\max}=(c_1,\ldots, c_n)$ to be
\begin{equation}
c_u=\left\{
\begin{array}{lll}
c_{\max,u}+1, & \mbox{when} & u\not\in J,\\
c_{\max,u}, &\mbox{when} & u\in J.
\end{array}
\right.
\end{equation}
Then
$
\scr M_R(\fk c_{\max})$ consists of morphisms similar to
$\msf u[j_k,\beta_{j_k}+a_kr]$. To be precise, the projection of $\msf u$
 to $[\overline{\cplane_{j_k}}]^ {\beta_{j_k}}_{\alpha_{j_k}}$ is either trivial or
$\msf u[j_k,\beta_{j_k}+a_kr]$.

\subsubsection{Reduce the integration on $\scr M_{R}(\fk c_{\min})$ to $\scr M_R(\fk c_{\max})$}

For $\fk c=(c_1,\ldots,c_n)$, define
$$
\fk c_j=(c_1,\ldots,c_{j-1},c_j+1,c_{j+1}\ldots, c_n).
$$
When $c_j\leq c_{\max,j},$
$
\scr M_{R}(\fk c_j)
\subseteq \scr M_{R}(\fk c)
$
is of codimension 1. As in \cite{HLR08}, we interpret
$\scr M_{R}(\fk c_j)$ as a zero locus of a transversal
section $\sigma$ of a line bundle, denoted by $\scr L_{\fk c_j}$, over $\scr M_R(\fk c)$. We separate the construction into several steps.

\vskip 0.1in
\noindent
{\em Step 1. line bundle $\tilde{\scr L}^0_{\fk c_j}$ over $\tilde{\scr M}^0_R(\fk c):=\Hol^0_R(\fk c)$.}
\vskip 0.1in
Let $\msf u\in \tilde{\scr M}^0_R(\fk c)$. Suppose
its restriction on $\msf C_0$ is
$$
\msf u=(u^0,u^1):\mathbb Z_r\ltimes \cplane_o
\rto\mathbb Z_r\ltimes\cplane^n,
$$
where $u^0(0)=0$.
Let us focus on
$
u_j^0: \cplane_o\to \cplane_j.$ Suppose
$$
c_j=\frac{\beta_j}{r}+k.
$$
Then
$$
u_j^0(s)=a_{j,k}x^{\beta_j+kr}+a_{j,k+1}x^{\beta_j+(k+1)r}+\ldots
$$
and $a_{j,k}\in (\check\cplane_o)^{\otimes (\beta_j+kr)}$.
Define $\tilde{\scr L}^0_{\fk c_j}$ such that
$$
\tilde{\scr L}^0_{\fk c_j}|_{\msf u}=(\check\cplane_o)^{\otimes (\beta_j+kr)}.
$$
Mean while,
$\tilde\sigma^0(\msf u)=a_{j,k}$ defines a section of this line bundle such that
$$
(\tilde\sigma^0)^{-1}(0)=\tilde{\scr M}^0_R(\fk c_j).
$$
Obviously,  $\tilde\sigma^0$ is transversal on $\tilde{\scr M}^0_{R}(\fk c_j)$ since $a_{j,k}$ can be arbitrary.

\vskip 0.1in
\noindent
{\em Step 2. Line bundle ${\scr L}_{\fk c_j}$ over ${\scr M}_R(\fk c)=\cplane^\ast_o\ltimes\tilde{\scr M}_R(\fk c)$.}
\vskip 0.1in
$\cplane_o^\ast$-action on $\tilde{\scr M}_R^0(\fk c)$
naturally acts on $\tilde{\scr L}_{\fk c_j}^0$ and
$\tilde\sigma^0$ is a $\cplane^\ast_o$-equivariant section. This leads
to a line bundle
$$
 {\scr L}_{\fk c_j}=\cplane^\ast_o\ltimes\tilde{\scr L}^0_{\fk c_j}\rto \cplane^\ast_o\ltimes\tilde{\scr M}_R(\fk c)=\scr M_R(\fk c)
$$
and a transversal section $\sigma$ such that
$$
\sigma^{-1}(0)=\scr M_R(\fk c_j).
$$
We conclude that
\def \om{\scr M}
\begin{lemma}
$\om_R(\fk c_j)\subseteq \om_R(\fk c)$ is
the Euler class of orbifold line bundle ${\scr L}_{\fk c_j}$. Hence
$$
\int_{\om_R(\fk c_j)}=\int_{\om_R(\fk c)}c_1({\scr L}_{\fk c_j}).
$$
\end{lemma}
\n
{\em Step 3. ${\scr L}_{\fk c_j}$ vs $\scr L$.}
\vskip 0.1in
By definitions, we have
$$
{\scr L}_{\fk c_j}=\tilde{\scr L}^{\otimes rc_j}|_{\scr M_R(\fk c)}={\scr L}^{\otimes c_j}|_{\scr M_R(\fk c)}$$
Therefore,
$$
\int_{\om_R(\fk c)}\psi
=\frac{1}{c_j}\int_{\om_R(\fk c)}c_1({\scr L}_{\fk c_j})
=\frac{1}{c_j}\int_{\om_R(\fk c_j)}.
$$
\begin{proposition}
Suppose
$$
\Lambda^{-1}(R)=\{(j_1,a_1),\ldots,(j_m,a_m)\}.
$$
Let $J=\{j_1,\ldots,j_m\}$.
The integration $\scr H$ can be reduced as
$$
\scr H
=\prod_{\ell=1}^n\frac{1}{c_{\max,\ell}!}\prod_{\ell\in J}
c_{\max,\ell}\int_{\om_{R}(\fk c_{\max})}
\psi^{|J|-1-d}\msf{rev}^\ast_\infty H^d.
$$
Here
$$
c_{\max,\ell}!=\prod_{k=0}^{[\tau(R,\ell)]}(c_{\max,\ell}-k).
$$
In fact, when $c_{\max,\ell}$ is an integer, it coincides the standard definition, when $c_{\max,\ell}$ is not an integer,
$$
c_{\max,\ell}!=c_{\max,\ell}(c_{\max,\ell}-1)\ldots(c_{\max,\ell}-[c_{\max,\ell}]).
$$
\end{proposition}
\begin{proof}
For each $\ell\not\in J$, from $\fk c_{\min}$ to $\fk
c_{\max}$ we repeat the reduction $[\tau(R,\ell)]+1$-steps. It exhausts $[\tau(R,\ell)]+1$'s
$\psi$ and contributes
$c_{\max}!$.

For $\ell\in J$, from $\fk c_{\min}$ to $\fk
c_{\max}$ we repeat the reduction $[\tau(R,\ell)]$-steps. It exhausts $[\tau(R,\ell)]$'s
$\psi$ and contributes
$(c_{\max}-1)!$.

The number of $\psi$ left is
\begin{align*}
&c-\sum_{\ell}([\tau(R,\ell)]+1)+|J|\\
=&(\dim_\cplane(\scr M_R)
-D_t-d)-\sum_{\ell}[\tau(R,\ell)]-n+|J|\\
=&|J|-1-d.
\end{align*}

\end{proof}

\subsubsection{Computation of invariants}\label{sec comput-inv}
\begin{theorem}\label{thm H}
The formula for $\scr H$ is
$$
\scr H=\frac{1}{r}\, R^d\prod_{\ell=1}^n\frac{1}{c_{\max}!}
=\frac{1}{r}\, R^d\prod_{\ell=1}^n\frac{1}{
(\frac{\beta_\ell}{r}+[\tau(R,\ell)])!}.
$$
\end{theorem}
This is equivalent to show that
\begin{equation}
\scr H':=\prod_{\ell\in J}
c_{\max,\ell}\int_{\scr M_{R}(\fk c_{\max})}
\psi^{|J|-1-d}\msf{rev}^\ast_\infty H^d=\frac{1}{r}\,R^d.
\end{equation}

For simplicity, we consider the two cases: $d=0$  and $1\leq d\leq \sharp\Lambda\inv(R)=\sharp J$ separately.
\vskip 0.1in
\noindent
{\em Case 1: $d=0$. }
\vskip 0.1in
For simplicity, assume that $(1,a)\in \Lambda^{-1}(R)$. Set
$$
\fk c=(c_{\max,1}, c_{\max,2}+1,\ldots,c_{\max,n}+1).
$$
Then
$$
\scr H'=c_{\max,1}\int_{\scr M_R(\fk c)}1.
$$
Now
$\scr M_R(\fk c)$ consists of only one element: the standard map $\msf u[1, a]$.  
We know that
$$
\int_{\scr M_R(\fk c)} 1= \frac{1}{\beta_1+ar}=\frac{1}{r}\frac{1}{c_{\max,1}}.
$$
Hence, combine the computation together we have $\scr H'=1/r$.

\vskip 0.1in\noindent
{\em Case 2: $d\neq0$.}
\vskip 0.1in
We now compute
\begin{equation}
\scr H''=\int_{\scr M_{R}(\fk c_{\max})} \psi^{|J|-1-d}\msf{rev}_\infty^\ast H^d.
\end{equation}
We apply the localization technique. For simplicity, assume that
$$
\Lambda^{-1}(R)=\{(1,a_1)\ldots,(m,a_m)\}.
$$
 The fixed point are $\msf u[k, \beta_k+a_kr], 1\leq k\leq m$.
For each $k$ we have the contributions to the localization formula:
\begin{itemize}
\item the action weights on the normal direction corresponding to $\msf p_j, 1\leq j\leq m, j\neq k$ are:
$$
(\lambda_j-\lambda_0\alpha_j)-(\lambda_k-\lambda_0\alpha_k)\frac{\beta_j+a_jr}{\beta_k+a_kr}
=(\lambda_j-\lambda_0\alpha_j)-(\lambda_k-\lambda_0\alpha_k)\frac{\alpha_j}{\alpha_k}
=\lambda_j-\lambda_k\frac{\alpha_j}{\alpha_k}.
$$
\item contribution from $\psi$:
$$
\frac{r}{\beta_k+a_kr}(\lambda_k-\alpha_k\lambda_0),
$$
\item contribution from $H$: $\lambda_k/\alpha_k$.
\item automorphism group: $\frac{1}{\beta_k+a_kr}$.
\end{itemize}
Then we have
$$
r\scr H''=\sum_{k=1}^m\left(\frac{r}{\beta_k+a_kr}\right)^{m-d}\cdot \frac{(\lambda_k-\alpha_k\lambda_0)^{m-1-d}
\left(\frac{\lambda_k}{\alpha_k}\right)^d }{\prod_{j\not=k}(\lambda_j-\lambda_k\frac{\alpha_j}{\alpha_k})  }.
$$
By simplification, and set $\lambda'_j= \lambda_j/\alpha_j$ we have
$$
r\scr H''=\frac{1}{\alpha_1\ldots\alpha_m}\sum_{k=1}^m\left(\frac{r \alpha_k}{\beta_k+a_kr}\right)^{m-d}
\cdot \frac{(\lambda'_k-\lambda_0)^{m-1-d}(\lambda_k')^d}
{\prod_{j\not=k}(\lambda'_j-\lambda'_k)  }.
$$
Note that
$$
\frac{r \alpha_k}{\beta_k+a_kr}=\frac{1}{R}.
$$
By using the fact that
$$
\sum_{k=1}^m\frac{(\lambda'_k-\lambda_0)^{m-1-d}(\lambda_k')^d}
{\prod_{j\not=k}(\lambda'_j-\lambda'_k)  }=\sum_{k=1}^m\frac{(\lambda_k')^{m-1}}
{\prod_{j\not=k}(\lambda'_j-\lambda'_k)  }=1,
$$
which is a consequence of
\[
\sum_{k=1}^m\frac{(\lambda_k')^d}
{\prod_{j\not=k}(\lambda'_j-\lambda'_k)  }=0,\,\,\,\forall\, 0\leq d\leq m-2
\]
and the expansion of Vandermonde determinant,
we have
$$
r\scr H''= \frac{1}{\alpha_1\ldots\alpha_m}\left(\frac{1}{R}\right)^{|J|-d}.
$$
Then
$$
\scr H'=\frac{1}{r}\prod_{\ell\in J}\frac{c_{\max,\ell}}{\alpha_\ell}\left(\frac{1}{R}\right)^{|J|-d}=\frac{1}{r}\, R^d,
$$
here we use the fact that
$$
\frac{c_{\max,\ell}}{\alpha_\ell}=R,\,\,\,\,\,
\forall \ell\in J.
$$

\begin{remark}
This completes the proof of (3) in Theorem \ref{thm c-dertmin-all}. In fact we show that
$$
\langle{\bf R}(\tau_c(\Theta_{(t)})|H_{(s)}^d)\rangle
=\frac{1}{r}\, R^d\prod_{\ell=1}^n\frac{1}{
(\frac{\beta_\ell}{r}+[\tau(R,\ell)])!}.
$$
We would like to emphasize that $R$ and $d$ are determined by $c$.
\end{remark}

\subsection{Special case: $\msf E=\Gamma^\mu\ltimes \cplane^n$}\label{sec 5.3}
Suppose that $t$ in ${\bf R}$ is $(\gamma), \gamma\in \Gamma$ with order $r$. Consider a morphism
$$\msf u\in \scr M_{\bf R}([\overline{\cplane^n}]^\mu_\fk a|
\msf{P}^\mu_\fk a).$$
Then for
$
\msf u: \msf C_0\rto \Gamma^\mu\ltimes \cplane^n,
$
\begin{equation}\label{eqn_u}
\pi\circ\msf u: \msf q_{0}\rto\langle \gamma\rangle\ltimes\{O\}
\end{equation}
Suppose that $\gamma$ acts on $\cplane^n$ as weight $\fk b=(\beta_1,\ldots,\beta_n)$.
We find that $\msf u|_{\msf C_0}$ is a composition of the following morphisms:
$$
\msf u:\msf C_0\xrightarrow{\msf v}\mathbb Z_r^\fk b\ltimes\cplane^n\hookrightarrow\Gamma^\mu\ltimes\cplane^n
$$
where we identify $\langle\gamma\rangle$ with $\mathbb Z_r$ via $\gamma\leftrightarrow \zeta_r$, and therefore
$$
\msf u:\msf S^2\xrightarrow{\bar{\msf v}} \oC
\hookrightarrow[\overline{\cplane^n}]^\mu_\fk a.
$$
Repeat the arguments in \S\ref{sec MR}, we get the following facts
\begin{eqnarray*}
\msf{Hol}_{\bf R}(\msf S^2,[\overline{\cplane^n}]^\mu_\fk a)
&\cong& \msf{Hol}_{{\bf R}}(\msf C_0,\Gamma^\mu\ltimes\cplane^n),\\
\msf{Hol}_{{\bf R}}(\msf C_0,\Gamma^\mu\ltimes\cplane^n)
&\cong&C_{\Gamma}(\gamma)
\ltimes\Hol_{\bf R}^0(\msf C_0,\langle\gamma\rangle^\fk b\ltimes\cplane^n).
\end{eqnarray*}
Next, note that we have a central extension
$$
1\rto \langle\gamma\rangle\cong \mathbb Z_r
\rto C_\Gamma(\gamma)\rto C'_\Gamma(\gamma)\rto 1,
$$
and
$$
1\rto \langle\gamma\rangle\cong\mathbb Z_r\rto \cplane^\ast_o\rto \cplane^\ast\rto 1.
$$
The first extension then canonically induces the following extension
$$
1\rto \cplane_o^\ast\rto \cplane_o^\ast\times_{\langle\gamma\rangle}C_\Gamma(\gamma)
\rto C'_\Gamma(\gamma)\rto 1.
$$
We conclude that
\begin{equation}
\scr M_{\bf R}(\msf S^2,[\overline{\cplane^n}]^\mu_\fk a)=
\frac{\Hol_{\bf R}(\msf C_0, \cplane^n\rtimes \Gamma^\mu)}
{\text{Aut}(\msf S^2\slash\{\msf q_0,\msf q_\infty\})}
\cong
\frac{\Hol^0_{\bf R}(\msf C_0, \cplane^n\rtimes \mathbb Z_r^\fk b)}{
\cplane_o^\ast\times_{\langle\gamma\rangle} C_\Gamma(\gamma)}
.
\end{equation}
We conclude that
\begin{proposition}\label{prop MR}
$\scr M_{\bf R}([\overline{\cplane^n}]^\mu_\fk a)$
is a fibration over $BH$ with an orbifold fiber $\scr M_{\bf R}(\oC)$, where $H=C'_\Gamma(\gamma)=C_\Gamma(\gamma)/\langle\gamma\rangle$.
\end{proposition}
\begin{proof}
In fact, we have
$$
\frac{\Hol^0_{\bf R}(\msf C_0, \cplane^n\rtimes \mathbb Z_r^\fk b)}{
\cplane_o^\ast\times_{\langle\gamma\rangle} C_\Gamma(\gamma)}\cong \frac{\Hol^0_{\bf R}(\msf C_0, \cplane^n\rtimes \mathbb Z_r^\fk b)/\cplane_o^\ast}{
 C'_\Gamma(\gamma)}\cong \frac{\scr M_{\bf R}(\oC)}{C'_\Gamma(\gamma)}.
$$
\end{proof}

\subsection{General cases: $\E\rto\msf B$}\label{sec 5.4}
\subsubsection{Moduli space}
Let $\msf u\in \Hol_{\bf R}(\msf S^2,\overline{\msf E}_\fk a)$. Suppose that
$$
\pi\circ\msf u: \msf q_0\rto\langle\gamma\rangle\ltimes \{x\}
$$
where $\gamma\in \Gamma=G_x$. Let $\msf B(t)$
be the twisted sector containg $\langle\gamma\rangle\ltimes \{x\}$.
Let $
\msf U:=\Gamma\ltimes U_x$ be
a local orbifold coordinate chart for $\msf B$.
Bundle $\msf E$ locally is
$$
\msf E|_{\msf U}\cong \Gamma\ltimes (U_x\times \cplane^n).
$$
 Let $\mu$ be the representation of $\Gamma$ action on $\cplane^n$.
Then the morphism $\msf u|_{\msf C_0}$ is a composition of the following morphisms:
$$
\msf u:\msf C_0\xrightarrow{\msf v}
\{x\}\times (\langle \gamma\rangle^\fk b\ltimes \cplane^n )\rightarrow\{x\}\times (\Gamma^\mu\ltimes\cplane^n ),
$$
where  $\fk b$
is the action weight of $\gamma$ on $\cplane^n$. Hence, we have
\begin{equation}
\msf u:\msf S^2\xrightarrow{\bar{\msf v}} \{x\}\times [\overline{\cplane^n}]^\fk b_\fk a \rightarrow\{x\}\times [\overline{\cplane^n}]^\mu_\fk a
\hookrightarrow{\overline{\msf E}_\fk a}.
\end{equation}
Apply the similar arguments as before, we have
\begin{equation}
\msf{Hol}_{\bf R}(\msf S^2,\overline{\msf E}_\fk a)
\cong \msf{Hol}_{\bf R}(\msf C_0, \msf E).
\end{equation}
Let $\msf B(t)=(B^1(t)\rightrightarrows B^0(t))$ be the corresponding twisted sector.
\begin{lemma}
Let $\msf{Hol}^0_{{\bf R}}(\msf C_0,\msf E)$ be as above.
Then
\begin{enumerate}
\item $\Hol^0_{\bf R}(\msf C_0,{\msf E})$
is a fibration
$$
\phi^0:\Hol^0_{\bf R}(\msf C_0,\msf E)
\rto B^0(t)
$$
whose fiber is $\Hol^0_{{\bf R'}}(\msf C_0, \mathbb Z^\fk b_r\ltimes \cplane^n)$, here ${\bf R'}$ is naturally determined by ${\bf R}$;
\item $\msf B(t)$ acts on $\Hol^0_{\bf R}(\msf C_0,\msf E)$ naturally, whose based map is $\phi^0$ ;
\item $\msf{Hol}_{{\bf R}}(\msf C_0,\msf E)
=\msf B(t)\ltimes\Hol^0_{\bf R}(\msf C_0,\msf E)$.
\end{enumerate}
\end{lemma}
\begin{proof}
Define $\phi^0$ to be
$$
\msf u\mapsto \pi\circ\msf u|_{\msf q_0}\in  B^0(t),
$$
then for $\langle \gamma\rangle\ltimes\{x\}\in B^0(t)$,
by Proposition \ref{prop MR} we know that
 its preimage is $\Hol^0_{\bf R'}(\msf S^2,[\overline{\cplane^n}]^\fk b_\fk a)$. This shows (1).

(2) and (3) are basically followed by the definitions. We skip the proof.
\end{proof}
\begin{remark}
Let $\msf B(t)$ be a twisted sector. $B^0(t)$
 consists of elements in the format $\langle \gamma\rangle\ltimes\{x\}$. Such an element
 $\langle \gamma\rangle\ltimes\{x\}$ carries
 a canonical group of arrows $\langle\gamma\rangle\subseteq B^1(t)$.
Denote the group by $\ker(\langle \gamma\rangle\ltimes\{x\})$. It forms a trivial
$\mathbb Z_r$-bundle $\ker(t)\rto B^0(t)$. $\ker(t)\subseteq B^1(t)$
is a canonical kernel of $\msf B(t)$. Let
$$
\msf B'(t)=\left(\frac{B^1(t)}{\ker(t)}\rightrightarrows
B^0(t)\right).
$$
It is known that $\msf B(t)$ is a $\mathbb Z_r$-gerbe over
$\msf B'(t)$.
\end{remark}

\begin{proposition}\label{prp moduli-E-B}
$\scr M_{\bf R}(\overline{\msf E}_\fk a|\msf{PE}_\fk a)$ is a fibration over $\msf B'(t)$ with fiber
$\scr M_{\bf R'}(\overline{[\cplane^n]}^\fk b_\fk a|
\msf P^\fk b_\fk a)$.
\end{proposition}
\begin{proof}
In fact,
$$
\scr M_{\bf R}(\overline{\msf E}_\fk a|\msf{PE}_\fk a)=\frac{\msf{Hol}_{\bf R}(\msf S^2,\overline{\msf E}_\fk a)}
{\Aut(\msf S^2\setminus\{\msf q_0,\msf q_\infty\})}
=\frac{\msf B(t)\ltimes \Hol^0_{\bf R}(\msf C_0,\msf E)}
{\Aut(\msf S^2\setminus\{\msf q_0,\msf q_\infty\})}
$$
$\msf B(t)$ is a $\mathbb Z_r$-gerbe over $\msf B'(t)$,
we may extend it with respect to ${\Aut(\msf S^2\setminus\{\msf q_0,\msf q_\infty\})}\cong \cplane^\ast
$ to get a $\cplane_o^\ast$-gerbe over $\msf B'(t)$, which we denote it by $\tilde{\msf B}(t)$. In fact,
$\tilde{\msf B}^1(t)$ can be obtained by $\msf B^1(t)\times_{\mathbb Z_r}\cplane^\ast_o$. Then
$$
\scr M_{\bf R}=\frac{\msf B(t)\ltimes \Hol^0_{\bf R}(\msf C_0,\msf E)}
{\Aut(\msf S^2\setminus\{\msf q_0,\msf q_\infty\})}
=\tilde{\msf B}(t)\ltimes \Hol^0_{\bf R}(\msf C_0,\msf E)
=\msf B'(t)\ltimes ( \cplane^\ast_o\ltimes \Hol^0_{\bf R}(\msf C_0,\msf E))
$$
Since
$
\Hol^0_{\bf R}(\msf C_0,\msf E)\rto B^0(t)
$
is a fibration with fiber $\Hol^0_{\bf R'}(\msf C_0,
[\overline{\cplane^n}]^\fk b_\fk a)$, we know that
$$
\cplane^\ast_o\ltimes \Hol^0_{\bf R}(\msf C_0,\msf E)
\rto B^0(t)
$$
is a fibration with fiber $\scr M_{\bf R'}([\overline{\cplane^n}]^\fk b_\fk a|\msf P^\fk b_\fk a)$. With $\msf B'(t)$ action, we know that
$\scr M_{\bf R}(\overline{\msf E}_\fk a|\msf{PE}_\fk a)$ is a fibration over $\msf B'(t)$ with fiber
$\scr M_{\bf R'}([\overline{\cplane^n}]^\fk b_\fk a|
\msf P^\fk b_\fk a)$.

\end{proof}

\subsubsection{Certain relative Gromov--Witten invariants}\label{sec rel-inv-E-B}

We consider certain relative Gromov--Witten invariants
associated to $\scr M_{\bf R}(\overline{\msf E}_\fk a|
\msf {PE}_\fk a)$.

To describe the insertions, we make some preparations.
\vskip 0.1in
\noindent
(1) {\em Basis of $H^\ast_{CR}(\msf B)$.}
\vskip 0.1in
Recall that
$$
H^\ast_{CR}(\msf B)=H^\ast(\msf{IB})=\bigoplus_{t\in \scr T^{\msf B}} H^\ast(\msf B(t)).
$$
For each $t$ let $\sigma_{(t)}$ be a basis of
$H^\ast(\msf B(t))$:
$$
\sigma_{(t)}=\{\theta_{(t)}^1,\ldots, \theta_{(t)}^{k(t)}\}.
$$
Set
$$
\sigma_\star= \bigsqcup_{t\in\scr T^{\msf B}} \sigma_{(t)},
$$
which forms a basis of $H^\ast_{CR}(\msf B)$.

Let $\sigma^\star$ be the dual basis of $\sigma_\star$. Then
$$
\sigma^\star=\bigsqcup_{t\in \scr T^{\msf B}}\sigma^{(t)},
$$
where $\sigma^{(t)}$ is the dual of $\sigma_{(\bar t)}$.
Suppose that
$$
\sigma^{(t)}=\{\theta^{(t)}_1,\ldots, \theta^{(t)}_{k(t)}\},
$$
where $\theta^{(t)}_i$ is the dual of $\theta_{(\bar t)}^i$.
\vskip 0.1in
\noindent
(2) {\em Basis  of $H^\ast_{CR}(\msf {PE}_\fk a)$}
\vskip 0.1in
Let $s\in \scr T^{\msf{PE}_\fk a}$ and $t=\pi_{\scr T}(s)$ (cf. Lemma \ref{lem pit-twist-sec}). Then $\msf {PE}_\fk a(s)$ is
a projectivization ${\msf {PE}^{(s)}_{\fk a(s)}}$ of a rank $r(s)$-vector bundle $\msf E^{(s)}$
over $\msf B(t)$ with respect to certain $S^1_{\fk a(s)}$ action. Set (cf. \eqref{E cohomo-of-PE})
$$
\Sigma^{(s)}
=\{\theta_j^{(t)}\cup H_{(s)}^\ell|\theta_j^{(t)}\in \sigma^{(t)},
0\leq \ell\leq r(s)-1 \},
$$
which forms a basis of $H^\ast(\msf{PE}_\fk a(s))$.
Set
$
\Sigma_\star$ to be the union of $\Sigma^{(s)}$. Similarly, let $\Sigma^\star$ be the dual of $\Sigma_\star$.

\vskip 0.1in
\noindent
(3) {\em Insertions.}
\vskip 0.1in
We are ready to describe the insertions for the relative Gromov--Witten invariants:
\begin{itemize}
\item insertion for marked point $\msf q_0$: recall that the evaluation map is
$$
\msf{ev}_0: \scr M_{\bf R}(\overline{\msf E}_\fk a|
\msf {PE}_\fk a)\rto \msf E(t);
$$
let $\Theta_{(t)}$ be the Thom form of the bundle $\msf E(t)\rto \msf B(t)$, the insertion is taken to be
$$
\msf {ev}_0^\ast(\theta_{(t)}^i\wedge\Theta_{(t)})\cup \psi^c.
$$
Hence, the insertion is supported at $\msf B$.
\item insertion for the relative marked point $\msf q_\infty$:
the evaluation map is
$$
\msf{rev}_\infty: \scr M_{\bf R}(\overline{\msf E}_\fk a|\msf{PE}_\fk a)\rto \msf{PE}_\fk a(s);
$$
it is crucial to note that $\pi_{\scr T}(s)=\bar t$ (cf. Remark \ref{rmk 5.7}).
We take the insertion to be
$$
\msf{rev}^\ast_\infty(\theta_j^{(\bar t)}\cup H_{(s)}^d).
$$
\end{itemize}
\begin{defn}
The pair of insertions
$(\tau_c(\theta^i_{(t)}\cup
\Theta_{(t)}),\theta_j^{(\bar t)}\cup H_{(s)}^d)$ is called
{\bf proper}  if
\begin{equation*}
2c+2d+\deg(\Theta_{(t)})=\dim(\scr M_{\bf R}(\overline{\msf E}_\fk a|\msf {PE}_\fk a))-\dim(\msf B(t))
=\dim(\scr M_{\bf R'}([\overline{\cplane^n}]^\fk b_\fk a|\msf{P}^\fk b_\fk a)).
\end{equation*}
The last ``$=$'' follows from Proposition \ref{prp moduli-E-B}.
\end{defn}

\vskip 0.1in
\noindent
(4) {\em Invariant $\langle\tau_c(\theta^i_{(t)}\cup
\Theta_{(t)})|\theta_j^{(\bar t)}\cup H_{(s)}^d \rangle$
for a pair of proper insertions.}
\vskip 0.1in
We compute the integration.
\begin{eqnarray*}
&&\langle(\tau_c(\theta^i_{(t)}
\cup\Theta_{(t)})
|\theta_j^{(\bar t)}\cup H_{(s)}^d\rangle\\
&=&\int_{\scr M_{\bf R}}\psi^c\cup\msf{ev}_0^\ast(\theta^i_{(t)})
\cup \msf{ev}_0^\ast \Theta_{(t)}\cup\msf{rev}_\infty^\ast(\theta_j^{(\bar t)})\cup
\msf{rev}^\ast_\infty(H_{(s)}^d)\\
&=&\int_{\msf B'(t)}\msf{ev}_0^\ast(\theta^i_{(t)})
\cup\msf{rev}_\infty^\ast(\theta_j^{(\bar t)})
\cdot \int_{\scr M_{\bf R'}}
 \psi^c\cup\msf {ev}_0^\ast \Theta_{(t)}\cup
\msf{rev}_\infty^\ast H_{(s)}^d\\
&=&r\delta^i_j\cdot \int_{\scr M_{\bf R'}}
\psi ^c\cup \msf {ev}_0^\ast \Theta_{(t)}\cup \msf{rev}_\infty^\ast(H_{(s)}^d).
\end{eqnarray*}
Here for the first ``$=$'' we use Proposition \ref{prp moduli-E-B};
for the second ``$=$'' we use the definition of properness
of the insertions; and $\delta^i_j$ is the Kronecker number.

We conclude that
\begin{theorem}\label{thm cpt-rel-inv}
Let $(\tau_c(\theta^i_{(t)}\cup
\Theta_{(t)}),\theta_j^{(\bar t)}\cup H_{(s)}^d)$ be a proper insertion for $\scr M_{\bf R}$. The invariant
$$
\langle{\bf R}(\tau_c(\theta^i_{(t)}\cup
\Theta_{(t)})|\theta_j^{(\bar t)}\cup H_{(s)}^d)\rangle
$$
is non-zero if
\begin{itemize}
\item $\theta^i_{(t)}$ and $\theta_j^{(\bar t)}$ are dual; and
\item $\langle{\bf R'}(\tau_c(\Theta_{(t)})|H_{(s)}^d)\rangle$
is the invariant computed in \S\ref{sec 5.2}, i.e, $\bf R'$
and $d$ are uniquely determined  by $c$.
\end{itemize}
\end{theorem}

\section{Weighted-blowup correspondence for orbifold Gromov--Witten invariants}

In this section we first consider admissible relative data $\scr R_{\Sigma_\star}^\bullet(\X|\Z)$ of a general relative pair
$(\X|\Z)$ with relative insertions coming from a choosing basis $\Sigma_\star$ of $H^*_{CR}(\Z)$. We give a partial order to
$\scr R_{\Sigma_\star}^\bullet(\X|\Z)$. Then we consider the pair $(\Xa|\Z)$ obtained by blowing up $\X$ along a symplectic sub-orbifold groupoid $\oS$ with weight $\wa$. There is a subset $\scr R_{\Sigma_\star,\K}^\bullet(\Xa|\Z)$ of
$\scr R_{\Sigma_\star}^\bullet(\Xa|\Z)$ in which all relative datum have absolute insertions coming from $\K$, which is the image of the Chen--Ruan cohomology of $\X$ via the natural projection $\kappa:\Xa\rto\X$ (cf. \S \ref{subsec blp-X-S} and \S  \ref{sec correspondence}).

\subsection{Partial order on $\scr R_{\Sigma_\star}^\bullet(\X|\Z)$}
\label{S partial-order}
Consider a relative pair $(\X|\Z)$ and take a basis $\Sigma_\star$ of the Chen--Ruan cohomology $H^*_{CR}(\Z)$.
We assume that elements in $\Sigma_\star$ are homogenous with respect to the degree
and the decomposition of the inertia sapce $\sf IZ$. Denote the dual basis by $\Sigma^\star$.
In this section we give a partial order on the set $\scr R_{\Sigma_\star}^\bullet(\X|\Z)$ of
admissible relative $\Sigma_\star$-data of $(\X|\Z)$.

Given the relative pair $(\X|\Z)$, by blowing up $\X$ along $\Z$ with trivial weight $\wa=(1)$ we degenerate $(\X|\Z)$ into
\begin{align}\label{eq dege-X|Z-weit(1)}
(\X|\Z)\xrightarrow{\text{degenerate}} (\X^-|\Z^-)\wedge_\Z (\Z^+|\X^+|\Z)=(\X|\Z)\wedge_\Z(\Z_\infty|\lN|\Z_0)
\end{align}
here $\lN$ is the weight $\wa=(1)$ projectification of the normal bundle $\N_{\Z|\X}$ of $\Z$ in $\X$. To distinguish different divisors in $\lN$, we write the original $\Z$ in $\X$ by $\Z_0\in \lN$, which is the zero section of $\lN$. We write the infinity section of $\lN$ by $\Z_\infty$. By gluing $\Z_\infty$ with the $\Z\in\X$ we get the original $(\X|\Z)$. Denote the fiber class of $\lN\rto\Z$ by $[F_\Z]$.

In the next we also consider relative $(\Sigma^\star,\Sigma_\star)$-data of $(\Z_\infty|\lN|\Z_0)$. A relative $(\Sigma^\star,\Sigma_\star)$-data
\[
{\bf R}^{+,\bullet}(J_\infty|I|J_0)
\]
consists of
\begin{itemize}
\item topological data
       \[
       {\bf R}^{+,\bullet}=\big(g,A,\fk t=(t_1,\ldots,t_k)\big|\fk r_\infty
       ,\fk r_0
       \big)
       \]
       with
       \[
       \fk r_\infty=((s^1_\infty,u^1_\infty),\ldots, (s^{h_\infty}_\infty,u^{h_\infty}_\infty)), \qq
       \fk r_0=((s^1_0,u^1_0),\ldots,(s^{h_0}_0,u^{h_0}_0)),
       \]
\item absolute insertions $I=(\tau_{d_1}\alpha_1,\ldots,\tau_{d_k}\alpha_k)$ with $\alpha_i\in H^*(\lN_{(t_i)})$,
\item relative insertions $\fk i(J_\0)=(\beta^1_\infty,\ldots,\beta^{h_\infty}_\infty)$ with $\beta^i_\infty\in \Sigma^\star$ and $\fk i(J_0)=(\beta^1_0,\ldots,\beta^{h_0}_0)$ with $\beta^i_0\in \Sigma_\star$.
\end{itemize}
Then $J_\0=(\fk r_\0, \fk i(J_\0))$ and $J_0=(\fk r_0, \fk i(J_0))$.

As in Section \ref{sec 4}, we call a relative $(\Sigma^\star,\Sigma_\star)$-data {\bf admissible} if the total degree of insertions matches the virtual dimension of the corresponding relative moduli space.
Denote the set of admissible relative $(\Sigma^\star,\Sigma_\star)$-data of $(\Z_\infty|\lN|\Z_0)$ by
$\scr R_{\Sigma^\star,\Sigma_\star}^\bullet(\Z_\infty|\lN|\Z_0)$.

There is a special kind of admissible $(\Sigma^\star,\Sigma_\star)$-data of $(\Z_\infty|\lN|\Z_0)$.
\begin{defn}\label{def N-mini-data}
A {\bf connected pre-$\lN$-minimal data} in $\scr R_{\Sigma^\star,\Sigma_\star}^\bullet(\Z_\infty|\lN|\Z_0)$
is a connected admissible data ${\bf R}^+(J_\infty|\varnothing|J_0)$
whose topological data is of the form
\[
{\bf R}^+=\big(g=0,A=u\cdot[F_\Z],\fk t=\varnothing|\fk r_\infty=(s_\infty, u),\fk r_0=(\bar s_\infty,u)\big).
\]
We call a connected pre-$\lN$-minimal data an {\bf $\lN$-minimal data} if
$\fk i(\check J_\0)=\fk i(J_0)$, hence $\check J_\infty= J_0$.

A {\bf disconnected pre-$\lN$-minimal ($\lN$-minimal resp.) data} of $(\Z_\infty|\lN|\Z_0)$ is a disjoint union of finite connected pre-$\lN$-minimal ($\lN$-minimal resp.) datum.
\end{defn}

The computation in \S \ref{sec 5} and Theorem \ref{thm cpt-rel-inv} implies that

\begin{proposition}\label{prop inv-minimal}
Let ${\bf R}^+(J_\infty|\varnothing|J_0)$
be a connected pre-$\lN$-minimal data.
Then the invariant
\[
\langle {\bf R}^+(J_\infty|\varnothing|J_0)\rangle^{\Z_\0|\lN|\Z_0}\not= 0
\]
if and only if
${\bf R}^+(J_\infty|\varnothing|J_0)$ is $\lN$-minimal,
i.e. $J_\infty=\check J_0$.
\end{proposition}

\begin{remark}\label{rmk empty-N-minimal}
For completeness we also call empty data
\[
\big(g=0,A=0,\fk t=\varnothing|\fk r_\infty=\varnothing,\fk r_0=\varnothing\big).
\]
$\lN$-minimal, and set its invariant to be $1$.
Note that when $A=0$, $\fk r_\0$ and $\fk r_0$ are both empty.
\end{remark}

\subsubsection{Definition of partial order}

Now we define the partial order on $\scr R_{\Sigma_\star}^\bullet(\X|\Z)$.
\begin{defn}\label{def partial ord}
Take two admissible $\Sigma_\star$-datum ${\bf R}^\bullet_i(I_i|J_i)\in\scr R_{\Sigma_\star}^\bullet(\X|\Z), i=1,2$.
We say that
\begin{align}\label{eq def-prec}
{\bf R}^\bullet_1(I_1|J_1)\prec {\bf R}^\bullet_2(I_2|J_2)
\end{align}
if when we degenerate $(\X|\Z)$ as \eqref{eq dege-X|Z-weit(1)},
there is an admissible $(\Sigma^\star,\Sigma_\star)$-data
${\bf R}^{+,\bullet}(J_\infty|I|J_0)\in\scr R_{\Sigma^\star,\Sigma_\star}^\bullet(\Z_\infty|\lN|\Z_0)$ such that
\begin{itemize}
\item[(P1)] the pair
    \[
    \big({\bf R}^\bullet_1(I_1|J_1),\,{\bf R}^{+,\bullet}(J_\0|I|J_0)\big)
    \]
    is a matched pair, i.e. $J_1=\check J_\0$ and
    \[
    {\bf R}^\bullet_2(I_2|J_2)={\bf R}^\bullet_1(I_1|J_1)\ast {\bf R}^{+,\bullet}(J_\infty|I|J_0),
    \]
    hence $J_2=J_0$, and moreover
\item[(P2)] when ${\bf R}^{+,\bullet}(J_\infty|I|J_0)$ is pre-$\lN$-minimal, it must be $\lN$-minimal. Then $J_2=J_0=\check J_\0=J_1$ and
    ${\bf R}^\bullet_2(I_2|J_2)={\bf R}^\bullet_1(I_1|J_1)$.
    (Here we allow empty $\lN$-minimal data).
\end{itemize}
\end{defn}

The main theorem in this subsection is the following result.
\begin{theorem}\label{thm partial-ord}
$(\scr R^\bullet_{\Sigma_\star}(\X|\Z),\prec)$ is a partial ordered set.
\end{theorem}

\subsubsection{Proof of Theorem \ref{thm partial-ord}}

Before giving the proof of Theorem \ref{thm partial-ord} we give some lemmas.

\begin{lemma}\label{lem 6.5-rel-to-minimal}
Let ${\bf R}^{\bullet}(I|J)\in \scr R_{\Sigma_\star}^\bullet(\X|\Z)$. Then there exists a unique $\lN$-minimal
data ${\bf R}^{+,\bullet}(\check J|\varnothing|J)\in \scr R_{\Sigma^\star,\Sigma_\star}^\bullet(\Z_\infty|\lN|\Z_0)$
such that
\begin{align}\label{eq R+lN-minimal}
{\bf R}^{\bullet}(I|J)\ast {\bf R}^{+,\bullet}(\check J|\varnothing|J)={\bf R}^{\bullet}(I|J).
\end{align}

Conversely, if there is another data
${\bf R}^{+,\bullet}(J_\infty|I^+|J_0)\in \scr R_{\Sigma^\star,\Sigma_\star}^\bullet(\Z_\infty|\lN|\Z_0)$
satisfying \eqref{eq R+lN-minimal} and the associated invariant is nonzero, then it must be the $\lN$-minimal data
${\bf R}^{+,\bullet}(\check J|\varnothing|J)$ in \eqref{eq R+lN-minimal}.
\end{lemma}
\begin{proof}
Suppose that the relative data of ${\bf R}^\bullet (I|J)$ is
\[
{\bf R}^\bullet=\big(g,A,\fk t=(t_1,\ldots,t_m)|\fk r=((s_1,u_1),\ldots,(s_h,u_h))\big),
\]
and the relative insertions are $\fk i(J)=(\beta_1,\ldots,\beta_h)$ with $\beta_k\in\Sigma_\star$.

The $\lN$-minimal data ${\bf R}^{+,\bullet}(\check J|\varnothing|J)$ is constructed as follow.
It consists of of $h=\sharp J$ connected components with $i$-th component ${\bf R}^+_i(\check J_i|\varnothing|J_i)$
being of the following form:
\begin{itemize}
\item topological data
\[
{\bf R}^+_i=\big(g=0,A^+_i=u_i\cdot [F_\Z],\fk t=\varnothing|\fk r_0=(s_i,u_i), \fk r_\infty=((\bar s_i,u_i))\big),
\]
\item the relative insertions are $\check\beta_i$
     for $\Z_\0$ and $\beta_i$ for $\Z_0$.
\end{itemize}
When $\fk r=\varnothing$, all ${\bf R}^+_i(\check J_i|\varnothing|J_i)$ are the empty $\lN$-minimal data.

Obviously, ${\bf R}^+_i(\check J_i|\varnothing|J_i)$ is $\lN$-minimal.
Hence
\[
{\bf R}^{+,\bullet}(\check J|\varnothing|J)=\bigsqcup_{i=1}^{h} {\bf R}^+_i(\check J_i|\varnothing|J_i)
\]
is $\lN$-minimal. Moreover, one can directly see that
\[
{\bf R}^\bullet (I|J)={\bf R}^\bullet (I|J)\ast {\bf R}^{+,\bullet}(\check J|\varnothing|J).
\]

Conversely, suppose there is another data
${\bf R}^{+,\bullet}(J_\infty|I^+|J_0)\in \scr R_{\Sigma^\star,\Sigma_\star}^\bullet(\Z_\infty|\lN|\Z_0)$
such that
\[
{\bf R}^{\bullet}(I|J)\ast {\bf R}^{+,\bullet}(J_\infty|I^+|J_0)={\bf R}^{\bullet}(I|J).
\]
We first consider the case that ${\bf R}^{+,\bullet}(J_\infty|I^+|J_0)$ is not empty. Then
\begin{itemize}
\item the homology class of ${\bf R}^{+,\bullet}(J_\infty|I^+|J_0)$ must be a fiber class $A^+=\sum_i u_i\cdot [F_\Z]$;
\item the genus of each connected component is zero;
\item $J_0=J, \, \check J_\infty=J$, $I^+=\varnothing$.
\end{itemize}
Since the genus of each component is zero, the homology class is a fiber class, every connected component of
${\bf R}^{+,\bullet}(J_\infty|I^+|J_0)$ would have exactly one relative marking mapped to each one of $\Z_0,\Z_\infty$.
So ${\bf R}^{+,\bullet}(J_\infty|I^+|J_0)$ has exactly $\sharp J$ components and every component is pre-$\lN$ minimal.
Then the invariant of ${\bf R}^{+,\bullet}(J_\infty|I^+|J_0)$ is the product of the invariant of each component.
Therefore by $J=J_0=\check J_\infty$,
Proposition \ref{prop inv-minimal} and the assumption
$\langle {\bf R}^{+,\bullet}(J_\infty|I^+|J_0)\rangle\neq 0$,
we see that each component of ${\bf R}^{+,\bullet}(J_\infty|I^+|J_0)$ is $\lN$-minimal.
Hence ${\bf R}^{+,\bullet}(J_\infty|I^+|J_0)$ is $\lN$-minimal.
One can then easily see that ${\bf R}^{+,\bullet}(J_\infty|I^+|J_0)={\bf R}^{+,\bullet}(\check J|\varnothing|J)$.
This also proves the uniqueness of ${\bf R}^{+,\bullet}(\check J|\varnothing|J)$.

When $J=\varnothing$, we take ${\bf R}^{+,\bullet}(\check J|\varnothing|J)$ to be the empty data in Remark \ref{rmk empty-N-minimal}.
\end{proof}

We can also degenerate $(\Z_\infty|\lN|\Z_0)$ by blowing up $\lN$ along $\Z_0$ with trivial weight. This procedure degenerates $(\Z_\infty|\lN|\Z_0)$ into two copies of itself
\begin{align}\label{eq dege-lN}
(\Z_\infty|\lN|\Z_0)\xrightarrow{\text{degenerate}} (\Z_\infty|\lN|\Z_0)\wedge_\Z (\Z_\infty|\lN|\Z_0).
\end{align}
By gluing $\Z_0$ in the first copy and $\Z_\infty$ in the second copy we get the original $(\Z_\infty|\lN|\Z_0)$.
\begin{lemma}\label{lem 6.6-dege-minimal}
Let ${\bf R}^\bullet(J_\infty|\varnothing|J_0)$ be a pre-$\lN$-minimal data. Suppose when we degenerate $(\Z_\infty|\lN|\Z_0)$ as \eqref{eq dege-lN} there holds
\[
{\bf R}^\bullet(J_\0|\varnothing|J_0)=
{\bf R}^{-,\bullet}(J_\0|\varnothing|J_0^-)\ast
{\bf R}^{+,\bullet}(J_\0^+|\varnothing|J_0)
\]
for some $(\Sigma^\star,\Sigma_\star)$-datum of $(\Z_\infty|\lN|\Z_0)$. Then both datum
${\bf R}^{-,\bullet}(J^-_\infty|\varnothing|J_0^-)$ and
${\bf R}^{+,\bullet}(J_\infty^+|\varnothing|J_0^+)$ are pre-$\lN$-minimal.

Suppose further more that
${\bf R}^\bullet(J_\infty|\varnothing|J_0)$ is $\lN$-minimal and, one of ${\bf R}^{-,\bullet}(J^-_\infty|\varnothing|J_0^-)$ and
${\bf R}^{+,\bullet}(J_\infty^+|\varnothing|J_0^+)$ is also $\lN$-minimal. Then
\[
{\bf R}^\bullet(J_\0|\varnothing|J_0)=
{\bf R}^{-,\bullet}(J_\0|\varnothing|J_0^-)=
{\bf R}^{+,\bullet}(J_\0^+|\varnothing|J_0)
\]
and are all $\lN$-minimal.
\end{lemma}
\begin{proof}
We only have to consider the connected and nonempty case. By the assumption we write the topological datum of
${\bf R}(J_\0|\varnothing|J_0)$,
${\bf R}^+(J_\0|\varnothing|J_0^+)$, and
${\bf R}^-(J^-_\0|\varnothing|J_0)$ as
\[\begin{split}
&{\bf R}=\big(g=0,A=u\cdot [F_\Z],
\fk t=\varnothing|
\fk r_\0=(s_\infty, u),
\fk r_0=(\bar s_\0,u)\big),\\
&{\bf R}^+=\big(g^+,A^+,\fk t^+=\varnothing|
\fk r_\0=(s_\0, u),
\fk r^+_0=(s^+_0,u^+_0)\big),\\
&{\bf R}^-=\big(g^-,A^-,\fk t^+=\varnothing|
\fk r^-_\infty=(s^-_\0, u^-_\0),
\fk r_0=(\bar s_\0,u)\big).
\end{split}\]
Then we have
\begin{itemize}
\item $A=A^+=A^-$,\,\, $g^+=g^-=g=0$;
\item $s_\infty=\bar s_0^+$, $s_0^+=\bar s_\0^-$,\,$s_\0^-=s_\0$;
\item $u=u^+_0=u^-_\0$.
\end{itemize}
Hence ${\bf R}^{-,\bullet}(J_\0|\varnothing|J_0^-)$ and
${\bf R}^{+,\bullet}(J_\0^+|\varnothing|J_0)$
both are pre-$\lN$-minimal. This proves the first part.

We next prove the second assertion.
Without loss of generality we assume that
${\bf R}^\bullet(J_\0|\varnothing|J_0)$,
and ${\bf R}^{-,\bullet}(J_\0|\varnothing|J_0^-)$
are $\lN$-minimal. Therefore
\[
J_\0=\check J_0,\qq J_\0=\check J_0^-.
\]
Combining with $J_0^-=\check J^+_\0$ we get
\[
J_\0=J_\0^+=\check J_0=\check J_0^-.
\]
So ${\bf R}^-(J_\0|\varnothing|J_0^-)
={\bf R}^+(J_\infty^+|\varnothing|J_0)
={\bf R}(J_\infty|\varnothing|J_0)$ and are all $\lN$-minimal.
\end{proof}

This lemma says that if we consider the partial order, similar as the one given in Definition \ref{def partial ord},
on $\scr R_{\Sigma^\star,\Sigma_\star}^\bullet(\Z_\infty|\lN|\Z_0)$, $\lN$-minimal datum are minimal elements.

Now we give the proof of Theorem \ref{thm partial-ord}.
\begin{proof}[Proof of Theorem \ref{thm partial-ord}]
The proof consists of several steps.

\n
{\em Step 1: Reflexivity.} This is the first part of Lemma \ref{lem 6.5-rel-to-minimal}. Hence ``$\prec$'' is reflexive.

\n
{\em Step 2: Transitivity.} Let ${\bf R}_i={\bf R}_i(I_i|J_i)\in \scr R_{\Sigma_\star}^\bullet(\X|\Z), i=1,2,3$ be three admissible relative datum of $(\X|\Z)$, ${\bf R}_2\prec {\bf R}_1$, ${\bf R}_3\prec {\bf R}_2$,
and
\[
{\bf R}_1={\bf R}_2\ast {\bf R}_a^{+,\bullet},\,\,\,\text{and}\,\,\,
{\bf R}_2={\bf R}_3\ast {\bf R}_b^{+,\bullet}
\]
with
\[
{\bf R}_a^{+,\bullet}={\bf R}_a^{+,\bullet}(J_a^+|I_a|J_1),\,\,\,
{\bf R}_b^{+,\bullet}={\bf R}_b^{+,\bullet}(J_b^+|I_b|J_2) \in\scr R_{\Sigma^\star,\Sigma_\star}^\bullet(\Z_\infty|\lN|\Z).
\]
Here $I_a,I_b$ are proper extensions of parts of $I_1$ and $I_2$,
and $\fk i(J^+_a), \fk i(J^+_b)\in \Sigma^\star$.
Then since $({\bf R}_2, {\bf R}_a^{+,\bullet})$ and $({\bf R}_3, {\bf R}_b^{+,\bullet})$ are matched pairs we have 
\[
J_2=\check{J}_a^+,\,\,\,\ J_3=\check J_b^+.
\]
Hence ${{\bf R}_b}^{+,\bullet}\ast {{\bf R}_a}^{+,\bullet}$ is an admissible relative $(\Sigma^\star,\Sigma_\star)$-data of $(\Z_\infty|\lN|\Z_0))$ and
\[
{\bf R}_3\ast ({\bf R}_b^{+,\bullet}\ast {\bf R}_a^{+,\bullet})={\bf R}_1.
\]

Now if ${\bf R}_b^{+,\bullet}\ast{\bf R}_a^{+,\bullet}$ is pre-$\lN$-minimal,
then Lemma \ref{lem 6.6-dege-minimal} implies that
${\bf R}_b^{+,\bullet}, {\bf R}_a^{+,\bullet}$ are both pre-$\lN$-minimal.
Then by (P2) in Definition \ref{def partial ord},
both ${\bf R}_b^{+,\bullet}, {\bf R}_a^{+,\bullet}$ are $\lN$-minimal.
Therefore by Lemma \ref{lem 6.6-dege-minimal}
\[
{\bf R}_b^{+,\bullet}\ast{\bf R}_a^{+,\bullet}={\bf R}_b^{+,\bullet}={\bf R}_a^{+,\bullet}
\]
is $\lN$-minimal. Consequently ${\bf R}_3\prec {\bf R}_1$.


\n
{\em Step 3: Antisymmetry.} Suppose now we have two admissible relative datum
${\bf R}^\bullet_1(I_1|J_1)$ and ${\bf R}^\bullet_2(I_2|J_2)$ of $(\X|\Z)$.
Suppose also that
\begin{align}
&{\bf R}^\bullet_1(I_1|J_1)\prec {\bf R}^\bullet_2(I_2|J_2)\label{eq 1<2},\\
&{\bf R}^\bullet_2(I_1|J_1)\prec {\bf R}^\bullet_1(I_2|J_2)\label{eq 2<1}.
\end{align}
Then as above we have
\[
{\bf R}^\bullet_1={\bf R}^\bullet_2\ast {\bf R}_a^{+,\bullet},\,\,\,\text{and}\,\,\,
{\bf R}^\bullet_2={\bf R}^\bullet_1\ast {\bf R}_b^{+,\bullet}
\]
for some
\[
{\bf R}_a^{+,\bullet}={\bf R}_a^{+,\bullet}(J_a^+|I_a|J_1),\,\,\,
{\bf R}_b^{+,\bullet}={\bf R}_b^{+,\bullet}(J_b^+|I_b|J_2) \in\scr R_{\Sigma^\star,\Sigma_\star}^\bullet(\Z_\infty|\lN|\Z).
\]
When one of homology classes of
${\bf R}_a^{+,\bullet}$ and ${\bf R}_b^{+,\bullet}$
is zero, the corresponding $(\Sigma^\star,\Sigma_\star)$-data
of $(\Z_\0|\lN|\Z)$ is empty,
then we have ${\bf R}^\bullet_1={\bf R}^\bullet_2$.
Therefore in the following we assume that both the homology classes of
${\bf R}_a^{+,\bullet}$ and ${\bf R}_b^{+,\bullet}$
are nonzero.
Hence $J_1\neq\varnothing$.
Then we get
\[
{\bf R}^\bullet_1={\bf R}^\bullet_1\ast ({\bf R}_b^{+,\bullet}\ast {\bf R}_a^{+,\bullet}).
\]
Since $J_1\neq\varnothing$, by the second part of Lemma \ref{lem 6.5-rel-to-minimal}, ${\bf R}_b^{+,\bullet}\ast {\bf R}_a^{+,\bullet}$ is the $\lN$-minimal datum determined by ${\bf R}^\bullet_1$ by the construction in the proof of Lemma \ref{lem 6.5-rel-to-minimal}. Therefore by the first part of Lemma \ref{lem 6.6-dege-minimal} both ${\bf R}_b^{+,\bullet}$ and ${\bf R}_a^{+,\bullet}$ are pre-$\lN$-minimal. Then by (P2) in Definition \ref{def partial ord} both ${\bf R}_b^{+,\bullet}$ and ${\bf R}_a^{+,\bullet}$ must be $\lN$-minimal. So the second part of Lemma \ref{lem 6.6-dege-minimal} implies that
\[
{\bf R}_a^{+,\bullet}={\bf R}_b^{+,\bullet}={\bf R}_b^{+,\bullet}\ast {\bf R}_a^{+,\bullet}
\]
So we have
\[
{\bf R}^\bullet_2={\bf R}_1^\bullet\ast {\bf R}_b^{+,\bullet} ={\bf R}_1^\bullet\ast ({\bf R}_b^{+,\bullet}\ast {\bf R}_a^{+,\bullet})={\bf R}_1^\bullet.
\]
Hence ``$\prec$'' is also anti-symmetry.

Therefore ``$\prec$'' is a partial order on $\scr R_{\Sigma_\star}^\bullet(\X|\Z)$.
\end{proof}

As a direct consequence of the finiteness of $\scr D_{{\bf R}^\bullet(I|J)}(I^+,I^-)$ with $I=I^+\cup I^-$ we have
\begin{lemma}\label{lem 6.7 finite-less}
Given an admissible data
${\bf R}^\bullet(I|J)\in\scr R_{\Sigma_\star}^\bullet(\X|\Z)$,
there are only finite admissible datum
${\bf R'}^\bullet(I'|J')\in\scr R_{\Sigma_\star}^\bullet(\X|\Z)$
such that
\[
{\bf R'}^\bullet(I'|J')\prec {\bf R}^\bullet(I|J).
\]
\end{lemma}

\subsection{Correspondence}\label{sec correspondence}

Suppose $\oS$ is a compact symplectic sub-orbifold groupoid of $\X$.
We blow up $\X$ along $\oS$ with weight $\wa$ to degenerate $(\X,\oS)$
into (cf. Section \ref{subsec blp-X-S} and Section \ref{subsec blp-X-S-inv})
\[
(\X,\oS)\xrightarrow{\text{degenerate}} (\X^-|\Z^-)\wedge_\Z(\X^+|\Z^+,\oS)=
(\Xa|\Z)\wedge_\Z (\lN_\wa|\Z,\oS)
\]
where $\lN_\wa$ and $\Z=\msf {PN}_\wa$ are the weight-$\wa$ projectification and weight-$\wa$ projectivization of the normal bundle of $\oS$ respectively. Then we have a canonical morphism $\kappa:\Xa\rto \X$ (cf. \S \ref{subsec blp-X-S}).
It induces a morphism on inertia spaces
\[
\msf I \kappa=\coprod_{(h)\in\scr T^\Xa} \kappa_{(h)}:\coprod_{(h)\in\scr T^\Xa}\Xa(h)\rto
\coprod_{(h)\in\scr T^\Xa}\X(\kappa_{\scr T}(h))
\]
and homomorphism
\[
\msf I \kappa^*=\coprod_{(h)\in\scr T^\Xa}\kappa_{(h)}^*:
\coprod_{(h)\in\scr T^\Xa}H^*(\X(\kappa_{\scr T}(h)))\rto
\coprod_{(h)\in\scr T^\Xa}H^*(\Xa(h)).
\]
As the manifold case, $\msf I\kappa^*$ is injective.
Denote the image of $\msf I\kappa^*$ by $\K$.

As we have done in \S\ref{sec rel-inv-E-B},
we fix a basis $\sigma_\star$ and its dual $\sigma^\star$ for $H^*_{CR}(\oS)$.
Then we get a basis $\Sigma_\star$ and its dual $\Sigma^\star$ for $H^*_{CR}(\Z)$.
Let $\scr R_{\Sigma_\star}^\bullet(\Xa|\Z)$ be the set of admissible relative
$\Sigma_\star$-data of $(\Xa|\Z)$.
We denote by $\scr R_{\Sigma_\star,\K}^\bullet(\Xa|\Z)$ the subset
which consists of admissible $\Sigma_\star$-data ${\bf R}^{\bullet}(I|J)$
such that $I\subseteq \K$. We call such a relative data a $(\Sigma_\star,\K)$-data.

In this section, we give a one-to-one correspondence between $\scr R_{\Sigma_\star,\K}^\bullet(\Xa|\Z)$ and
$\scr A_{\sigma_\star}^\bullet(\X,\oS)$.
For the relative pair $(\Xa|\Z)$ we have the partial ordered set
$(\scr R_{\Sigma_\star}^\bullet(\Xa|\Z),\prec)$ by the construction in previous subsection.
The partial order ``$\prec$'' also gives a partial order on its subset
$\scr R_{\Sigma_\star,\K}^\bullet(\Xa|\Z)$.
We transfer this partial order to $\scr A_{\sigma_\star}^\bullet(\X,\oS)$ via the correspondence we will construct.

As we have done in previous subsection, we need to consider
relative $(\Sigma^\star,\Sigma_\star)$-data of $(\Z_\0|\lN|\Z)$
with $\lN$ being the trivial weight projectification of the normal bundle of
$\Z$ in $\Xa$. In the following we will write $\lN$ as $\lN_\Z$ for emphasis.
We also need to consider relative data of $(\lN_\wa|\Z)$. We denote relative data of it by ${\bf R}_{\lN_\wa}$.
\subsubsection{Correspondence on data}

We give a map from
$\scr R_{\Sigma_\star,\K}^\bullet(\Xa|\Z)$ to $\scr A_{\sigma_\star}^\bullet(\X,\oS)$.
Let
$
{\bf R}^\bullet(I|J)$ be an admissible relative $(\Sigma_\star,\K)$-data where the relative data is
\[
J=(\fk r(J),\fk i(J)),
\]
with
\[
\fk r(J)=((s_1,u_1),\ldots,(s_h,u_h)),\;\;\;
\fk i(J)=(\beta_1,\ldots, \beta_h)
\]
and $\beta_k\in \Sigma_\star$. Hence $\check\beta_k$ are of the form
$\theta^{(\bar t_k)}_{j_k}\cup H_{(\bar s_k)}^{\ell_k}\in \Sigma^{\bar s_k}\subseteq \Sigma^\star$ with $\bar s_k$ the inverse of $s_k$ and $\bar t_k=\pi_{\scr T}(\bar s_k)$.

For the $k$-th relative marking with relative data $(s_k,u_k)$ and insertion $\beta_k$ we associate it an admissible relative data ${\bf R}^+_{\lN_\wa,k}(I_{\msf S,k}|J^+_k)$ of $(\lN_\wa|\Z, \oS)$ like the one discussed in \S\ref{sec 5}:
\begin{enumerate}
\item
     The relative marking $\msf q_\infty$ is associated to $\bar s_k$ with contact order $u_k$ hence the homology class $A_k=R_k[F]$ is fixed with $R_k=u_k$ (cf. Lemma \ref{lem 5.8});
\item
      The relative insertion is
      $\check\beta_k=\theta^{(\bar t_k)}_{j_k}\cup H_{(\bar s_k)}^{\ell_k}$;
\item The absolute marking is mapped to the twisted sector of $\oS_{(t_k)}$ with $t_k=\pi_{\scr T}(s_k)$;
\item
      The absolute insertion is
      $(\theta_{(t_k)}^{j_k}\cup \Theta_{(t_k)})\psi^{c_k}$
      and $c_k$ is determined by ${\fk r}_k=(s_k,u_k)$ and $\ell_k$, i.e. by dimension constraint,
      where $\Theta_{(t_k)}$ is the Thom form of
      $\msf S_{(t_k)}$ in $\msf X_{(t_k)}$ and $\theta_{(t_k)}^{j_k}$
      is the orbifold Poincar\'e dual of
      $\theta^{(\bar t_k)}_{j_k}$ in $H^*_{CR}(\oS)$.
\end{enumerate}
Therefore
$$
{\bf R}^+_{\lN_\wa,k}=\big( g=0, A=A_k,\fk t=(t_k)|\fk r=(\bar s_k, u_k) \big)$$
and
\[
I_{\msf S,k}=((\theta_{(t_k)}^{j_k}\cup \Theta_{(t_k)})\psi^{c_k}),\qq
\fk i(J^+_{k})=(\theta^{(\bar t_k)}_{j_k}\cup H_{(\bar s_k)}^{\ell_k}).
\]
Let ${\bf R}_{\lN_\wa}^{+,\bullet}(I_S|J^+)$ be the disjoint union of
${\bf R}^+_{\lN_\wa,k}(I_{\msf S,k}|J^+_k)$. One can directly see that the pair
\[
({\bf R}^\bullet(I|J),\,\, {\bf R}_{\lN_\wa}^{+,\bullet}(I_\msf S|J^+))
\]
is a matched pair. Since $I\in\K$, we have
\[
{\bf A}^\bullet(I_\kappa;I_S):=
{\bf R}^\bullet(I|J)\ast {\bf R}_{\lN_\wa}^{+,\bullet} (I_S|J^+)
\]
with $I_\kappa=(\msf I\kappa^\ast)\inv(I)\in H^*_{CR}(\X)$.
Moreover ${\bf A}^\bullet(I_\kappa;I_S)\in\scr A_{\sigma_\star}(\X,\oS)$.
We define
\begin{align}\label{eq def-Psi-Jnonempty}
\Psi({\bf R}^\bullet(I|J)):={\bf A}^\bullet(I_\kappa;I_S).
\end{align}

When the relative data ${\bf R}^\bullet(I|J)$ has empty relative insertions, i.e.
\[
\fk r=\varnothing,\qq J=\varnothing,
\]
we must have $A\cdot [\Z]=0$.
We define the corresponding absolute data to be
\begin{align}\label{eq def-Psi-Jempty}
\Psi({\bf R}^\bullet(I|\varnothing)):={\bf A}^\bullet(I_\kappa;\varnothing)
={\bf A}^\bullet(I_\kappa)
\end{align}
with ${\bf A}^\bullet=(g,\kappa_*A,\fk t)$.

The above construction gives us a map
$$
\Psi: \scr R_{\Sigma_\star,\K}^\bullet(\Xa|\Z)\rto \scr A_{\sigma_\star}^\bullet(\X,\oS).
$$

We next construct a left inverse map of $\Psi$.

\begin{lemma}\label{lem rel-outof-IS}
Given an absolute insertion relative to $\sf S$ of the form
$I_\oS=(\theta^j_{(t)}\cup \Theta_{(t)})\psi^{c}$,
there exists a unique admissible relative data
${\bf R}^+_{\lN_\wa}(I_\oS|J^+)={\bf R}^+((\theta^j_{(t)}\cup \Theta_{(t)})\psi^{c}|\theta^{(\bar t)}_j\cup H_{(s)}^{\ell})$
of $(\lN_\wa|\Z,\oS)$, whose invariant is nonzero and homology class is
a fiber class of $\lN_\wa\rto\oS$.
\end{lemma}
\begin{proof}
By Theorem \ref{thm c-dertmin-all} and Lemma \ref{lem 5.8}, the number $c$ and the twisted sector $t$ uniquely determine
the topological data
\[\begin{split}
&{\bf R}^+_{\lN_\wa}=\big(g=0,A=R[F],\fk t=(t)|\fk r=(s,u)\big),\\
\end{split}\]
with $u=R$. We assign the insertions
\[
I_{\msf S}=(\theta^j_{(t)}\cup \Theta_{(t)})\psi^{c},\,\,\,\,\text{and}\,\,\,\,
\fk i(J^+)=\theta^{(\bar t)}_j\cup H_{(s)}^{\ell},
\]
with $\ell$ being determined by the dimension constraint.
So ${\bf R}_{\lN_\wa}^+(I_\oS|J^+)$ is an admissible relative data.
Its invariant is nonzero follows from Theorem \ref{thm cpt-rel-inv}.
\end{proof}

By using this lemma we could get a left inverse of $\Psi$. Actually

\begin{theorem}\label{thm crepdns-on-datum}
When $\codim\oS\geq 4$, $\Psi$ is a bijection. When $\codim\oS=2$, i.e. $\oS$ is a symplectic divisor, $\Psi$ is injective.
\end{theorem}
\begin{proof}
{\em Injectivity.}
First consider the case with $\fk r\neq\varnothing$.
Then injectivity follows from the construction
${\bf R}^+_k(I_{\msf S,k}|J^+_k)$ and Lemma \ref{lem rel-outof-IS}.
Secondly, for the case with $\fk r=\varnothing$, by the definition \eqref{eq def-Psi-Jempty} above, that $\Psi$ is injective is equivalent
to that if $\kappa_*A_1=\kappa_*A_2$ and
$A_i\cdot [\Z]=0$ then $A_1=A_2$.
The equality $\kappa_*A_1=\kappa_*A_2$ implies
that $A_1-A_2$ is a fiber class of $\Z\rto \oS$.
Then by $A_i\cdot [\Z]=0$ we conclude $A_1=A_2$.
This proves the injectivity of $\Psi$.

\v\n
{\em Surjectivity}. We consider the case with codimension greater than 2.
Let ${\bf A}^\bullet(I;I_\oS)\in \scr A_{\sigma_\star}^\bullet(\sf X,S)$.
First consider the case with $I_\oS\neq \varnothing$.
Suppose that
$$
I_\oS=((\theta_{(t_1)}^{j_1}\cup \Theta_{(t_1)})\psi^{c_1},
\ldots, (\theta_{(t_h)}^{j_h}\cup\Theta_{(t_h)})\psi^{c_h})).
$$
Then for each $I_{\oS,k}:=(\theta_{(t_k)}^{j_k}\cup \Theta_{(t_k)})\psi^{c_k}$ we may get a unique ${\bf R}^+_{\lN_\wa,k}(I_{\msf S,k}|J^+_k)$ of $(\lN_\wa|\Z,\oS)$
(with homology class $R_i[F]$, where $R_i$ is determined by $c_i$ and $t_i$)
by Lemma \ref{lem rel-outof-IS}.
Let ${\bf R}^{+,\bullet}_{\lN_\wa}(\tilde I_\oS|J^+)$ being the disjoint union of ${\bf R}^+_{\lN_\wa,k}(I_{\oS,k}|J^+_k)$.

There is an admissible relative data
${\bf R}^\bullet(I^\kappa|J)$ of $(\Xa|\Z)$ with $I^\kappa=\kappa^\ast(I)$
such that
\begin{align}\label{eq needed-in-coresp-1}
{\bf R}^\bullet(I^\kappa|J)\ast {\bf R}^{+,\bullet}_{\lN_\wa}(\tilde I_\msf S|J^+)={\bf A}^\bullet(I;I_\msf S).
\end{align}
We next construct it.

Suppose the homology class of ${\bf A}^\bullet$ is $A$.
Let $\kappa^!A$ being the class in $\Xa$ such that $\kappa_*(\kappa^!A)=A$ and
$\kappa^!A\cdot[\Z]=0$.
Then the homology class of ${\bf R}^\bullet$ is
\[
\kappa^!A-\sum_i R_i[F].
\]
The genus of ${\bf R}^\bullet$ is the same as the genus of ${\bf A}^{\bullet}$.
The absolute insertions are $I^\kappa:=\kappa^*(I)$;
the relative insertions are
\[
J=(\check J^+_1,\ldots \check J^+_h),
\]
i.e. the dual of $J^+$ in ${\bf R}_{\lN_\wa}^{+,\bullet}( I_{\msf S}|J^+)$.
This gives us the ${\bf R}^\bullet (I^\kappa|J)$.
We define ${\bf R}^\bullet(I^\kappa|J)$ to be
$\Psi^{-1}({\bf A}^\bullet(I;I_\msf S))$.

If $I_\oS=\varnothing$, then we set
$\Psi\inv({\bf A}^{\bullet}(I;\varnothing))
={\bf R}^\bullet(I^\kappa|\varnothing)$ with
${\bf R}^{\bullet}=(g,p^!A,\fk t|\fk r=\varnothing)$.

This $\Psi\inv$ is the inverse of $\Psi$.
\end{proof}

\begin{remark}
The $\Psi$ is not surjective in codimension 2 case is because of the following fact. For this case $\Z\rto\oS$ is not a fiber bundle with generic
fiber being a weighted space (or finite quotient of weighted space), but a $\integer_{a_1}$-gerbe, where $\wa=(\alpha_1)$ is the blowup weight.
Consider an admissible relative data
${\bf R}(I|J)$ with ${\bf R}=(g=0,A,\fk t|(s,u))$
and $\fk i(J)=\beta=\kappa_{(s)}^*(\theta_{\kappa_{\scr T}(s)})$.
Then we have $A\cdot [\Z]=u$, and
\[
\Psi({\bf R}(I|J))={\bf A}(I_\kappa;I_\oS)
\]
with ${\bf A}=(g=0,\kappa_*A,(\kappa_{\scr T}(\fk t),\kappa_{\scr T}(s)))$,
$I_\kappa=(\msf I\kappa^*)\inv(I)$ and
$I_\oS=(\theta_{\kappa_{\scr T}(s)}\cup \Theta_{\kappa_{\scr T}(s)})\psi^{c}$ with $c=[u\cdot \alpha_1]$
by dimension constraint (cf. Proposition \ref{prop dimc}).

Suppose $\deg \theta_{\kappa_{\scr T}(s)}>0$.
Take a $\theta_{\kappa_{\scr T}(s)}'$ such that
$\deg \theta_{\kappa_{\scr T}(s)}'=\deg\theta_{\kappa_{\scr T}(s)}-1$.
Now we construct another admissible absolute data ${\bf A}(I_\kappa;I_\oS')$
out of the absolute data ${\bf A}(I_\kappa;I_\oS)$.
The topological datum are the same except that
the absolute insertion is
\[
I_\oS'=(\theta_{\kappa_{\scr T}(s)}'\cup \Theta_{\kappa_{\scr T}(s)} )\psi^{c+1}\neq I_\oS.
\]
Then ${\bf A}(I_\kappa;I_\oS')$ has no inverse image under $\Psi$.
Otherwise suppose $\Psi({\bf R}'(I'|J'))={\bf A}(I_\kappa;I_\oS')$.
Then the homology class $A'$ of ${\bf R}'$ satisfies
\[
\kappa_*A'=\kappa_*A.
\]
Since now $\Z\rto \oS$ is just a $\integer_{\alpha_1}$-gerbe,
we have $A'=A$. Then $A'\cdot[\Z]=A\cdot[\Z]=u$.
But as above since $\Psi({\bf R}'(I'|J'))={\bf A}(I_\kappa;I_\oS')$, by the dimension constraint
we must have
\[
[\alpha_1 \cdot (A'\cdot[\Z])]=[\alpha_1\cdot u]=c+1.
\]
A contradiction.
\end{remark}

\subsubsection{Correspondence on invariants}
By using $\scr R_{\Sigma_\star,\K}^\bullet(\Xa|\Z)$ we get a real linear space
\begin{align}
\rone_{\Sigma_\star,\K}=\rone\{\scr R_{\Sigma_\star,\K}^\bullet(\Xa|\Z)\}.
\end{align}
Denote by
\[
v_{\Sigma_\star,\K}\in\rone_{\Sigma_\star,\K}
\]
the vector with coordinate over ${\bf R}^{\bullet}(I|J)$ being the invariant $\langle{\bf R}^{\bullet}(I|J)\rangle^{\Xa|\Z}$, and the vector
\[
v_{\sigma_\star}\in\rone_{\Sigma_\star,\K}
\]
with coordinate over ${\bf R}^{\bullet}(I|J)$ being the invariant $\langle\Psi({\bf R}^{\bullet}(I|J))\rangle^\X$.

\begin{theorem}\label{thm lower-traingl}
There is a linear map $L:\rone_{\Sigma_\star,\K}\rto \rone_{\Sigma_\star,\K}$ having the following properties.
\begin{itemize}
\item $L(v_{\Sigma_\star,\K})=v_{\sigma_\star}$;
\item the matrix of $L$ with respect to the basis $\scr R_{\Sigma_\star,\K}^\bullet(\Xa|\Z)$
      is lower triangle and the diagonals are all nonzero.
\end{itemize}
Hence by Lemma \ref{lem 6.7 finite-less} $L$ is invertible.
\end{theorem}

\begin{proof}
Take a relative data ${\bf R}^{\bullet}(I|J)$, and its correspondence absolute data $\Psi({\bf R}^{\bullet}(I|J))={\bf A}^\bullet(I_\kappa;I_\oS)$ of $(\X,\oS)$. We degenerate $\X$ via blowup along $\oS$ with weight $\fk a$. Then the degeneration formula expresses the absolute invariant
\begin{align}\label{eq def-L}
\langle{\bf A}^\bullet(I_\kappa;I_\oS)\rangle^\X
=\sum_{\Omega\in\scr D_{{\bf A}^\bullet(I_\kappa;I_\oS)}(I^+,I^-)}
c_\Omega
\cdot
\langle {\bf R}^\bullet_-(I^-|J_\msf Z^-)\rangle^{\Xa|\Z}
\cdot
\langle{\bf R}^\bullet_{\lN_\wa,+}(I^+|J_{\msf Z}^+)\rangle^{\lN_\wa|\Z}.
\end{align}
We require that $J^+_\Z$ and $J^-_\Z$ belong to $\Sigma^\star$ and $\Sigma_\star$ respectively, and $I^-$ belong to $\K$. Hence all ${\bf R}^\bullet_-(I^-|J_\msf Z^-)$ in the RHS belongs to $\scr R_{\Sigma_\star,\K}^\bullet(\Xa|\Z)$. Therefore all of them are mapped to absolute datum by $\Psi$.

We define
\[
L(\ldots,\langle {\bf R}^{\bullet}(I|J)\rangle^{\Xa|\Z},\ldots):=
(\ldots,\langle{\bf A}^\bullet(I_\kappa;I_\oS)\rangle^\X,\ldots)
.
\]
Then we have $L(v_{\Sigma_\star,\K})=v_{\sigma_\star}$.
From \eqref{eq def-L} we see that
$c_\Omega\langle{\bf R}^\bullet_{\lN_\wa,+}(I^+|J_{\msf Z}^+)\rangle^{\lN_\wa|\Z}$ is the coefficient of the matrix of $L$ with respect to the basis $\scr R_{\Sigma_\star,\K}^\bullet(\Xa|\Z)$ of $\rone_{\Sigma_\star,\K}$.

We next prove that the matrix of $L$ is lower triangle and has nonzero diagonal. This consists of 2 steps.
\v
\n
{\em Step 1: Lower triangle.} This is equivalent to the statement that all relative datum ${\bf R}^\bullet_-(I^-|J_\msf Z^-)$ on the RHS of \eqref{eq def-L} are lower than ${\bf R}^{\bullet}(I|J)$.

Consider a pair $({\bf R}^\bullet_-(I^-|J_\msf Z^-),{\bf R}^\bullet_{\lN_\wa,+}(I^+|J_{\msf Z}^+))$ appearing in the RHS of \eqref{eq def-L}. We next show that either $({\bf R}^\bullet_-(I^-|J_\msf Z^-)\prec ({\bf R}(I|J)$ or $\langle{\bf R}^\bullet_{\lN_\wa,+}(I^+|J_{\msf Z}^+))\rangle^{\lN_\wa|\Z}=0$.

${\bf R}^\bullet_{\lN_\wa,+}(I^+|J_{\msf Z}^+)$ is a relative data of $(\lN_\wa|\Z,\oS)$. We degenerate $\lN_\fk a$ along $\oS$ by blowing up $\lN_\wa$ along $\oS$ with weight $\fk a$. Then we get
\begin{align}\label{eq dege-lNa}
(\lN_\fk a|\Z,\oS)\xrightarrow{\text{degenerate}} (\Z|\lN_\fk a^-|\Z_0)\wedge (\lN_\fk a^+|\Z,\oS)=(\Z_\infty|\lN_\Z|\Z_0)\wedge (\lN_\fk a|\Z,\oS)
\end{align}
where
\begin{itemize}
\item $\lN_\Z$ is the trivial weight projectification of the normal bundle of $\Z$ in $\lN_\wa$,
\item $\Z_\infty$ is the infinite section of $\lN_\Z$, and corresponds to the original $\Z$ in $\lN_\fk a$ after we glue back,
\item by gluing $\Z_0$ with the last $\Z$ in $\lN_\wa^+$ we get the original $(\lN_\wa|\Z,\oS)$.
\end{itemize}
Note that $I_\oS$ is contained in $I^+$. Now we apply Lemma \ref{lem rel-outof-IS} to $I_\msf S\in I^+$ and get an admissible relative data
\begin{align}\label{eq relative-R-S}
{\bf R}^{\bullet}_{\lN_\wa}(I_\msf S|J^+)
=\bigsqcup_{k=1}^{\sharp I_{\oS}}
{\bf R}^+_{\lN_\wa,k}( I_{\oS,k}|J^+_k)
\end{align}
of $(\lN_\wa^+|\sf Z,S)=(\lN_\wa|\Z,\oS)$. By the construction of $\Psi$ we see that the relative data
$$J^+=(J^+_1,\ldots, J^+_{\sharp I_{\oS}})=\check J,$$
is the dual of $J$ in the original ${\bf R}^{\bullet}(I|J)$.

There is an admissible relative data
\begin{align}
{\bf R}^{\bullet}(J^+_\Z|I^+\setminus I_\msf S|J)
\end{align}
of $(\Z_\infty|\lN_\Z|\Z_0)$ such that
\begin{align}
{\bf R}^{\bullet}(J^+_\Z|I^+\setminus I_\oS|J)
\ast
{\bf R}^{\bullet}_{\lN_\wa}( I_\msf S|\check J)
={\bf R}^\bullet_{\lN_\wa,+}(I^+|J_{\msf Z}^+).
\end{align}
In fact, when $\codimc\oS\geq 2$, following  the same proof in Theorem \ref{thm crepdns-on-datum} we get this relative data;
when $\codimc\oS=1$, because that the homology class in ${\bf A}^\bullet$ is obtained from ${\bf R}^\bullet$ by $\Psi$, we can also get this relative data.

Hence
\begin{align}
{\bf R}^\bullet_-(I^-|J_\msf Z^-)
\ast
{\bf R}^{\bullet}(J^+_\Z|I^+\setminus I_\oS|J)
\ast
{\bf R}^{\bullet}_{\lN_\wa}( I_\msf S|\check J)
={\bf A}^\bullet(I_\kappa;I_\msf S).
\end{align}
which implies
\begin{align}
{\bf R}^\bullet_-(I^-|J_\msf Z^-)
\ast
{\bf R}^{\bullet}(J^+_\Z|I^+\setminus I_\oS|J)
={\bf R}^\bullet(I|J)
\end{align}
by the construction and injectivity of $\Psi$.
Hence when ${\bf R}^{\bullet}(J^+_\Z|I^+\setminus I_\oS| J)$ is not pre-$\lN$-minimal\footnote{Here $\lN$ is the projectification of the normal bundle of $\Z$ in $\Xa$, hence the $\lN_\Z$ in \eqref{eq dege-lNa}.}, ${\bf R}^\bullet_-(I^-|J_\msf Z^-)\prec {\bf R}(I|J)$.

Now suppose that ${\bf R}^{\bullet}(J^+_\Z|I^+\setminus I_\oS| J)$ is pre-$\lN$-minimal but not $\lN$-minimal. Then
\[
I^+=I_\oS,                        \qq
\fk r(J^+_\Z)=\fk r (\check J),            \qq
\fk i(J^+_\Z)\neq  \fk i(\check J),
\]
and
\[
{\bf R}^{\bullet}(J^+_\Z|I^+\setminus I_\oS|J)=
{\bf R}^{\bullet}(J^+_\Z|\varnothing|J)=
\bigsqcup_{k=1}^{\sharp I_{\oS}}{\bf R}_k(J_{\Z,k}^+|\varnothing|J_k)
\]
are disjoint union of pre-$\lN$-minimal datum of $(\Z_\infty|\lN_\Z|\Z_0)$. Then at least one of them are not $\lN$-minimal. Without loss of generality suppose the first one ${\bf R}^{\bullet}_1(J_{\Z,1}^+|\varnothing|J_1)$ is not $\lN$-minimal, i.e.
$\fk i(J^+_{\Z,1})\neq \fk i(\check J_1)$.

On the other hand, now
\[
{\bf R}^\bullet_{\lN_\wa,+}(I^+|J_\Z^+)=
{\bf R}^{\bullet}(J_\Z^+|\varnothing|J)
\ast {\bf R}^{\bullet}_{\lN_\wa}(I_\oS|\check J)=
\bigsqcup_{k=1}^{\sharp I_{\oS}}
\Big({\bf R}_k(J^+_{\Z,k}|\varnothing|J_k)\ast
{\bf R}^+_{\lN_\wa,k}( I_{\oS,k}|\check J_k)\Big).
\]
Consider the first component
${\bf R}_1(J^+_{\Z,1}|\varnothing|J_1)
\ast
{\bf R}^+_{\lN_\wa,1}( I_{\oS,1}|\check J_1)$.
The topological data of this component is the following:
\begin{itemize}
\item the genus is zero,
\item the homology class is a fiber class, and
\item there is one relative marking mapped to the infinity section $\Z_\0$ of $\lN_\wa$ and one absolute marking mapped to the zero section $\oS$ of $\lN_\wa$.
\end{itemize}
We degenerate $(\lN_\wa|\Z;\oS)$ as in \eqref{eq dege-lNa} and then the degeneration formula gives us
\[
\langle {\bf R}^{\bullet}_1(J^+_{\Z,1}|\varnothing|J_1)
\ast
{\bf R}^+_{\lN_\wa,1}( I_{\oS,1}|\check J_1) \rangle^{\lN_\wa|\Z}
=\sum_{J'\in\Sigma^\star}
c_{J'}
\cdot
\langle{\bf R}^{\bullet}_1(J^+_{\Z,1}|\varnothing|J') \rangle^{\Z_\0|\lN_\Z|\Z_0}
\cdot \langle {\bf R}^+_{\lN_\wa,1}(I_{\oS,1}|\check J')\rangle^{\lN_\wa|\Z}
\]
Then by Theorem \ref{thm cpt-rel-inv} we get
\[
\langle {\bf R}^{\bullet}_1(J^+_{\Z,1}|\varnothing|J_1)
\ast
{\bf R}^+_{\lN_\wa,1}(I_{\oS,1}|\check J_1) \rangle^{\lN_\wa|\Z}
=c_{J_1}
\cdot
\langle {\bf R}^{\bullet}_1(J^+_{\Z,1}|\varnothing|J_1)\rangle^{\Z_\0|\lN_\Z|\Z_0}
\cdot
\langle {\bf R}^+_{\lN_\wa,1}(I_{\oS,1}|\check J_1)\rangle^{\lN_\wa|\Z}.
\]
Here $c_{J_1}$ is a nonzero constant.
However since $\fk i(J^+_{\Z,1})\neq \fk i(\check J_1)$, by Proposition \ref{prop inv-minimal}
\[
\langle {\bf R}^{\bullet}_1(J^+_{\Z,1}|\varnothing|J_1)\rangle^{\Z_\infty|\lN_\Z|\Z_0}=0
\]
which implies that the invariant
\[\begin{split}
\langle{\bf R}^+_{\lN_\wa}(I^+|J_{\msf Z}^+)\rangle^{\lN_\wa|\Z}&=\prod_k\langle {\bf R }^{\bullet}_k(J^+_{\Z,k}|\varnothing|J_k)\ast {\bf R}^+_{\lN_\wa,k}(I_{\oS,k}|\check J_k) \rangle^{\lN_\wa|\Z} 
=0.
\end{split}\]

Finally, suppose ${\bf R}^{\bullet}(J^+_\Z|I^+\setminus I_\oS| J)$ is $\lN$-minimal. Then \[
I^+=I_\oS,                        \qq
\fk r(J^+_\Z)=\fk r (\check J),            \qq
\fk i(J^+_\Z)=  \fk i(\check J),
\]
and
\[
{\bf R}^{\bullet}(J^+_\Z|I^+\setminus I_\oS|J)=
{\bf R}^{\bullet}(J^+_\Z|\varnothing|J)=
\bigsqcup_{k=1}^{\sharp I_{\oS}}{\bf R}_k(J_{\Z,k}^+|\varnothing|J_k)
\]
are disjoint union of $\lN$-minimal datum of $(\Z_\infty|\lN_\Z|\Z_0)$. Therefore
\[
{\bf R}^\bullet_{\lN_\wa,+}(I^+|J_\Z^+)=
{\bf R}^{\bullet}(J_\Z^+|\varnothing|J)
\ast {\bf R}^{\bullet}_{\lN_\wa}(I_\oS|\check J)={\bf R}^{\bullet}_{\lN_\wa}(I_\oS|\check J).
\]
Consequently, ${\bf R}^\bullet_{-}(I^-|J_\Z^-)={\bf R}^\bullet(I|J)$.

Therefore every ${\bf R}^\bullet_-(I^-|J_\msf Z^-)$ in the RHS of \eqref{eq def-L} is either lower than or equal to ${\bf R}^\bullet(I|J)$ or has zero coefficient. Therefore the matrix of $L$ is lower triangle.
\v\n
{\em Step 2: Non-vanishing diagonal.} Obviously
\begin{align}
{\bf R}^\bullet(I|J)\ast
{\bf R}^{\bullet}_{\lN_\wa}( I_\oS|J^+)
={\bf A}^\bullet(I_\kappa;I_\msf S)
\end{align}
for the admissible relative data ${\bf R}^\bullet_{\lN_\wa}(I_\oS|J^+)$ in
\eqref{eq relative-R-S}, hence $J^+=J$.
The data ${\bf R}^{\bullet}_{\lN_\wa}( I_\oS|J^+)$ is exactly the proper data we discussed in \S\ref{sec 5}. By Theorem \ref{thm cpt-rel-inv} we see that the invariant
\begin{align}
\langle {\bf R}_{\lN_\wa}^{\bullet}( I_\oS|J^+)\rangle^{\lN_\wa|\Z}\neq 0.
\end{align}
Hence the pair $({\bf R}^\bullet(I|J),{\bf R}^{\bullet}_{\lN_\wa}(I_\oS|J^+))$ appears in the RHS of \eqref{eq def-L}.

When $J=\varnothing$, then $I_\oS=\varnothing$. Then in the degeneration formula for ${\bf A}^\bullet(I_\kappa;\varnothing)$, the coefficient of ${\bf R}^\bullet(I|\varnothing)$ is $1$.

This shows that the diagonal are all non-zero.
The proof is accomplished.
\end{proof}

\begin{remark}
Our theorem generalizes the results of \cite{MauP06,HLR08} on the blowup correspondence of ordinary blowup along symplectic submanifold to the case of weighted blowup along symplectic manifold, and moreover, to the orbifold case.
\end{remark}
\begin{remark}
Our results hold for the following more general cases:
\begin{itemize}
\item the absolute insertions $I$ in ${\bf A}^{\bullet}(I;I_\msf S)$ contain psi-classes, i.e. descendent insertions,
\item  $\oS$ is a disjoint union $\oS=\coprod_{1\leq i\leq I}\oS_i$ of symplectic sub-orbifolds of $\X$
provided that for $\oS_i, 1\leq i\leq I$, there exist symplectic neighborhoods $\U_i$
which do not intersect with each other. Then after blowing up, the exceptional divisors are disjoint from each other and have pairwise disjoint tubular neighborhoods.
\end{itemize}
\end{remark}

We may restrict the correspondence on a subset of $\scr R_{\Sigma_\star,\K}^\bullet(\Xa|\Z)$. As in \S \ref{subsec 1.2}, let $[pt]_{\sf ker}\in H^{\dim \X}(\msf{ker\,X})$ denote the point class of a connected component of $\msf{ker\,X}$. Then $\msf I\kappa^*[pt]_{\sf ker}$ is the point class of the corresponding connected component of $\msf{ker\,\underline{X}}_{\wa}$.

We denote by $\scr R_{\Sigma_\star,\K,0,[pt]_{\sf ker}}^\bullet(\Xa|\Z)$ the subset of $\scr R_{\Sigma_\star,\K}^\bullet(\Xa|\Z)$ which consists of genus zero relative datum with $\msf I\kappa^*[pt]_{\sf ker}$ being the first absolute insertion. Then we have the corresponding subset
$
\scr A_{\sigma_\star,0,[pt]_{\sf ker}}^\bullet(\X,\oS)
$ of
$
\scr A_{\sigma_\star}^\bullet(\X,\oS)
$
which consists of genus zero absolute datum with $[pt]_{\sf ker}$ being the first insertion.
Let
\[
\Psi_{0,[pt]_{\sf ker}}:\scr R_{\Sigma_\star,\K,0,[pt]_{\sf ker}}^\bullet(\Xa|\Z)\rto \scr A_{\sigma_\star,0,[pt]_{\sf ker}}^\bullet(\X,\oS)
\]
be the restriction of
$
\Psi:\scr R_{\Sigma_\star,\K}^\bullet(\Xa|\Z)\rto \scr A_{\sigma_\star}^\bullet(\X,\oS)$. It is a bijection when $\codim\oS\geq 4$ and an injection when $\codim\oS=2$.
We also have the corresponding linear space $\rone_{\Sigma_\star,\K,0,[pt]_{\sf ker}}$ and vectors
\[
v_{\Sigma_\star,\K,0,[pt]_{\sf ker}},\,\, v_{\sigma_\star,0,[pt]_{\sf ker}}\in\rone_{\Sigma_\star,\K,0,[pt]_{\sf ker}}
\]
The linear map $L$ also restricts on $\rone_{\Sigma_\star,\K,0,[pt]_{\sf ker}}$, which we denote by $L_{0,[pt]_{\sf ker}}$.
\begin{theorem}\label{thm correspds-pt}
The linear transformation $L:\rone_{\Sigma_\star,\K}\rto\rone_{\Sigma_\star,\K}$
restricts to a linear transformation
$L_{0,[pt]_{\sf ker}}:\rone_{\Sigma_\star,\K,0,[pt]_{\sf ker}}\rto \rone_{\Sigma_\star,\K,0,[pt]_{\sf ker}}$ and
$$
L_{0,[pt]_{\sf ker}}(v_{\Sigma_\star,\K,0,[pt]_{\sf ker}}) =v_{\sigma_\star,0,[pt]_{\sf ker}}.
$$
Moreover, the matrix of $L_{0,[pt]_{\sf ker}}$ with respect to the basis
$\scr R_{\Sigma_\star,\K,0,[pt]_{\sf ker}}^\bullet(\Xa|\Z)$ is a lower triangle matrix with non-vanishing diagonals.
\end{theorem}

In this manner, when $[pt]_{\sf ker}$ runs through the generators of all components of $\msf{ker\,X}$ we get $\# \scr T^\X_{\sf ker}$ linear transformations.

\begin{remark}\label{rmk other-restriction-of-L}
There are also other restrictions of $L$ by taking different subset of
$\scr R_{\Sigma_\star,\scr K}(\Xa|\Z)$. For instance,
we could replace 
$[pt]_{\sf ker}$
by the point class of a general twisted sector of $\X$ not only $\msf{ker\,X}$.

    Let $\X(t)$ be a twisted sector of $\X$. After blowing up $\X$ along $\oS$ with weight $\wa$, we get its direct transformation $\Xa(t)$ in $\Xa$ (cf. \S \ref{subsec blp-X-S}). So
    $\msf I\kappa:\Xa(t)\rto \X(t)$. Then the pull back $\msf I\kappa^*([pt]_{(t)})$ of the point class $[pt]_{(t)}$ of $\X(t)$ is the point class of $\Xa(t)$.

    We say that $\X(t)$ is {\bf away from $\oS$} if $\X(t)\nsubseteq\sf IS$. Obviously, when $\X(t)$ is away from $\oS$, $\Xa(t)$ is away from the exceptional divisor $\Z$.

Now we replace the point class $[pt]_{\sf ker}$
    in the construction of $L_{0,[pt]_{\ker}}$
    by the point class $[pt]_{(t)}$ of a twisted sector $\X(t)$ that is away from $\oS$. Then we also get a linear transformation $L_{0,[pt]_{(t)}}:\rone_{\Sigma_\star,\K,0,[pt]_{(t)}} \rto\rone_{\Sigma_\star,\K,0,[pt]_{(t)}}$, and two vectors
    $
    v_{\Sigma_\star,\K,0,[pt]_{(t)}},v_{\sigma_\star,0,[pt]_{(t)}}\in \rone_{\Sigma_\star,\K,0,[pt]_{(t)}}.
    $

    As Theorem \ref{thm correspds-pt}, $L_{0,[pt]_{(t)}}$ satisfies
    $$
    L_{0,[pt]_{(t)}}(v_{\Sigma_\star,\K,0,[pt]_{(t)}}) =v_{\sigma_\star,0,[pt]_{(t)}},
    $$
    and the matrix of $L_{0,[pt]_{(t)}}$ with respect to the basis $\scr R_{\Sigma_\star,\K,0,[pt]_{(t)}}^\bullet(\Xa|\Z)$ is a lower triangle matrix with non-vanishing diagonals.
\end{remark}
R. Wang and the first two authors will study other kinds of restrictions of $L$ in \cite{CDW17a}.

\section{Symplectic uniruledness}

\subsection{Symplectic uniruledness of orbifold groupoids}

Let$(\X,\omega)$ be a compact symplectic orbifold groupoid. As in \S \ref{subsec 1.2}, for $\msf{ker\,X}=\msf{IX}^{\text{top}}$, we denote by $[pt]_{\sf ker}\in H^{\dim\X}(\msf{ker\,X})$ the point class of a connected component of $\msf{ker\,X}$.

\begin{defn}\label{def uniruled-orbifold}
We say that $(\X,\omega)$ is {\bf symplectic uniruled} if there is a nonzero genus zero Gromov--Witten invariant of the form
\begin{align}\label{E def-uniruled}
\langle[pt]_{\sf ker},\alpha_2,\ldots,\alpha_k\rangle^\X_{0,A}\neq 0
\end{align}
with $k\geq 1$ and $0\neq A\in H_2(|\X|;\integer)$.
\end{defn}

\begin{remark}\label{rmk 7.2}
When $\X$ is ineffective, $\msf{ker\,X}$ has several connected components. Then the above definition means that \eqref{E def-uniruled} holds for the point class of one of the components of $\msf{ker\,X}$.
\end{remark}


We also have a characterization of symplectic uniruledness by using descendent invariants
as the manifold case (cf. \cite[Theorem 4.10]{HLR08}).
\begin{theorem}\label{thm condi unirl}
$(\X,\omega)$ is symplectic uniruled if and only if there is
a nonzero, possibly disconnected genus zero descendent orbifold Gromov--Witten invariant of the form
\[
\langle[pt]_{\sf ker},\tau_{d_2}\alpha_2,\ldots,\tau_{d_k}\alpha_k \rangle^\X_{0,A}\neq 0
\]
with $k\geq 1$ such that the component with the $[pt]_{\sf ker}$ insertion has nonzero homology class. Here when $\X$ is ineffective, Remark \ref{rmk 7.2} also applies.
\end{theorem}


\begin{proof}

In the manifold case, the proof of \cite[Theorem 4.10]{HLR08} used only axioms of Gromov--Witten invariants, the boundary relations of $\psi$-classes, and some properties on genus zero invariants with zero curve class. These axioms and boundary relations of $\psi$-classes are also true for orbifold Gromov--Witten invariants. The orbifold Gromov--Witten theory also has similar properties for genus zero invariants with zero curve class. So the proof of \cite[Theorem 4.10]{HLR08}  also proves this theorem.
\end{proof}

\subsection{Invariance of symplectic uniruledness under weighted blowups}

The general correspondence results in \S \ref{sec correspondence} implies that symplectic uniruledness is invariant under weighted blowups. Let $\oS$ be a compact symplectic sub-orbifold groupoid and $\Xa$ be the weighted-$\wa$ blowup of $\X$ along $\oS$.
\begin{theorem}\label{thm inv-sym-unirule}
$\X$ is symplectic uniruled if and only if $\Xa$ is symplectic uniruled.
\end{theorem}

\begin{proof}
Suppose $\X$ is symplectic uniruled, then there is a non-zero class $A\in H_2(|\X|;\integer)$ and a connected nonzero genus 0 orbifold Gromov--Witten invariant of the form
\[
\langle[pt]_{\sf ker},\alpha_2,\ldots,\alpha_k\rangle^\X_{0,A}.
\]
with $A\neq 0$. By blowing up $\X$ along $\oS$ with weight $\wa$ we degenerate $\X$ into $\X^+\wedge_\Z\X^-$
with $\X^-=\Xa$. Then by degeneration formula, we express the invariant $
\langle[pt]_{\sf ker},\alpha_2,\ldots,\alpha_k\rangle^\X_{0,A}
$
as a combination of possibly disconnected relative invariants of
$
(\X^-|\Z)=(\Xa|\Z)$
and $(\X^+|\Z)$. For those $\alpha_i$ supporting away from $\oS$, we can put
$\alpha_i^-=\msf I\kappa^*\alpha_i\in\scr K$. (Note that there would be no such $\alpha_i$.) For $[pt]_{\sf ker}$ we choose $([pt]_{\sf ker})^-$ to be $\msf I\kappa^*[pt]_{\sf ker}$, the point class of the corresponding connected component of $\msf{ker\,\underline{X}}_{\wa}$.
Then there must be a nonzero invariants of the form
\begin{align}\label{eqn non0inv}
\langle{\bf R}^\bullet(\msf I\kappa^*([pt]_{\sf ker}),\alpha^-_{i_1},\ldots, \alpha^-_{i_l}|J)\rangle^{\Xa|\Z},
\end{align}
and the component containing the $\msf I\kappa^*[pt]_{\sf ker}$ insertion has nonzero curve class. This nonzero invariant shows that $v_{\Sigma_\star,0,[pt]_{\sf ker}}\in \rone_{\Sigma_\star,\K,0,[pt]_{\sf ker}}(\Xa|\Z)$ is nonzero.
Now we apply Theorem \ref{thm correspds-pt} to the pair $(\Xa,\oS=\Z)$, i.e.
we blowup $\Xa$ along $\Z$ with weight $\wa=(1)$ to get $\Xa\wedge_\Z \lN_\Z $.
Then we get a nonzero absolute descendent invariant of $\Xa$ with $\msf I\kappa^*[pt]_{\sf ker}$ insertion and nonzero curve class. Since $\msf I\kappa^*[pt]_{\sf ker}$ is the point class of a component of ${\sf ker\,\underline{X}}_\wa$,
$\Xa$ is symplectic uniruled by Theorem \ref{thm condi unirl}.

Conversely, suppose that
\[
\langle [pt]_{\sf ker},\beta_2,\ldots,\beta_k\rangle_{0,A}^{\Xa}\neq 0
\]
with $A\neq 0$, where $[pt]_{\sf ker}$ is the point class of a component of ${\sf ker\,\underline{X}}_\wa$. By the linearity of orbifold Gromov--Witten invariants, we can assume that $\beta_j\in \K=\msf I\kappa^*(H^*_{CR}(\X)), 1\leq j\leq l$,
and for $l+1\leq j\leq k$, $\beta_j$ are of the form $\sigma_j\cup[\Z_{(s)}]$,
with $[\Z_{(s)}]$ the Thom class of $\Z_{(s)}$ in $(\Xa)_{(s)}$.
Then we apply degeneration formula to the degeneration of $\Xa$ given by blowup of $\Xa$ along $\Z$ with trivial weight $\wa=(1)$,
and set $([pt]_{\sf ker})^+=0$, $\beta_j^-=\beta$ for $1\leq j\leq l$, and $\beta^-_j=0$ for $l+1\leq j\leq k$.
There must be a nonzero relative invariant of $(\Xa|\Z)$ of the form
$\langle{\bf R}^\bullet([pt]_{\sf ker}, \varpi|J)\rangle^{\Xa|\Z}$
where $\varpi$ has at most $l$ insertions and all belong to $\msf I\kappa^*(H^*_{CR}(\X))$,
$J$ is a $\Sigma_\star$-relative data,
and the connected component containing the $[pt]_{\sf ker}$ insertion has nonzero curve class. This point class $[pt]_{\sf ker}$ is in fact $\msf I\kappa^*[pt]_{\sf ker}$, the pull back of the point class of the corresponding connected component of $\sf ker\,\X$. So we next apply $L_{0,[pt]_{\sf ker}}$ in Theorem \ref{thm correspds-pt} for the pair $(\X,\oS)$,
we get a nonzero descendent invariant with $[pt]_{\sf ker}$ (the point class of the corresponding connected component of ${\sf ker\,X}$) insertion and nonzero curve class.
So $\X$ is symplectic uniruled by Theorem \ref{thm condi unirl}.
\end{proof}

\subsection{$(t)$-symplectic uniruledness}\label{subsec general-symplec-uniruled}
We may also consider other kinds of symplectic uniruledness by involving point classes of general twisted sectors of $\X$ not only $\msf{ker\,X}$.  For instance, we could consider the following situation.

Let $(t)\in\scr T^\X$, and $\X(t)$ be the corresponding twisted sector of $\X$. We say that $\X$ is {\bf $(t)$-symplectic uniruled} if there is a nonzero (possibly descendent) orbifold Gromov--Witten invariant of the form
\[
\langle [pt]_{(t)},\tau_{d_2}\alpha_2,\ldots,\tau_{d_k}\alpha_k \rangle_{0,A}^\X\neq 0
\]
with $0\neq A\in H_2(|\X|;\integer)$.

Then by Theorem \ref{thm correspds-pt} and Remark \ref{rmk other-restriction-of-L}, the argument that proves Theorem \ref{thm inv-sym-unirule} will show the following result.
\begin{theorem}
Let $(t)\in\scr T^\X$. Suppose that $\X(t)$ is away from $\oS$. So its direct transformation $\Xa(t)$ is away from $\Z$. Then $\X$ is $(t)$-symplectic uniruled if and only if $\Xa$ is $(t)$-symplectic uniruled.
\end{theorem}

Here by $\X(t)$ being away from $\oS$ means that $\X(t)\nsubseteq \msf{IS}$, see Remark \ref{rmk other-restriction-of-L}. This assumption ensure that when we use the degeneration formula, we could choose $[pt]_{(t)}^-$ to be $[pt]_{(t)}$.

\end{document}